\documentclass[11pt,letter]{article}

\usepackage[margin=1in]{geometry}

\usepackage{dsfont, amssymb,amsmath,amscd,latexsym, amsthm, amsxtra,amsfonts}

\usepackage[dvipsnames,svgnames]{xcolor}

\usepackage{lineno}

\usepackage[active]{srcltx}

\usepackage{tikz}
\usepackage{natbib}
\usepackage{bbm}
\usepackage{enumerate}
\usepackage{mathrsfs}

\usepackage{comment}         
\usepackage{cases}
\usepackage{circuitikz}

\usepackage{subcaption} 

\usepackage{verbatim}
\usepackage{graphicx}
\usepackage{epstopdf}
\usepackage{bm}
\usetikzlibrary{automata,arrows,positioning,calc}
\usepackage[pdfstartview=FitH, bookmarksnumbered=true,bookmarksopen=true, colorlinks=true, pdfborder=001, citecolor=blue, linkcolor=blue,urlcolor=blue]{hyperref}
\numberwithin{equation}{section}
\usepackage{graphics}
\usepackage{booktabs}
\usepackage{algorithm,algorithmic}

\usepackage{mathtools}
\mathtoolsset{showonlyrefs=true}

\allowdisplaybreaks

\newtheorem{theorem}{Theorem}[section]
\newtheorem{corollary}[theorem]{Corollary}
\newtheorem{proposition}[theorem]{Proposition}
\newtheorem{lemma}[theorem]{Lemma}

\theoremstyle{definition}
\newtheorem{definition}[theorem]{Definition}

\theoremstyle{definition}
\newtheorem{example}{Example}
\theoremstyle{definition}
\newtheorem{remark}{Remark}
\theoremstyle{definition}
\newtheorem{assumption}{Assumption}

\usepackage{amsmath}
\DeclareMathOperator*{\argmax}{arg\,max}

\newcommand{\md}{\mathrm{d}}

\newcommand{\mR}{\mathbb{R}}

\newcommand{\mN}{\mathbb{N}}

\newcommand{\mE}{\mathbb{E}}

\newcommand{\mP}{\mathbb{P}}

\newcommand{\mT}{\mathbb{T}}

\newcommand{\ind}{\mathbbm{1}}

\renewcommand{\epsilon}{\varepsilon}

\newcommand{\M}{\mathcal{M}}

\newcommand{\F}{\mathcal{F}}
\newcommand{\E}{\mathcal{E}}

\newcommand{\Hi}{\mathcal{H}}

\newcommand{\Ti}{\mathcal{T}}

\newcommand{\B}{\mathcal{B}}

\newcommand{\cX}{\mathcal{X}}
\newcommand{\cY}{\mathcal{Y}}
\newcommand{\cZ}{\mathcal{Z}}

\newcommand{\cW}{\mathcal{W}}
\newcommand{\cL}{\mathcal{L}}

\newcommand{\cJ}{\mathcal{J}}

\newcommand{\cP}{\mathcal{P}}
\newcommand{\Epsilon}{\mathcal{E}}

\newcommand{\barS}{\bar{S}}

\newcommand{\bF}{\mathbf{F}}

\newcommand{\bp}{\mathbf{p}}

\newcommand{\bmu}{\bm{\mu}}

\newcommand{\bphi}{\bm{\phi}}

\newcommand{\sL}{\mathscr{L}}

\newcommand{\sS}{\mathcal{S}}

\newcommand{\weakstar}{\overset{\ast}{\rightharpoonup}}

\begin{document}

\title{Time-inconsistent mean-field stopping problems: A regularized equilibrium approach}

\author{Xiang Yu \thanks{Email:\ xiang.yu@polyu.edu.hk, Department of Applied Mathematics, The Hong Kong Polytechnic University, Kowloon, Hong Kong}
\and  
Fengyi Yuan\thanks{Email:\ fengyiy@umich.edu, Department of Mathematics, University of Michigan, Ann Arbor, MI, US}	}
\date{\vspace{-0.3in}}

\maketitle

\begin{abstract}
    This paper studies mean-field Markov decision processes (MDPs) with centralized stopping under non-exponential discount functions. The problem differs fundamentally from most existing studies on mean-field optimal control/stopping due to its time inconsistency by nature. We look for the subgame perfect relaxed equilibria, namely the randomized stopping policies that satisfy the time-consistent planning with future selves from the perspective of the social planner. On the other hand, unlike many previous studies on time-inconsistent stopping where decreasing impatience plays a key role, we are interested in the general discount function without imposing any conditions. As a result, the study on the relaxed equilibrium becomes necessary as the pure-strategy equilibrium may not exist in general. We formulate relaxed equilibria as fixed points of a complicated operator, whose existence is challenging by a direct method. To overcome the obstacles, we first introduce the auxiliary problem under the entropy regularization on the randomized policy and the discount function and establish the existence of the regularized equilibria as fixed points to an auxiliary operator via the Schauder fixed point theorem. Next, we show that the regularized equilibrium converges as the regularization parameter $\lambda$ tends to $0$ and the limit corresponds to a fixed point to the original operator, and hence is a relaxed equilibrium. We also establish some connections between the mean-field MDP and the N-agent MDP when $N$ is sufficiently large in our time-inconsistent setting. 

\end{abstract}

\section{Introduction}\label{intro}

{   Dynamic decision problems are ubiquitous in wide applications across finance, economics, and operations research. In some large and complex systems, the decision maker may confront a large volume of mutually interacting individual data that influences the total rewards, but he cannot control these individual state processes. Instead, he makes decisions on an aggregate level. For example, ETF options are important tools for practitioners to hedge systematic risk, e.g., the downward risk of an index. The exercise of American ETF options is a stopping problem that fits the previous type of decision making: whether to stop the whole system and take the rewards or to wait for larger rewards in the future. Prices of ETFs are of course determined by prices of all constituent stocks (which may have intricate interactions), but the decision maker (the option holder in this particular circumstance) can not influence or trade individual stocks. For more details of this example, see Subsection \ref{exm:ETF:option}. 

Two important features of these problems can be pinned down as the interactions in the large system and the decision-making on an aggregate level, which motivate us to consider a mean-field control framework and utilize the mean-field solution to approximate the decision making in the original large but finite system. In this paper, by employing the Markov Decision Process (MDP) with entropy regularization, we propose a fixed point method to investigate the aforementioned mean-field problems; see Sections \ref{MF-MDP} and \ref{exisresults}. Our approach does not rely on dynamic programming principle and is naturally applicable to some time-inconsistent decision making problems under non-exponential discount functions. We also establish connections among the mean-field problem, McKean-Vlasov type limit problem, and the original finite-agent problem; see Section \ref{finiteconvergencesection}.

}

\subsection{  Literature Review}

Motivated by the large population of agents with cooperative interactions, mean-field control problems have attracted a lot of attention during the last decade. The goal of this type of control problem is to determine the centralized optimal strategy for the social planner, where the controlled McKean-Vlasov system and the reward depend not only on the state and action but also on the distribution of the state. It has been shown that this centralized optimal control can often serve as the asymptotic approximation for the stochastic system with a large number of agents under the {  social optimality; see, e.g., \cite{Lacker2017}, \cite{Motte2022} and \cite{CDJS2023}}. Successful applications of mean-field control can be found in power management, finance, ride-sharing, and epidemic control, among others.

The mainstream of existing studies on mean-field control problems in both discrete-time and continuous-time frameworks rely crucially on the dynamic programming principle (DPP) or time consistency. It is well-known that on the strength of DPP, the global time optimization problem can be decomposed into a series of local time recursive optimization problems, which enables the characterization of the value function as the viscosity solution to the associated Bellman equation. In fact, unlike the counterpart of single agent's control problems, the controlled McKean-Vlasov system exhibits {   non-linear dependency on population distribution that renders the desired DPP more difficult to verify}. Recently, {  this} time-consistency issue is resolved in \cite{Pham17} in the general continuous time framework with strict control by enlarging the state space to the Wasserstein space of probability measures such that the flow property on the law of the state process holds. Later, \cite{DPT} considers the strong and weak formulation of the mean-field control in both Markovian and non-Markovian settings and establishes the DPP by interpreting controls as probability measures on an appropriate canonical space. In the discrete-time framework, \cite{Motte2022} and \cite{Bauerle2023} also investigate the mean-field MDP with general controls via the studies of the Bellman equation operator under the umbrella of DPP. See \cite{carmona2018probabilistic} for the extensive review of existing studies on this topic. 
 
Recently, \cite{N23} investigates the mean-field optimal stopping of the continuous-time McKean-Vlasov systems without common noise. By considering the weak formulation in terms of the joint marginal law of the stopped underlying process and the survival process, \cite{N23} establish the DPP and derive the dynamic programming equation as an obstacle problem on the Wasserstein space. In contrast, we focus on the discrete-time Markov decision process (MDP) under the mean-field interaction in the presence of common noise. More importantly, unlike the aforementioned studies on mean-field optimal control/stopping, we consider the mean-field stopping problem under the non-exponential discount, rendering the problem time inconsistent by nature. The inherent time inconsistency implies that the optimal strategy we find today will not be optimal at future dates, and hence the conventional methods relying on DPP are not applicable in our setting.

Since the early studies in \cite{Huang2018} and \cite{Huang2019}, the so-called {\it iteration approach} to tackle time-inconsistent stopping problems has been seen great success in various contexts in both discrete and continuous time framework; see, e.g., \cite{HHZ20}, \cite{BZZ21}, \cite{HY2021}, \cite{Huang2022}, \cite{DZY23}, just to name a few. It is worth noting that, this series of studies only focuses on {\it pure stopping policies} as hitting times, namely the stopping policies can be identically described by stopping regions, and the so-called time-consistent {\it mild equilibria} can be obtained through iterations of operators defined on stopping regions. Randomized stopping policies or mixed stopping strategies, on the other hand, have been considered recently in different topics related to stopping decisions, e.g., time-inconsistent stopping in continuous time (\cite{Christensen2020}, \cite{BSK22}), the mean-standard deviation stopping in discrete time (\cite{Bayraktar2019}) and the discrete-time Dynkin games (\cite{Christensen2023}). To the best of our knowledge, our work appears as the first study to incorporate randomized stopping policies in discrete-time mean-field MDP under the non-exponential discount.

The main reason that \cite{Huang2019} and some subsequent studies do not need to consider randomized stopping policies is that the pure-strategy equilibrium policies (equivalently, equilibrium stopping regions) always exist in their contexts under the key assumption of the decreasing impatience on the discount functions\footnote{also known as {\it present bias} in behavioral economics.}, and they can further provide optimality among all equilibria either in general models or in some concrete examples. Some previous theoretical and empirical studies focus on specific discount functions, such as the hyperbolic and quasi-hyperbolic discount, which inherently exhibit decreasing impatience; see, e.g., \cite{Loewenstein1992}, \cite{Laibson1997} and \cite{Rabin2015}. However, there is also more emerging research acknowledging the inadequacy of the decreasing impatience in describing the complex decision behavior of individuals or on the aggregation level. For example, \cite{AS12} and \cite{ABGHW16} observe the {\it increasing impatience} from the aggregate level via their experimental studies. On the other hand, \cite{ABRW10} and \cite{BGR16} identify individual heterogeneity, i.e., they find the co-existence of decreasing impatience and increasing impatience among different subjects. For more discussions on limitations of decreasing impatience, we refer readers to \cite{Rohde2018} and references therein. There are also theoretical papers studying general forms of discount functions that can accommodate both decreasing and increasing impatience; see \cite{BRW09}, \cite{Greco2023} and references therein. The aforementioned studies suggest removing the restriction of decreasing impatience to better match the observed diversity in real-life decision-making. In response, we consider the time-inconsistent mean-field stopping problem under the most general discount function without imposing additional assumptions, where we can still establish the existence of the relaxed equilibria as randomized stopping policies or mixed stopping strategies. Our result can thus be regarded as a version of {\it Nash Theorem} because we can show that the mixed-strategy equilibrium generically exists without any structural assumptions in an intra-personal game framework that achieves time-consistent planning. We also note that for some time inconsistent MDP for a single agent's control problem, some results on the existence of equilibrium control can be found, however, subjecting to other structural assumptions such as the finite state space (see \cite{HZ2021} and \cite{Bayraktar2023}) or the finite time horizon (see \cite{BH23}).

\subsection{  Our Contributions}
Contrary to existing studies on mean-field control/stopping that predominantly rely on DPP and the resulting dynamic programming equation, we are the first to consider the time-inconsistent mean-field MDP with stopping under the non-exponential discount. Our methodology differs fundamentally from these papers. Due to our general discount function, we instead employ the time-consistent planning proposed in \cite{S55}. The goal is to find a subgame perfect equilibrium as a randomized stopping policy (see Definition \ref{originalequilibrium}) for the social planner in the model with mean-field interaction such that no future selves of the social planner have the incentive to deviate. Moreover, we work with the general mean-field model to incorporate the common noise that affects the entire system. Under the game-theoretic thinking, the key mathematical challenge is to show the existence of the fixed point to the complicated operator \eqref{purefixedpoint} or \eqref{mixedequilibriumchar}. 


Contrary to some time-inconsistent stopping problems under the critical assumption of decreasing impatience, we embrace more general discount functions in the present paper that allow both decreasing and increasing impatience documented in the literature. To ensure the existence of the equilibrium in this general setting, we have to work with the randomized stopping policies or the so-called mixed-strategy stopping policies satisfying the time-consistent planning with future selves, which will be called relaxed equilibria in our setting (see Definition \ref{originalequilibrium}). However, establishing the existence of such relaxed equilibria is a challenging task in general, which corresponds to solving a technical fixed point problem. Inspired by \cite{Bayraktar2023} and some recent studies on exploratory RL problems under entropy regularization, we first work with the regularized equilibrium (see Definition \ref{eqdef}) in the presence of entropy. In particular, we utilize the specific entropy regularization on both the randomized policy and the discount function (see \eqref{valueregular}, \eqref{axvalueregular} and \eqref{discountreg}), and verify the equivalence between the regularized equilibrium in the form of a Gibbs measure and the fixed point of the auxiliary operator (see \eqref{equifix} and Proposition \ref{eqchar}) that possesses the more explicit structure. By some novel technical arguments, we are able to show that the deduced operator enjoys certain properties such that the existence of regularized equilibria can be obtained with the help of Schauder's fixed point theorem (see Theorem \ref{fixedpoint} and Corollary \ref{extregulareq}).   

{   In many recent studies on reinforcement learning (RL) for mean-field control, both DPP and entropy regularization play essential roles in designing some successful model-free iterative learning algorithms such as the Q-learning; see, e.g., \cite{Car19}, \cite{GGWX23} and \cite{weiyu2023}. Through an example of option exercise (see Subsection \ref{exm:ETF:option}), we show that even in the time-inconsistent setting, the entropy regularization still helps the design of model-free RL algorithms. In fact, it is at the core of our algorithm to learn the randomized policy satisfying the time-consistent constraint. Without the entropy regularization, we still need to consider the randomized policy as the relaxed equilibrium due to the general discount function. However, the expressions of the unregularized relaxed equilibrium policies in indifferent regions become obscure, which cannot be implemented in the policy iterations. For example, in \eqref{mixedequilibriumchar}, we have no information of $\phi^*$ in the region such that $f_{\phi^*}=r$. With the help of entropy regularization, however, we can utilize the explicit fixed point characterization of the regularized equilibrium in \eqref{equifix} to devise the policy iteration rule and further employ the neural TD algorithm for the policy evaluation.
}

To further pass from the regularized equilibrium to the relaxed equilibrium without entropy, some key steps are to show the convergence of the regularized equilibrium as $\lambda\rightarrow 0$ and verify that the limit indeed satisfies the fixed point to the operator in \eqref{mixedequilibriumchar} (see numerical illustrations in Figures \ref{solutiondemovalue} and \ref{solutiondemostrategy}).
We encounter several new challenges caused by the uncountable state space of probability measures such that some estimations obtained for the regularized equilibrium may explode as $\lambda$ vanishes. In response, we propose some novel and tailor-made arguments as follows: Firstly, we start with a direct weak-$*$ convergence and upgrade it to uniform convergence by considering the one-step transition. Secondly, by taking advantage of the special structure of the stopping problem, we show that the regularized value function converges to the unregularized one (see \eqref{uconvergencev}) using the equivalent formulation of the randomized stopping time (see Lemma \ref{formulationclassical}), and the fixed point condition can be established by properly choosing some test functions and leveraging on the weak-$*$ convergence and the uniform convergence after one-step transition. Moreover, the example in Section \ref{Illexample} shows the necessity to consider the relaxed equilibria as mixed-strategy stopping policies because the pure-strategy equilibrium in \cite{Huang2019} may not exist under the general discount in the lack of the decreasing impatience.

In Section \ref{finiteconvergencesection}, we also discuss important connections between the N-agent MDP with the centralized stopping and the mean-field MDP with stopping (see Figure \ref{MDPrelations}). It is shown that the regularized equilibrium not only will converge to the relaxed equilibrium in Definition \ref{originalequilibrium}, but also serves as a good approximated equilibrium in the N-agent MDP when $N$ is sufficiently large (see Definition \ref{Nepsilondef}, Remark \ref{rmk:finiteconvergence} and Theorem \ref{epsiloneq-N}). Admittedly, there are extensive studies on the approximation of the solution to the multi-agent control (or multi-agent game) by the mean-field solution in the mean-field control (or mean-field game), we stress that most conclusions still rely on the time-consistency or DPP. By contrast, our work appears as the first result in the time-inconsistent setting that establishes the approximation of the time-consistent equilibrium for the multi-agent centralized stopping by the regularized equilibrium in the mean-field model under entropy regularization. Our approximation result not only exemplifies the theoretical advantage of the regularized equilibrium approach but also can significantly reduce the dimension in the practical implementations of the equilibrium policy in the model with finite agents.

The remainder of this paper is organized as follows. Section \ref{MF-MDP} first introduces the problem formulation of the time-inconsistent mean-field MDP with stopping under non-exponential discount together with the definitions of the relaxed equilibrium and the regularized equilibrium under the entropy regularization. Section \ref{exisresults} presents the main results of this paper, namely the existence of the regularized equilibria and the convergence of the regularized equilibrium towards a relaxed equilibrium as the regularization parameter $\lambda\rightarrow 0$. {  Section \ref{finiteconvergencesection} further discusses the connections between our mean-field MDP with stopping and the N-agent MDP problem and establishes an approximated equilibrium to the N-agent MDP using the regularized equilibrium in the mean-field model. In Section \ref{sec:exm}, we provide two examples to illustrate financial applications of our theoretical results. Section \ref{sec:extension} discusses the extension to the asynchronous stopping of a proportion of the population. Finally, proofs of all main results are collected in Section \ref{proofs}. }

\section{Problem Formulation}\label{MF-MDP}

\subsection{Notations}
Let $\mT=\{0,1,2,\cdots\}$ be a discrete-time set of decision epochs, and let $S$ denote the state space, which is assumed to be a compact metric space equipped with the Borel sigma field $\sS$. To accommodate the feature of mean-field stopping for the large population, we introduce $\cP(S)$ as the space of probability measures on $S$, and denote the enlarged state space $\barS:= \cP(S)\bigcup \{\triangle\}$. Here, each $\mu\in \barS$ stands for the distribution state of the population, and $\triangle\in \barS$ is used to indicate the stopped state. We denote by $\cW_1$ the standard 1-Wasserstein metric on $\cP(S)$. Note that $(\cP(S),\cW_1)$ is a compact metric space, we can therefore choose a large enough constant $C$ depending only on $S$ such that $\cW_1$ can be extended to $\barS$ by defining $\cW_1(\mu,\triangle)=\cW_1(\triangle,\mu)=C$ for any $\mu\in \cP(S)$. We will also consider the 1-Wasserstein metric on $\cP(\barS)$ and denote it by $\overline{\cW}_1$.

To describe the decision-making of stopping time, we consider the simple action space $A=\{0,1\}$ (homogeneous for all states $x\in S$), where the action $a=1$ means that the {\it social planner} stops the whole system while $a=0$ indicates that the social planner chooses to continue. Due to the mean-{  field} interaction, it is natural for us to consider the randomized policy of stopping time for the social planner. In particular, let us introduce a sequence of measurable mappings $\bphi=\{\phi_k\}_{k\in \mT}$, where $\phi_k:\barS \to \cP(A)=[0,1]$, in which $\phi_k(\mu)\in [0,1]$ depicts the probability of stopping for the social planner when the population state is at $\mu\in\barS$. Let $\F$ denote the set of all policies. We also consider the {\it stationary} policies $\F_S:=\{\bphi\in \F:\phi_k=\phi_0,\forall k\in \mT \}$.

The common noise has been well documented in wide applications of multi-agent and mean-field systems, which may have significant impact on the group decision making. To incorporate the common noise into our setting, let us consider a Polish space $\cZ$, and any $z\in \cZ$ can be seen as the state of the public information that affects all individual dynamics. Let us denote $Z^0$ as a random variable valued in $\cZ$ and let $\cL(\cZ)$ be its distribution. Moreover, for $q\in [0,1]$, let $\xi^q$ be a Bernoulli random variable such that $\xi^q=1$ with probability $q$. It is assumed throughout the paper that these two random variables are independent, and at each step of the transition, we choose independent copies of them.

For any two probability spaces $(\cX,\F_\cX,P)$, $(\cY,\F_\cY)$ and a measurable mapping $F:(\cX,\F_\cX)\to (\cY,\F_\cY)$, we denote by $F_\# P$ the distribution of $F$ under $P$. Specifically, $F_\# P$ is a probability measure on $(\cY,\F_\cY)$, defined by $F_\# P(B)=P(x\in \cX:F(x)\in B)$, $\forall B\in \F_\cY$. We also use the notation $P_F$ to denote the distribution of $F$ under $P$.

\subsection{Mean-field Markov Decision Process with Stopping}
We consider the discrete-time mean-field Markov decision process where the social planner makes the centralized decision of stopping on behalf of the population, see also \cite{Motte2022} and \cite{Bauerle2023} for the more general mean-field MDP in the context of control problems. For a given policy $\bphi$ and the population distribution $\mu_k$ at time step $k$, the transition rule from $k$ to $k+1$ is given by
\begin{equation}\label{transition}
	\mu_{k+1}=T(\mu_k,\xi^{\phi_k(\mu_k)},Z^0).
\end{equation}
Here, $T:\barS\times A\times \cZ\to \barS$ is the transition operator of population distributions, $\xi^{\phi_k(\mu_k)}$ is {  a Bernoulli random variable with the success probability $\phi_k(\mu_k)$, representing the randomness from randomized (mixed) strategies, and $Z^0$ stands for the common noise.} 

We recall that $\triangle$ is added to the state space $\barS$ to indicate the stopped state. We will consider $T_0(\cdot,z):\cP(S)\to \cP(S)$, and $T_0(\triangle,z)\equiv \triangle$. 
Because we only focus on stopping time problems, $T$ can be rewritten as 
\begin{equation}\label{T0def}
    T(\mu,a,z)=\left\{
    \begin{aligned}
        &T_0(\mu,z), & a=0,\\
        &\triangle, & a=1.
    \end{aligned}
    \right.
\end{equation}
It is worth noting that, $T_0$ can be induced from the transition probability of a MDP with the state space $S$. Indeed, if the original MDP has the transition kernel $Q^{z}(x,\md x')$, we can define $T_0$ by
\[
T_0(\mu,z)(\md x')=\int_S Q^{z}(x,\md x')\mu(\md x).
\]
In Subsection \ref{relationlimit}, we also show that $T_0$ can be induced from a McKean-Vlasov type MDP, in which $T_0$ is not linear in $\mu$. 

Next, we consider a reward function $r:\barS\to \mR$. Let us set $r(\triangle)=0$ as $\triangle$ is the stopped state that produces no reward. We also consider a general discount function $\delta:\mT\to [0,1]$ satisfying $\delta(0)=1$, and $\delta(k+1)\leq \delta(k)$ for any $k\in \mT$. Other than that, we do not impose additional assumptions on the discount function such as the decreasing impatience in \cite{Huang2019}.

Finally, let us consider the value function under the prescribed policy $\bphi\in \F$, which is defined by 
\begin{equation}\label{Jdef}
J^{\bphi}(\mu):=\sum_{k=0}^{\infty}\delta(k)\mE^{\mu,\bphi}r(\mu_k)\phi_k(\mu_k),
\end{equation}
where $\mE^{\mu,\bphi}:= \mE^{\mP^{\mu,\bphi}}$ is the expectation operator under $\mP^{\mu,\bphi}$, which is a probability measure on $\Omega:= \barS^{\infty}$ induced by the transition rule \eqref{transition} such that $\mP^{\mu,\bphi}(\mu_0=\mu)=1$. Here, $\bmu=\{\mu_k\}_{k\in \mT}$ is the canonical process. The existence of the unique probability measure $\mP^{\mu,\bphi}$ is guaranteed by Ionescu-Tulcea Theorem.

To better address the path-dependency caused by the randomized stopping time, we embed the randomness from the randomized policy into the definition of $\mP^{\mu,\phi}$, thus our definition of the value function seems different from the classical randomized stopping time in the discrete-time framework in the existing literature; see, e.g., \cite{Solan2012}, \cite{Bayraktar2019}, \cite{Be19}, \cite{Christensen2023}; see also the continuous-time counterpart in \cite{Christensen2020}, \cite{Dong2022}, and et al. However, the next result shows that our formulation is equivalent to theirs to some extent.
\begin{lemma}\label{formulationclassical}
   Let $\tilde{\mP}^{\mu}$ be the probability measure induced by the transition rule $\mu_{k+1}=T_0(\mu_k,Z^0)$ and the initial condition $\mu_0=\mu$, and let $\tilde{\mE}^{\mu}$ denote its expectation. Then for any $\bphi\in \F$, $\mu\in \barS$ and $k\in \mT$, it holds that
    \[
    \mE^{\mu,\bphi}r(\mu_k)\phi_k(\mu_k)=\tilde{\mE}^\mu r(\mu_k)\phi_k(\mu_k)\prod_{j=0}^{k-1}(1-\phi_j(\mu_j)).
    \]
    From this point onwards, it is assumed by convention that $\prod_{k=0}^{-1}\equiv 1$.
\end{lemma}
\begin{proof}
    Our argument is based on mathematical induction. First, note that $k=0$ is obvious by the initial condition. Let us assume that the conclusion holds for some $k\geq 1$. Then for any $\mu\in \barS$ and $\bphi\in \F$, we consider the strategy $\tilde{\bphi}$ given by $\tilde{\phi}_k=\phi_{k+1}$. By applying the induction hypothesis to $\mu_1$ and $\tilde{\bphi}$, we get
    \[
    \mE^{\mu_1,\tilde{\bphi}}r(\mu_k)\tilde{\phi}_k(\mu_k)=\tilde{\mE}^{\mu_1}r(\mu_k)\tilde{\phi}_k(\mu_k)\prod_{j=0}^{k-1}(1-\tilde{\phi}_j(\mu_j)).
    \]
    Taking $\mE^{\mu,\bphi}$ on both sides and using the transition rule \eqref{transition}, we deduce that the left hand side is equal to $\mE^{\mu,\bphi}r(\mu_{k+1})\tilde{\phi}_k(\mu_{k+1})=\mE^{\mu,\bphi}r(\mu_{k+1})\phi_{k+1}(\mu_{k+1})$. The right-hand side, on the other hand, can be computed by
    \begin{align*}
        \mE^{\mu,\bphi}\tilde{\mE}^{\mu_1}r(\mu_k)\tilde{\phi}_k(\mu_k)\prod_{j=0}^{k-1}(1-\tilde{\phi}_j(\mu_j))&=\mE^0\tilde{\mE}^{T(\mu,\xi^{\phi_0(\mu)},Z^0)}r(\mu_k)\tilde{\phi}_k(\mu_k)\prod_{j=0}^{k-1}(1-\tilde{\phi}_j(\mu_j))\\
        &=(1-\phi_0(\mu))\mE^0\tilde{\mE}^{T_0(\mu,Z^0)}r(\mu_k)\tilde{\phi}_k(\mu_k)\prod_{j=0}^{k-1}(1-\tilde{\phi}_j(\mu_j))\\
        &=(1-\phi_0(\mu))\tilde{\mE}^{\mu}\tilde{\mE}^{\mu_1}r(\mu_k)\tilde{\phi}_k(\mu_k)\prod_{j=0}^{k-1}(1-\tilde{\phi}_j(\mu_j))\\
        &=(1-\phi_0(\mu))\tilde{\mE}^\mu r(\mu_{k+1})\tilde{\phi}_k(\mu_{k+1})\prod_{j=0}^{k-1}(1-\tilde{\phi}_j(\mu_{j+1}))\\
        &=\tilde{\mE}^\mu r(\mu_{k+1})\phi_{k+1}(\mu_{k+1})\prod_{j=0}^{k}(1-\phi_j(\mu_{j})).
    \end{align*}
    Hence, the conclusion holds for $k+1$, which completes the proof.
\end{proof}
\begin{remark}\label{stoppingremark}
	It is noted that $\phi\in \F_S$ can indeed be interpreted as a {\it randomized} stopping policy. In fact, by Lemma \ref{formulationclassical}, \eqref{Jdef} can be expressed as (we write $\phi_k:=\phi_k(\mu_k)$):
	\begin{align*}
	J^{\bphi}(\mu)&=\sum_{k=0}^{\infty}\delta(k)\tilde{\mE}^{\mu}r(\mu_k)\phi_k\prod_{k'=0}^{k-1}(1-\phi_{k'})\\
	&=\tilde{\mE}^{\mu}\mE^{\tau}\delta(\tau)r(\mu_\tau),
	\end{align*}
where $\mP^{\tau}$ is a (conditional) probability measure on $\mT$ such that $\mP^{\tau}(\tau=k)=\phi_k\prod_{k'=0}^{k-1}(1-\phi_{k'})$. This is in line with some classical definitions of the randomized stopping time in the literature.
\end{remark}

\subsection{Relaxed Equilibria and Regularized Equilibria}
We first define a relaxed equilibrium policy, which is at the core of our study. For technical convenience, we only consider stationary equilibria. 
\begin{definition}\label{originalequilibrium}
	$\phi^*\in \F_S$ is said to be a \textit{relaxed equilibrium} if,
	\begin{equation}\label{eqcond}
		J^{\psi\oplus_1\phi^*}(\mu)\leq J^{\phi^*}(\mu),\forall \mu\in \barS, \psi\in [0,1].
	\end{equation}
Here, $\tilde{\bphi}:=\psi\oplus_1\phi^*\in \F$ is given by $\tilde{\bphi}_0(\mu)=\psi$ and $\tilde{\bphi}_k(\mu)=\phi^*(\mu)$, $k\geq 1$, $\mu\in \barS$.
\end{definition}

{  
\begin{remark}
We adopt the notion of ``relaxed" equilibrium to indicate the nature of the randomized stopping, analogue to the ``relaxed policy" that has been commonly used in the context of reinforcement learning. In contrast to the pure strategy equilibrium in some existing studies of time-inconsistent stopping problems (e.g. \cite{Huang2019}), the weaker definition of the relaxed equilibrium usually guarantees its existence when the conventional pure strategy equilibrium may not exist, see Remark \ref{rmk:nonexit-pure}. We also note that this ``relaxed equilibrium" is the same as and exchangeable to the ``mixed strategy equilibrium" in the language of game theory. 
\end{remark}
}

To study the existence of relaxed equilibria, let us consider the {\it auxiliary value function} (see also \cite{Bayraktar2023} for MDP with the relaxed control) defined by 
\begin{equation}\label{auxdef}
\tilde{J}^{\phi^*}(\mu)=\sum_{k=0}^{\infty}\delta(1+k)\mE^{\mu,\phi^*} r(\mu_k)\phi^*(\mu_k).
\end{equation}
Apparently, we have
\begin{align*}
J^{\psi\oplus_1\phi^*}(\mu)=&r(\mu)\psi+\mE^0\tilde{J}^{\phi^*}(T(\mu,\xi^{\psi},Z^0))\\
=&r(\mu)\psi+(1-\psi)\mE^0\tilde{J}^{\phi^*}(T_0(\mu,Z^0)).
\end{align*}
Therefore, \eqref{eqcond} is equivalent to solving the fixed point problem:
\begin{equation}\label{purefixedpoint}
	\phi^*(\mu)\in \argmax_{\psi\in [0,1]}\left\{   r(\mu)\psi+(1-\psi)\mE^0\tilde{J}^{\phi^*}(T_0(\mu,Z^0))\right\},
\end{equation}
which is also equivalent to find the fixed point of the problem: 
\begin{equation}\label{mixedequilibriumchar}
    \phi^*(\mu)=\left\{
    \begin{aligned}
    &1, &r(\mu)>f_{\phi^*}(\mu),\\
    &0, &r(\mu)<f_{\phi^*}(\mu),
    \end{aligned}
    \right.
\end{equation}
where we denote $f_{\phi^*}(\mu):=\mE^0\tilde{J}^{\phi^*}(T_0(\mu,Z^0))$. 

To tackle this challenging fixed point problem, we resort to the method of regularization under entropy, which will provide a more explicit structure of the auxiliary fixed point problem with the regularization parameter $\lambda>0$. For $\phi\in [0,1]$, let us consider the Shannon entropy $\E(\phi):= -\phi\log \phi-(1-\phi)\log (1-\phi)$. In our framework, let us consider the following regularized expected payoff that
\begin{align}
&J_\lambda^{\bphi}(\mu):= \sum_{k=0}^{\infty}\delta_\lambda(k)\mE^{\mu,\bphi}\left[r(\mu_k)\phi_k(\mu_k)+\lambda\E(\phi_k(\mu_k))\right],\label{valueregular}\\
&\tilde{J}_\lambda^{\bphi}(\mu):= \sum_{k=0}^\infty\delta_\lambda(k+1)\mE^{\mu,\bphi}\left[r(\mu_k)\phi_k(\mu_k)+\lambda\E(\phi_k(\mu_k))\right].\label{axvalueregular}
\end{align}
To facilitate some technical arguments in the present paper, especially in the absence of the decreasing impatience condition, we propose to work with the special regularized discount  
\begin{align}\label{discountreg}
\delta_\lambda(k):= \delta(k)\left(\frac{1}{1+\lambda} \right)^{k^2}.
\end{align} 
That is, in addition to the usual entropy regularization, we enforce another level of regularization on the discount such that the discount function decays sufficiently fast. We finally let $\lambda \to 0$ and do not impose any conditions on $\delta$. In particular, we need neither the decreasing impatience property, which is required for the existence of pure strategy equilibria (c.f. \cite{Huang2019}), nor the summable assumption in the control problem (c.f. \cite{Bayraktar2023}).
\begin{definition}\label{eqdef}
$\phi^*\in \F_{S}$ is said to be a {\it regularized equilibrium} with the regularization parameter $\lambda>0$ if
\begin{equation}\label{eqcondmixregular}
	J_\lambda^{\psi\oplus_1\phi^*}(\mu)\leq J_\lambda^{\phi^*}(\mu),\forall \mu\in \barS,\psi\in [0,1].
\end{equation}
\end{definition}

\subsection{Standing Assumptions}
To establish the existence of regularized equilibria and relaxed equilibria, we need to mandate the following assumptions throughout the paper:
\begin{assumption}\label{as1}
 The reward function $r$ is Lipschitz continuous on $\barS$, i.e., there exists a constant $L_2>0$ such that
    \[
        |r(\mu)-r(\mu')|\leq L_2\cW_1(\mu,\mu'),\forall \mu,\mu'\in \barS.
    \]
\end{assumption}
\begin{remark}
    As it is assumed that $S$ is compact, Assumption \ref{as1} implies that $r$ is uniformly bounded, i.e., there exists a constant $L_1>0$ such that $\sup_{\mu\in \barS}|r(\mu)|\leq L_1$.
\end{remark}
\begin{assumption}\label{Tassumption}
    There exists a constant $L_3>0$ such that for $\mu,\mu'\in \barS$, we have
    \begin{equation}\label{TLipbarS}
\bar{d}_{\rm TV}(T_0(\mu,\cdot)_\# \cL(Z), T_0(\mu',\cdot)_\# \cL(Z))\leq L_3\cW_1(\mu,\mu'),
    \end{equation}
    where $\bar{d}_{\rm TV}$ is the total variation distance on $\cP(\barS)$.
\end{assumption}
\begin{assumption}\label{Tassumption2}
    There exists a constant $L_4>0$ such that for any $\mu,\mu'\in \barS$ and $z\in \cZ$, we have
    \begin{equation}\label{TlipS}
     \cW_1(T_0(\mu,z),T_0(\mu',z))\leq L_4\cW_1(\mu,\mu').   
    \end{equation}
\end{assumption}

Assumption \ref{Tassumption2} is the conventional Lipschitz condition that has been widely used in the control theory, and it is naturally induced by similar conditions in the corresponding problem on $S$; see Proposition \ref{assumptionrelation}. As for Assumption \ref{Tassumption}, intuitively, \eqref{TLipbarS} asserts that given two distinctive initial distributions that are close to each other, the resulting one-step transited distribution $\mu$ and $\mu'$ (both are random due to the common noise) should also be close in the sense of total variation. We comment that this is impossible without common noise because $\bar{d}_{\rm TV}(\delta_{\mu_1},\delta_{\mu_2})\equiv 1$ so long as $\mu_1\neq \mu_2$. However, Assumption \ref{Tassumption} is quite natural when $S$ is finite and the common noise is continuous, as shown in the next example. 
In Subsection \ref{relationlimit}, we will also show how this example can be constructed from the corresponding problem on $S$; see Example \ref{exm1construct}.
\begin{example}\label{exm1}
    We take $S=\{1,2\}$, so that $\cP(S)=[0,1]$. Let $\cZ=[0,1]$, and $Z\sim U$, the uniform distribution on $[0,1]$. We define the transition rule as follows: $T_0(\mu,z)=(1-\epsilon)T_1(\mu)+\epsilon z$, where $\epsilon>0$. Here, $T_1$ is a deterministic transition rule, which can be induced from the representative agent problem on $S$; see Example \ref{exm1construct}. We suppose that $T_1$ is Lipschitz. Now, for $x_0:= T_1(\mu)$ and $x_0':= T_1(\mu')$, we can easily conclude that $T_0(\mu,Z)$ has density function $f(\nu)=\frac{1}{\epsilon}\ind_{[(1-\epsilon]x_0,(1-\epsilon)x_0+\epsilon]}(\nu)$ and $T_0(\mu',Z)$ has density function $f'(\nu)=\frac{1}{\epsilon}\ind_{[(1-\epsilon)x_0',(1-\epsilon)x_0'+\epsilon]}(\nu)$. Without loss of generality, we assume $x_0\leq x_0'\leq x_0+\frac{\epsilon}{1-\epsilon}$, so that 
    \begin{align*}
    \bar{d}_{\rm TV}(T_0(\mu,\cdot)_\# U,T_0(\mu',\cdot)_\# U)&=\frac{1}{2}\int_0^1 |f(\nu)-f'(\nu)|\md \nu  \\
    &=\frac{1-\epsilon}{\epsilon}|x_0-x_0'|\\
    &\leq C\cW_1(\mu,\mu').
    \end{align*}
    Similar arguments can be applied to a more general setting with $|S|=d>2$.
\end{example}
 
To prove the existence of relaxed equilibria in Definition \ref{originalequilibrium}, we also need the following enhancement of Assumption \ref{Tassumption}. For simplicity, let us define $T^\mu:= T_0(\mu,\cdot)_\#\cL(Z) $ for any $\mu\in \cP(S)$, where $\cL(Z)$ is the law of the common noise.
 \begin{assumption}\label{dominate}
     There exists a common reference probability measure on $\cP(S)$, denoted by $T^0$, such that for any $\mu\in \cP(S)$, $T^\mu$ is absolutely continuous with respect to $T^0$. Moreover, if we denote by $f^0(\mu,\cdot)$ the density function of $T^\mu$, we also assume that for any $\mu,\mu'\in \cP(S)$,
     \begin{equation}\label{densityLip}
         \int_{\barS} |f^0(\mu,\nu)-f^0(\mu',\nu)|T^0(\md \nu)\leq L_4\cW_1(\mu,\mu').
     \end{equation}
 \end{assumption}
Clearly, \eqref{densityLip} implies \eqref{TLipbarS}. We stress that Example \ref{exm1} (with any finite $S$) actually satisfies Assumption \ref{dominate} by taking $T^0$ to be the uniform distribution on the probability simplex $\{(p_1,p_2,\cdots,p_d)\in \mR^d| p_i\geq 0, \sum_{i=1}^d p_i=1\}$. In Subsection \ref{relationlimit}, we will provide a concrete example with general $S$ (possibly infinite); see Example \ref{dominateexm}. In this regard, Assumption \ref{dominate} is in fact not restrictive. We also remark that, in the general theory of MDP, a similar assumption on the state space (which is often assumed to be merely a metric space) has been proposed; See, for example, \cite{jaskiewicz2021markov}.

\section{Main Results}\label{exisresults}
We first investigate the existence of the regularized equilibrium with a fixed regularization parameter $\lambda>0$, for which we need some additional notations as preparations. Let $C(\barS)$ be the set of continuous functions $V:\barS\to \mR$. Because $S$ is compact, $\barS$ is a compact metric space under $\cW_1$. Thus, $C(\barS)$ equipped with supreme norm is Banach. Next, let us consider an operator $\Ti^\lambda_1:C(\barS)\to \F_{S}$ given by,
\begin{equation}\label{T1def}
\Ti^\lambda_1(v)(\mu):= \frac{\exp\left( \frac{1}{\lambda}r(\mu)   \right)}{\exp\left( \frac{1}{\lambda}r(\mu)   \right)+\exp\left( \frac{1}{\lambda}\mE^0v(T_0(\mu,Z^0))   \right)}.
\end{equation}
Here, we recall that $\mE^0$ is the expectation with respect to the common noise $Z^0$.
Further, let us define another operator $\Ti^\lambda_2:\F_{S}\to C(\barS)$ by
\[
\Ti^\lambda_2(\phi)(\mu):= \tilde{J}_\lambda^\phi(\mu).
\]
We have the following characterization of the regularized equilibrium.
\begin{proposition}\label{eqchar}
Suppose Assumption \ref{as1} is in force. Then $\phi^*\in \F_{S}$ is a regularized equilibrium with regularization parameter $\lambda>0$ (per Definition \ref{eqdef}) if and only if 
\begin{align}\label{equifix}
\phi^*=\Ti^\lambda_1\circ\Ti^\lambda_2(\phi^*).
\end{align}
\end{proposition}
\begin{proof}
    For any $\psi\in [0,1]$, by definition of $\mP^{\mu,\phi}$, we have that
    \begin{align*}
        J_\lambda^{\psi\oplus_1\phi^*}(\mu)=&r(\mu)\psi-\lambda\psi\log \psi-\lambda(1-\psi)\log(1-\psi)\\
        &+\sum_{k=1}^\infty \delta(k) \mE^{\mu,\psi}\mE^{\mu_1,\phi^*}[r(\mu_k)\phi(\mu_k)-\lambda\E(\phi(\mu_k))]\\
        =&r(\mu)\psi-\lambda\psi\log \psi-\lambda(1-\psi)\log(1-\psi)\\
        &+\mE^0\tilde{J}_\lambda^{\phi^*}(T(\mu,\xi^{\psi},Z^0))\\
        =&r(\mu)\psi-\lambda\psi\log \psi-\lambda(1-\psi)\log(1-\psi)\\
        &+(1-\psi)\mE^0\tilde{J}_\lambda^{\phi^*}(T_0(\mu,Z^0)).
    \end{align*}
   The equilibrium condition \eqref{eqcondmixregular} is thus equivalent to 
    \[
\phi^*(\mu)\in \argmax_{\psi\in [0,1]}\{r(\mu)\psi+(1-\psi)\mE^0\Ti^\lambda_2(\phi^*)(T_0(\mu,Z^0)) -\lambda \psi\log \psi-\lambda(1-\psi)\log(1-\psi)   \}.
    \]
 Thanks to the convexity of $x\mapsto x\log x$, the maximum is attained at
    \[
\phi^*(\mu)=\frac{1}{1+\exp\left(\frac{1}{\lambda}[\mE^0\Ti^\lambda_2(\phi^*)(T_0(\mu,Z^0))-r(\mu)]    \right)},
    \]
    which is equivalent to $\phi^*=\Ti^\lambda_1\circ\Ti^\lambda_2(\phi^*)$.
\end{proof}
We are now ready to state the main results in this section. Thanks to the regularization, we are able to obtain fixed points that are Lipschitz continuous, i.e., there exists some $C>0$, such that $|\phi^*(\mu)-\phi^*(\mu')|\leq \cW_1(\mu,\mu')$. We denote by $\F_{S}^{\rm Lip}\subset \F_S$ the set of all Lipschitz strategies.
\begin{theorem}\label{fixedpoint}
    If Assumptions \ref{as1}, \ref{Tassumption} and \ref{Tassumption2} are in force, there exists a policy $\phi_\lambda\in \F^{\rm Lip}_{S}$ such that $\phi_\lambda=\Ti^\lambda_1\circ\Ti^\lambda_2(\phi_\lambda)$.
\end{theorem}
The proof of Theorem \ref{fixedpoint} will be postponed to Section \ref{proofs}. In view of Proposition \ref{eqchar} and Theorem \ref{fixedpoint}, the existence of a regularized equilibrium is a direct consequence. 
\begin{corollary}\label{extregulareq}
Suppose Assumptions \ref{as1}, \ref{Tassumption} and \ref{Tassumption2} are in force. Then there exist a regularized equilibrium $\phi_\lambda\in \F^{\rm Lip}_{S}$ for any regularization parameter $\lambda>0$ (see Definition \ref{eqdef}).
\end{corollary}
One advantage of the regularized equilibrium is that it naturally gives an approximated equilibrium of the original problem. In Section \ref{finiteconvergencesection}, it is shown that the regularized equilibrium indeed provides an approximated equilibrium to the $N$-agent problem.
\begin{theorem}\label{epsiloneq-MF}
  Suppose Assumptions \ref{as1}, \ref{Tassumption} and \ref{Tassumption2} hold. For any $\epsilon>0$, $\phi_\lambda$ is an $\epsilon$-equilibrium of the original problem (per Definition \ref{originalequilibrium}), i.e., for every $\mu\in \barS$,
	\begin{equation}\label{epsilon-eqcond}
		J^{\psi\oplus_1\phi_\lambda}(\mu)\leq J^{\phi_\lambda}(\mu)+\epsilon,\forall\psi\in [0,1].
	\end{equation}
  provided that $\lambda$ is sufficiently small.
\end{theorem}
\begin{proof}
    For simplicity, we fix $\mu\in \barS$ and $\psi\in [0,1]$ and denote $\phi'=\psi\oplus_1\phi_\lambda$ . By Corollary \ref{extregulareq} and Definition \ref{eqdef}, we have
    \[
    \sum_{k=0}^\infty \delta_\lambda(k)\mE^{\mu,\phi'}[r(\mu_k)\phi'_k(\mu_k)+\lambda \E(\phi'_k(\mu_k))]\leq \sum_{k=0}^\infty \delta_\lambda(k)\mE^{\mu,\phi_\lambda}[r(\mu_k)\phi_\lambda(\mu_k)+\lambda \E(\phi_\lambda(\mu_k))],
    \]
    implying 
    \begin{equation}\label{epsilon-eq-proof-1}
        J^{\psi\oplus_1\phi_\lambda}(\mu)\leq J^{\phi_\lambda}(\mu)+\delta J^\lambda+\delta \E^\lambda,
    \end{equation}
    where
    \begin{align*}
        &\delta J^\lambda:= \sum_{k=0}^\infty |\delta_\lambda(k)-\delta(k)|(\mE^{\mu,\phi'}|r(\mu_k)\phi'_k(\mu_k)| +\mE^{\mu,\phi_\lambda}|r(\mu_k)\phi_\lambda(\mu_k)| ),\\
        &\delta \E^\lambda:= \lambda \sum_{k=0}^\infty\delta_\lambda(k)(\mE^{\mu,\phi'}|\E(\phi'_k(\mu_k))|+\mE^{\mu,\phi_\lambda}|\E(\phi_\lambda(\mu_k))|).
    \end{align*}
    Because $r$, $\E(\phi')$ and $\E(\phi_\lambda)$ are all uniformly bounded, we use Proposition \ref{lambdavanishing} below to conclude that $\delta \E^\lambda\to 0$ as $\lambda\to 0$. On the other hand, by Lemma \ref{formulationclassical} and $\sum_k^\infty \phi_k(\mu)\prod_{j=0}^{k-1} (1-\phi_j(\mu))\leq 1$ for any $\phi\in \F$, $\mu\in \barS$, dominated convergence theorem is applicable to argue that $\delta J^\lambda\to 0$ as $\lambda\to 0$. Thus, we can choose $\lambda$ sufficiently small, such that $\delta J^\lambda+\delta \E^\lambda <\epsilon$. Note that, the choice of $\lambda$ is independent of $\psi\in [0,1]$ and $\mu\in \barS$. We thus get \eqref{epsilon-eqcond} from \eqref{epsilon-eq-proof-1}.
\end{proof}
Another main result of this paper is the next theorem, confirming the existence of the relaxed equilibria to the original time-inconsistent mean-field stopping problem without entropy regularization. The existence of equilibria of the time-inconsistent control/stopping problems is known to be challenging. {  Thanks to the regularized equilibrium approach, the existence of relaxed equilibria can be obtained without any additional assumption on the discount function. In particular, there is no need to assume special forms of discount functions such as the quasi-hyperbolic discount in a series of recent papers \cite{BJN2020}, \cite{jaskiewicz2021markov} and \cite{BRW2022}.
Furthermore, we can remove the decreasing impatience condition, which plays a key role in \cite{Huang2019} and some subsequent studies.}

\begin{theorem}\label{existoriginal}
    Let Assumptions \ref{as1}, \ref{Tassumption2} and \ref{dominate} hold. There exists a relaxed equilibrium for the original problem as in Definition \ref{originalequilibrium}.
\end{theorem}

{  
The proof of Theorem \ref{existoriginal} is presented in Subsection \ref{proof:existoriginal}.

\begin{remark}
	Problems in Sections \ref{MF-MDP} and \ref{exisresults} (regularized or not) resemble classical single-agent stopping problems with time-inconsistency. Indeed, our main results (and their proofs in Section \ref{proofs}) do not explicitly rely on the space $\bar S$ and its distance $\cW_1$. They can be easily extended to more general metric spaces. Nevertheless, because we focus on mean-field problems, the state space is inherently continuous. In this regard, our results seem to be new even if our problem is interpreted as a single-agent problem. Indeed, the existence of equilibria in time-inconsistent problems is usually based on structural assumptions such as finite state space or decreasing impatience of discount functions, while in this paper we can obtain its existence without imposing the conventional assumptions thanks to the entropy regularization. Moreover, the entropy regularized equilibrium also allows us to devise some explicit policy iteration learning algorithm to find the time-consistent equilibrium stopping policies in a mean-field model, which is also new to the literature; see Subsection \ref{exm:ETF:option}. 
	
	\end{remark}
}

\section{Connections to the $N$-agent Problem}\label{finiteconvergencesection}
The goal of this section is to discuss some connections between the mean-field MDP in \eqref{transition} and \eqref{Jdef}, denoted by ({\bf MF-MDP}) and the multi-agent MDP ({\bf N-MDP}). To this end, we first introduce a limiting Mckean-Vlasov type MDP ({\bf Limit-MDP}). We provide the propagation of chaos results and show the convergence of value function under a Lipschitz strategy, both with explicit convergence rate. We then show that the transition rule $T_0$ can be naturally induced by the transition rule (denoted by $T^r_0$) of ({\bf Limit-MDP}), and Assumptions \ref{as1}-\ref{dominate} can be implied by similar assumptions on $T^r_0$ and the common noise. Finally, we show that the regularized equilibrium given in Corollary \ref{extregulareq} can also serve as an $\epsilon$-equilibrium of ({\bf N-MDP}), provided that $N$ is sufficiently large and $\lambda$ is sufficiently small. The relationship among these MDP models is summarized in Figure \ref{MDPrelations}.
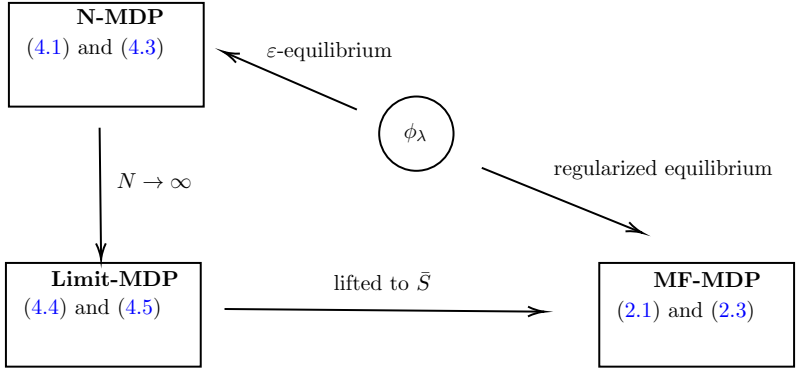
\begin{figure}
\centering
\tikzset{every picture/.style={line width=0.75pt}} 
\tikzset{global scale/.style={
    scale=#1,
    every node/.append style={scale=#1}
  }
}
\begin{tikzpicture}[x=0.75pt,y=0.75pt,yscale=-1,xscale=1,global scale=0.75]

\draw   (39,18) -- (170,18) -- (170,87.8) -- (39,87.8) -- cycle ;
\draw    (100,102) -- (100.98,188.8) ;
\draw [shift={(101,190.8)}, rotate = 269.35] [color={rgb, 255:red, 0; green, 0; blue, 0 }  ][line width=0.75]    (10.93,-3.29) .. controls (6.95,-1.4) and (3.31,-0.3) .. (0,0) .. controls (3.31,0.3) and (6.95,1.4) .. (10.93,3.29)   ;
\draw   (37,193) -- (168,193) -- (168,262.8) -- (37,262.8) -- cycle ;
\draw   (436,193) -- (567,193) -- (567,262.8) -- (436,262.8) -- cycle ;
\draw    (184,224) -- (397,225.78) ;
\draw [shift={(399,225.8)}, rotate = 180.48] [color={rgb, 255:red, 0; green, 0; blue, 0 }  ][line width=0.75]    (10.93,-3.29) .. controls (6.95,-1.4) and (3.31,-0.3) .. (0,0) .. controls (3.31,0.3) and (6.95,1.4) .. (10.93,3.29)   ;
\draw    (357,129) -- (463.15,172.05) ;
\draw [shift={(465,172.8)}, rotate = 202.08] [color={rgb, 255:red, 0; green, 0; blue, 0 }  ][line width=0.75]    (10.93,-3.29) .. controls (6.95,-1.4) and (3.31,-0.3) .. (0,0) .. controls (3.31,0.3) and (6.95,1.4) .. (10.93,3.29)   ;
\draw   (288,106) .. controls (288,92.19) and (299.19,81) .. (313,81) .. controls (326.81,81) and (338,92.19) .. (338,106) .. controls (338,119.81) and (326.81,131) .. (313,131) .. controls (299.19,131) and (288,119.81) .. (288,106) -- cycle ;
\draw    (273,89.8) -- (184.84,51.6) ;
\draw [shift={(183,50.8)}, rotate = 23.43] [color={rgb, 255:red, 0; green, 0; blue, 0 }  ][line width=0.75]    (10.93,-3.29) .. controls (6.95,-1.4) and (3.31,-0.3) .. (0,0) .. controls (3.31,0.3) and (6.95,1.4) .. (10.93,3.29)   ;

\draw (83,21) node [anchor=north west][inner sep=0.75pt]   [align=left] {\textbf{N-MDP}};
\draw (49,40) node [anchor=north west][inner sep=0.75pt]   [align=left] {\eqref{transtition-N-MDP} and \eqref{JNdef}};
\draw (110,131) node [anchor=north west][inner sep=0.75pt]   [align=left] {$N\to \infty$};
\draw (66,197) node [anchor=north west][inner sep=0.75pt]   [align=left] {\textbf{Limit-MDP}};
\draw (47,215) node [anchor=north west][inner sep=0.75pt]   [align=left] {\eqref{transition-limitMDP} and \eqref{JlimitMDP}};
\draw (471,198) node [anchor=north west][inner sep=0.75pt]   [align=left] {\textbf{MF-MDP}};
\draw (446,217) node [anchor=north west][inner sep=0.75pt]   [align=left] {\eqref{transition} and \eqref{Jdef}};
\draw (256,198) node [anchor=north west][inner sep=0.75pt]   [align=left] {lifted to $\barS$};
\draw (303,97) node [anchor=north west][inner sep=0.75pt]   [align=left] {$\phi_\lambda$};
\draw (404,122) node [anchor=north west][inner sep=0.75pt]   [align=left] {regularized equilibrium};
\draw (211,43) node [anchor=north west][inner sep=0.75pt]   [align=left] {$\epsilon$-equilibrium};
\end{tikzpicture}
\caption{Relationship among different MDP models.}
\label{MDPrelations}
\end{figure}

\subsection{The Convergence of ({\bf N-MDP}) to ({\bf Limit-MDP}) }\label{convergencesubsection}
Let $[N]:= \{1,2,\cdots,N\}$ be the index set of agents in the multi-agent MDP model. In this section, we consider a fixed probability space $(\Omega,\bF,\mP)$ that supports the following randomnesses:
\begin{itemize}
    \item[(1)] The initial states of all agents $\{\xi^i\}_{i\in [N]}$.
    \item[(2)] The idiosyncratic noises of all agents $\{Z^i_k\}_{i\in [N],k\in \mT}$.
    \item[(3)] The common noise $\{Z^0_k\}_{k\in \mT}$, as defined in Section \ref{MF-MDP}.
    \item[(4)] The random device for the social planner to determine whether to stop or not $\{U_k\}_{k\in \mT}$.
\end{itemize}

It is assumed that all these randomness are independent. Suppose that $\{\xi^i\}_{i\in [N]}$ are valued in $S$ with a common distribution $\nu_0\in \cP(S)$, and $\{Z^i_k\}_{i\in [N],k\in \mT}$ are valued in $\cZ$ with a common distribution $\cL(\cZ)'$. Moreover, the random devices $\{U_k\}_{k\in \mT}$ are all uniform random variables on $[0,1]$. For any $\phi\in \F_S$, it induces a strategy in the multi-agent problem, given by $\phi_N(\vec{x})=\phi(\frac{1}{N}\sum_{j\in [N]}\delta_{x_j})$, where $\vec{x}=(x_1,\cdots,x_N)$. 

With the above notations, let us consider the transition rule in the multi-agent problem:
\begin{equation}\label{transtition-N-MDP}
    X^{i,N,\phi}_{k+1}=T^r\left(X^{i,N,\phi}_k,\frac{1}{N}\sum_{j\in [N]}\delta_{X_k^{i,N,\phi}},\ind_{\{U_{k+1}\leq \phi_N(\vec{X}^{N,\phi}_k)\}},Z^i_{k+1},Z^0_{k+1}\right), X^{i,N,\phi}_0=\xi^i.
\end{equation}
It is assumed that $T^r$ satisfies
\begin{equation}\label{Tr0def}
    T^r(x,\mu,a,z',z)=\left\{
    \begin{aligned}
        &T^r_0(x,\mu,z',z), & a=0,\\
        &\triangle_S, & a=1,
    \end{aligned}
    \right.
\end{equation}
with a function $T^r_0:S\times \barS\times \{0,1\}\times \cZ\times \cZ\to S$. Here, as in the mean-field counterpart, $\triangle_S$ is a state added to $S$ to label the stopped agents. Therefore, we also assume $T^r_0(\triangle_S,\mu,a,z',z)\equiv\triangle_S$. In this multi-agent model, the total reward of the (induced) strategy $\phi_N$ is defined by 
\begin{equation}\label{JNdef}
    J^{i,N,\phi}(\xi^i)=\sum_{k=0}^\infty \delta(k)\mE \left[f\left(X^{i,N,\phi}_k,\frac{1}{N}\sum_{j\in [N]}\delta_{X^{i,N,\phi}_k}\right)\phi_N(\vec{X}^{N,\phi}_k)\right],
\end{equation}
for some reward function $f:S\times \barS\to \mR$. Letting $N\to \infty$ in \eqref{transtition-N-MDP} and \eqref{JNdef}, we heuristically derive the limiting problem, i.e., ({\bf Limit-MDP}), with transition rule and value function given by  
\begin{align}
    & X^{i,\phi}_{k+1}=T^r\left(X^{i,\phi}_k,\mP^0_{X^{i,\phi}_k},\ind_{\{U_{k+1}\leq \phi(\mP^0_{X^{i,\phi}_k})\}},Z^i_{k+1},Z^0_{k+1}\right), X^{i,\phi}_0=\xi^i, \label{transition-limitMDP}\\
     &J^{i,\phi}(\xi^i)=\sum_{k=0}^\infty \delta(k)\mE \left[f\left(X^{i,\phi}_k,\mP^0_{X^{i,\phi}_k}\right)\phi(\mP^0_{X^{i,\phi}_k})\right]\label{JlimitMDP}.
\end{align}
To establish some convergence results, we impose the following Lipschitz assumptions on $T^r_0$ and $f$; {   see similar assumptions (($\rm \bf HF_{lip}$) and ($\rm \bf Hf_{lip}$)) in \cite{Motte2022}}.
\begin{assumption}\label{Tr-f-Lip}
    There exist constants $L_5,L_6>0$ such that for any $x,x'\in S$, $\mu,\mu'\in \barS$, $z,z'\in \cZ$ we have
    \begin{align*}
    &\int_\cZ d(T^r_0(x,\mu,z',z),T^r_0(x',\mu',z',z))\cL(\cZ)'(\md z')\leq L_5(d(x,x')+\cW_1(\mu,\mu')),\\
\text{and}\ \ \  &|f(x,\mu)-f(x',\mu')|\leq L_6 (d(x,x')+\cW_1(\mu,\mu')).
    \end{align*}
\end{assumption}
\begin{remark}
    Because $S$ is assumed to be compact (thus so is $\barS$), Assumption \ref{Tr-f-Lip} implies that $f$ is uniformly bounded by a constant $L_7>0$.
\end{remark}
Before presenting the convergence result, we need the following lemma on non-asymptotic estimation of empirical measures, which is taken from \cite{Motte2022}. We denote by diam$(S)$ the diameter of the compact metric space $S$.
\begin{lemma}\label{MNrate}
    Define
   \begin{equation}\label{MNdef}
       M_N:= E_{\vec{Y}\sim \nu^{\otimes N}}\cW_1\left(\nu_{N,\vec{Y}},\nu\right),
   \end{equation}
   where $\vec{Y}=(Y_1,Y_2,\cdots,Y_N)$ is a vector of i.i.d. random variables with distribution $\nu$, $E_{\vec{Y}\sim \nu^{\otimes N}}$ is the corresponding expectation operator, and $\nu_{N,\vec{Y}}:= \frac{1}{N}\sum_{j\in [N]}\delta_{Y_j}$ is the empirical measure. Then, we have the following results:
   \begin{itemize}
       \item[(1).] If $S\subset \mR^d$ for some $d\geq 1$, then: $M_N=O(N^{-1/2})$ for $d=1$, $M_N=O(N^{-1/2}\log(1+N))$ for $d=2$, and $M_N=O(N^{-1/d})$ for $d\geq 3$.
       \item[(2).] If $S$ is finite, then $M_N=O(N^{-1/2})$.
       \item[(3).] If for all $\delta>0$, the smallest number of balls with radius $\delta$ covering $S$ is smaller than $O\left(\left({\rm diam}(S)/\delta    \right)^\theta\right)$ for some $\theta>2$, then $M_N=O(N^{-1/\theta})$.
   \end{itemize}
\end{lemma}
\begin{proof}
    See Lemma 2.1 and Remark 2.3 of \cite{Motte2022}.
\end{proof}
\begin{remark}
    In \eqref{MNdef}, we use the notation $E_{\vec{Y}\sim \nu^{\otimes N}}$ instead of $\mE$ to avoid confusion with the expectation of probability measure $\mP$, which includes all randomnesses considered in this section. However, when considering the mean-field limit, randomnesses from $Z^0$ and $U$ are not aggregated. A more precise expression of $M_N$ is
     \[
    M_N := \sup_{\nu\in \barS} \int_{S^N}\cW_1\left(\frac{1}{N}\sum_{j\in [N]}\delta_{y_j},\nu   \right)\nu^{\otimes N}(\md y_1,\md y_2,\cdots,\md y_N),
    \]
    where $\nu^{\otimes N}:= \nu\times \nu\times \cdots \times \nu$ ($N$ times ) is the product measure. However, for ease of presentation, we use \eqref{MNdef}.
\end{remark}

\begin{theorem}\label{finiteconvergence}
    Under Assumption \ref{Tr-f-Lip}, for any given $\phi\in \F^{\rm Lip}_S$, $k\in \mT$, $i\in [N]$ and $\lambda>0$, we have that
    \begin{align}
    &\mE d(X^{i,N,\phi}_k,X^{i,\phi}_k)\leq C_1(k,L_5,\|\phi\|_{\rm Lip})M_N,\label{Xconvergence}\\
    &\mE |J^{i,N,\phi}_\lambda(\xi^i)-J^{i,\phi}_\lambda(\xi^i)|\leq C_2(\lambda, L_5,L_6,L_7,\|\phi\|_{\rm Lip})M_N,\label{valueconvergence}
    \end{align}
    where $J^{i,N,\phi}_\lambda$ and $J^{i,\phi}_\lambda$ are defined in the same way as in \eqref{JNdef} and \eqref{JlimitMDP} with the discount function replaced by $\delta_\lambda:= \delta(k)\left( \frac{1}{1+\lambda}   \right)^{k^2}$, and $M_N$ together with its convergence rate is given in Lemma \ref{MNrate}.  
\end{theorem}

\begin{proof}[Proof of Theorem \ref{finiteconvergence}]
    We first prove \eqref{Xconvergence} by the mathematical induction. The statement clearly holds for $k=0$ because $X^{i,N,\phi}=X^{i,\phi}=\xi^i$. Suppose now \eqref{Xconvergence} holds for $k$, we can estimate $\mE d(X^{i,N,\phi}_{k+1},X^{i,\phi}_{k+1})$ by
    \begin{align}
    \mE d(X^{i,N,\phi}_{k+1},X^{i,\phi}_{k+1})\leq & \mE d(X^{i,N,\phi}_{k+1},X^{i,\phi}_{k+1})\ind_{\{U_{k+1}> \phi_N(\vec{X}_k^{N,\phi})\vee \phi(\mP^0_{X^{i,N,\phi}_k})\}} \nonumber\\
    &+\mE d(X^{i,N,\phi}_{k+1},X^{i,\phi}_{k+1})\ind_{\{U_{k+1}\in [\phi_N(\vec{X}_k^{N,\phi}), \phi(\mP^0_{X^{i,N,\phi}_k})]\}} \nonumber\\
    \leq & \mE d\left(T^r_0\left(X^{i,N,\phi}_k,\nu_{N,\vec{X}^{N,\phi}_k},Z^i_{k+1},Z^0_{k+1}\right),T^r_0\left(X^{i,\phi}_k,\mP^0_{X^{i,\phi}_k},Z^i_{k+1},Z^0_{k+1}\right)\right)\label{dXest1}\\
    &+ 2{\rm diam}(S)\mE | \phi_N(\vec{X}_k^{N,\phi})-\phi(\mP^0_{X^{i,N,\phi}_k})    |\nonumber\\
    \leq & L_5\mE\left( d\left(X^{i,N,\phi}_k, X^{i,\phi}_k \right)+\cW_1(\nu_{N,\vec{X}^{N,\phi}_k},  \mP^0_{X^i_k} ) \right)\nonumber\\
    &+2{\rm diam}(S)\|\phi\|_{\rm Lip}\mE\cW_1(\nu_{N,\vec{X}^{N,\phi}_k},  \mP^0_{X^{i,\phi}_k}).\nonumber
    \end{align}
Recall the notation $\nu_{N,\vec{y}}:= \frac{1}{N}\sum_{j\in [N]}\delta_{y_j}$ for any $N$-dimensional vector $\vec{y}$. We claim that, for two vectors $\vec{y}, \vec{y'}$,  
\begin{equation}\label{empiricalest}
    \cW_1(\nu_{N,\vec{y}},\nu_{N,\vec{y'}})\leq \frac{1}{N}\sum_{j\in [N]}d(y_j,y'_j).
\end{equation}
Indeed, let {\bf j} be a uniform random variable on $[N]$, then $(y_{\bf j}, y'_{\bf j})$ is a coupling of $\nu_{N,\vec{y}},\nu_{N,\vec{y'}}$. It follows from the definition of Wasserstein metric that 
\[
 \cW_1(\nu_{N,\vec{y}},\nu_{N,\vec{y'}})\leq E_{\bf j}|y_{\bf j}-y'_{\bf j}|\leq \frac{1}{N}\sum_{j\in [N]}d(y_j,y'_j),
\]
which gives \eqref{empiricalest}. 

Now, thanks to \eqref{empiricalest}, we further obtain that
    \begin{equation}\label{dXest-W1} 
    \begin{aligned}
      \mE\cW_1(\nu_{N,\vec{X}^{N,\phi}_k},  \mP^0_{X^{i,N}_k})&\leq   \mE\cW_1(\nu_{N,\vec{X}^{N,\phi}_k}, \nu_{N,\vec{X}^{\phi}_k})+ \mE\cW_1(\nu_{N,\vec{X}^{\phi}_k}, \mP^0_{X^{i,\phi}_k})\\
      &\leq \mE d(X^{i,N,\phi}_k,X^{i,\phi}_k)+M_N,
    \end{aligned}
    \end{equation}
where we have used the law of iterated expectation that
\[
\mE\cW_1(\nu_{N,\vec{X}^{\phi}_k}, \mP^0_{X^{i,\phi}_k})=\mE  E_{\vec{X}\sim (\mP^0_{X^{i,\phi}_k})^{\otimes N} } \cW_1(\nu_{N,\vec{X}}, \mP^0_{X^{i,\phi}_k})\leq M_N.
\]
Plugging \eqref{dXest-W1} back into \eqref{dXest1} and using the induction hypothesis of step $k$, we get that
\begin{align*}
\mE d(X^{i,N,\phi}_{k+1},X^{i,\phi}_{k+1})\leq & C'_1 \mE d(X^{i,N,\phi}_k,X^{i,\phi}_k) + C_2'M_N \\
\leq & (C'_1 C_1(k)+C'_2)M_N,
\end{align*}
where $C'_1:=2L_5+2{\rm diam}(S)\|\phi\|_{\rm Lip} $ and $C'_2:=L_5+2{\rm diam}(S)\|\phi\|_{\rm Lip}$. Taking $C_1(k+1):=C'_1 C_1(k)+C'_2$, we complete the proof of \eqref{Xconvergence}. Moreover, we can take the explicit form of $C_1(k)=\frac{(C_1')^kC_2'}{C_1'-1}$.

We next show \eqref{valueconvergence}. By \eqref{JNdef}, \eqref{JlimitMDP}, \eqref{dXest-W1}, Assumption \ref{Tr-f-Lip} and direct computations, one can derive that
\begin{align*}
|J^{i,N,\phi}_\lambda(\xi^i)-J^{i,\phi}_\lambda(\xi^i)|\leq &\sum_{k=0}^\infty \delta_\lambda(k)\mE \left|   f(X^{i,N,\phi}_k,\nu_{N,\vec{X}^{N,\phi}_k})\phi(\nu_{N,\vec{X}^{N,\phi}_k})- f(X^{i,\phi}_k,\mP^0_{X^{i,\phi}_k})\phi(\mP^0_{X^{i,\phi}_k})  \right|\\
\leq & L_6\sum_{k=0}^\infty \delta_\lambda(k)\left[\mE d(X^{i,N,\phi}_k,  X^{i,\phi})+\mE \cW_1(\nu_{N,\vec{X}^{N,\phi}_k},\mP^0_{X^{i,\phi}_k})\right]\\
&+L_7 \|\phi\|_{\rm Lip} \sum_{k=0}^\infty \delta_\lambda(k)\cW_1(\nu_{N,\vec{X}^{N,\phi}_k},\mP^0_{X^{i,\phi}_k})\\
\leq & C'_3 \sum_{k=0}^\infty \delta_\lambda(k)((C'_1)^k M_N +M_N)\\
\leq & \left(C'_4 + C_3' \sum_{k=0}^\infty \delta_\lambda(k)(C'_1)^k      \right)M_N,
\end{align*}
where $C_4',C_3'$ are constants depending on $L_5$, $L_6$, $L_7$ and $\|\phi\|_{\rm Lip}$. Clearly, $\sum_{k=0}^\infty \delta_\lambda(k)(C'_1)^k<\infty$ for $\lambda>0$. Therefore, choosing $C_2=C'_4 + C_3' \sum_{k=0}^\infty \delta_\lambda(k)(C'_1)^k  $ yields to the desired \eqref{valueconvergence}.
\end{proof}

\begin{remark}\label{rmk:finiteconvergence}
    {  The proof of Theorem \ref{finiteconvergence} is inspired by \cite{Motte2022}. Nevertheless, our proof addresses some distinct technical challenges that only appear in our framework.} We establish convergence results for a given {\it Lipschitz} strategy and we also need to consider a {\it discount regularization} $\delta_\lambda$ (see also Section \ref{proofs} of this paper). There are two main reasons. Firstly, our mean-field MDP is time-inconsistent, and the closed-loop control is more natural and convenient to define the sub-game time-consistent relaxed equilibria. A key difficulty of the closed-loop formulation is that controls vary with state variables, which requires us to consider the convergence of controls as $N\to \infty$. In contrast, in the convergence analysis of \cite{Motte2022} (see Theorem 2.1 therein), the control is fixed as $\alpha^{i,\pi}$ when $N\to\infty$. Secondly, a reason to consider the regularized discount  $\delta_\lambda$ is that our problem is an infinite horizon stopping problem under non-exponential discount. The argument used in the proof of Theorem 2.1 in \cite{Motte2022} relies heavily on the structure of exponential discount and is not applicable to general $\delta$. Although some arguments may hold valid in our case if we further assume that $\delta$ decays sufficiently fast, we choose not to impose such restrictive structural conditions on $\delta$. In subsection \ref{epsilon-eq-finite}, it is shown that Theorem \ref{finiteconvergence} is sufficient in finding an $\epsilon$-equilibria for the multi-agent problem. We emphasize that with positive $\lambda$, the convergence rate of $J_\lambda$ is improved to $M_N$ compared to $M_N^\gamma$ with a constant $\gamma\leq 1$. As a price to pay, the constant $C_2$ may explode when $\lambda\to 0$.
\end{remark}

\subsection{The Relationship between ({\bf Limit-MDP}) and ({\bf MF-MDP})}\label{relationlimit}
In this subsection, we establish the connection between ({\bf Limit-MDP}) in subsection \ref{convergencesubsection} and ({\bf MF-MDP}) in the main body of this paper. We consider the function $T_0:\barS\times \cZ\to \barS$ defined by 
\begin{equation}\label{relationT0}
T_0(\mu,z):= T^r_0(\cdot,\mu,\cdot,z)_\# (\mu \times \cL(\cZ)'),
\end{equation}
where $T^r_0$ is from \eqref{Tr0def}. We recall that $\cL(\cZ)'$ is the distribution of idiosyncratic noises. If we assume $\triangle:= \delta_{\triangle_S}$, and define a function $T:\barS\times \{0,1\}\times \cZ\to \barS $ as in \eqref{T0def}, it is clear that
\[
T(\mu,a,z)=T^r(\cdot,\mu,a,\cdot,z)_\#(\mu\times \cL(\cZ)').
\]
Consider $\mP^0:= \mP(\cdot| \{U_k\}_{k\in \mT}, \{Z^0_k\}_{k\in \mT})$ as a probability measure on $\Omega$. From the first equation in \eqref{transition-limitMDP}, we derive
\begin{align*}
\mP^0_{X^{i,\phi}_{k+1}}&=(X^{i,\phi}_{k+1})_\#\mP^0 \\
&=\left(T^r\left(X^{i,\phi}_k,\mP^0_{X^{i,\phi}_k},\ind_{\{U_{k+1}\leq \phi(\mP^0_{X^{i,\phi}_k})\}},Z^i_{k+1},Z^0_{k+1}\right)\right)_\# \mP^0\\
&=T^r_\# \circ \left(X^{i,\phi}_k,\mP^0_{X^{i,\phi}_k},\ind_{\{U_{k+1}\leq \phi(\mP^0_{X^{i,\phi}_k})\}},Z^i_{k+1},Z^0_{k+1}\right)_\# \mP^0\\
&=\left(T^r\left(\cdot,\mP^0_{X^{i,\phi}_k},\ind_{\{U_{k+1}\leq \phi(\mP^0_{X^{i,\phi}_k})\}},\cdot,Z^0_{k+1}\right)\right)_\#(\mP^0_{X^{i,\phi}_k}\times \cL(\cZ)')\\
&=T(\mP^0_{X^{i,\phi}_k},\ind_{\{U_{k+1}\leq \phi(\mP^0_{X^i_k})\}},Z^0_{k+1}),
\end{align*}
where we have used facts that $U_{k+1}$ and $Z^0_{k+1}$ are deterministic under $\mP^0_{X^{i,\phi}_k}$, and $Z^i$ is independent from $Z^0$. We denote $\mu_k:= \mP^0_{X^{i,\phi}_k}$, it then holds that $\mu_0=\nu_0$ and 
\[
\mu_{k+1}=T(\mu_k,\ind_{\{U_{k+1}\leq \phi(\mu_k)\}},Z^0_{k+1}).
\]
Note that given $\mu_k$, $\ind_{\{U_{k+1}\leq \phi(\mu_k)\}}$ has the same distribution as $\xi^{\phi(\mu_k)}$, which is a Bernoulli random variable with success probability $\phi(\mu_k)$. Choosing $r(\mu):= \int_S f(x,\mu)\mu(\md x)$, and by \eqref{transition}, \eqref{JlimitMDP} and the definition of $\mP^{\mu,\phi}$, we conclude that
\begin{equation}\label{valuerelation}
    J^{i,\phi}(\xi^i)=\sum_{k=0}^\infty \delta(k)\mE^{\nu_0,\phi}r(\mu_k)\phi(\mu_k)=J^{\phi}(\nu_0).
\end{equation}
In other words, ({\bf Limit-MDP}) naturally connects to ({\bf MF-MDP}) by lifting the state space to $\barS$. In fact, this is our motivation to consider the mean-field formulation at the beginning. We now discuss how assumptions on $T^r_0$ induce assumptions on $T_0$.
\begin{proposition}\label{assumptionrelation}
    Assumption \ref{Tr-f-Lip} implies Assumptions \ref{as1} and \ref{Tassumption2}.
\end{proposition}
\begin{proof}
    We first prove the Lipschitz continuity of $r$. Under Assumption \ref{Tr-f-Lip}, for any $x,x'\in S$ and $\mu,\mu'\in \barS$, it holds that $|f(x,\mu)-f(x,\mu')|\leq L_6\cW_1(\mu,\mu')$ and $|f(x,\mu)-f(x',\mu)|\leq L_6d(x,x')$. Consequently,
    \begin{align*}
        |r(\mu)-r(\mu')|&\leq \left| \int_S f(x,\mu)\mu(\md x)-\int_S f(x,\mu)\mu'(\md x)\right|+\int_S |f(x,\mu)-f(x,\mu')|\mu'(\md x) \\
        &\leq 2L_6\cW_1(\mu,\mu').
    \end{align*}
    Let us denote by ${\rm Lip}_1$ the set of 1-Lipschitz functions on $S$. To show the Lipschitz continuity of $T_0$, we first note that, for any $f\in {\rm Lip}_1$, $x,x'\in S$, $\mu,\mu'\in \barS$ and $z,z'\in \cZ$,
   \begin{align*}
    &|f(T^r_0(x,\mu,z',z))-f(T^r_0(x',\mu,z',z))|\leq L_5d(x,x'),\\
\text{and}\ \   &|f(T^r_0(x,\mu,z',z))-f(T^r_0(x,\mu',z',z))|\leq L_5\cW_1(\mu,\mu').
   \end{align*}
   By definitions of the Wasserstein metric and of the push-forward operator, we have that
   \begin{align*}
    \cW_1(T_0(\mu,z),T_0(\mu',z))=&\sup_{f\in {\rm Lip}_1}\left|\int_S f(x)T^r_0(\cdot,\mu,\cdot,z)_\#(\mu\times \cL(\cZ)')(\md x)\right. \\
    &\left.-\int_S f(x)T^r_0(\cdot,\mu',\cdot,z)_\#(\mu'\times \cL(\cZ)')(\md x)                 \right|\\
    =&\sup_{f\in {\rm Lip}_1}\left| \int_{S\times \cZ} f(T^r_0(x,\mu,z',z))\mu(\md x)\cL(\cZ)(\md z')\right.\\
    &\left.-\int_{S\times \cZ} f(T^r_0(x,\mu',z',z))\mu'(\md x)\cL(\cZ)(\md z')\right|\\
    \leq &\int_\cZ \sup_{f\in {\rm Lip}_1}\left|\int_S f(T^r_0(x,\mu,z',z))\mu(\md x) -\int_S f(T^r_0(x,\mu,z',z))\mu'(\md x)           \right|\\
    &+\int_{S\times \cZ}|f(T^r_0(x,\mu,z',z))-f(T^r_0(x,\mu',z',z))| \mu'(\md x)\cL(\cZ)'(\md z')\\
    \leq & 2L_5\cW_1(\mu,\mu').
   \end{align*}
   The conclusion readily follows.
\end{proof}
We now show how Example \ref{exm1} can be constructed from ({\bf Limit-MDP}), hence Assumptions \ref{Tassumption} and \ref{dominate} can be verified. 
\begin{example}\label{exm1construct}
    Suppose $|S|=d\in  \mN$, $d\geq 2$. We assume that the common noise takes values in $\cZ:=\barS$, and the idiosyncratic noises take values in $[0,1]\times \{0,1\}$. We write $z'$ as $z'=(z'',l)\in [0,1]\times \{0,1\}$. It is also assumed that the idiosyncratic noises have the same distribution $\cL(\cZ)'= U_{[0,1]}\times \B(\epsilon)$ for $\epsilon>0$. Here, $U_{[0,1]}$ is the uniform distribution on $[0,1]$ and $\B(\epsilon)$ is the Bernoulli distribution with the success probability $\epsilon$. The slight modification does not change all results and proofs in this section because the space of idiosyncratic noises can actually be any measurable space instead of $\cZ$. Furthermore, we take a measurable map $R^r:[0,1]\times \barS\to S $, such that for any $z\in \barS$, $R^r(\cdot,z)_\# U_{[0,1]}=z$. Such an $R^r$ exists by Lemma 5.29 in \cite{carmona2018probabilistic}. Finally, we take a function $T^r_1:S\times \barS\times [0.1]\to S$ and consider that
    \[
T^r_0(x,\mu,z'',l,z):=\left\{
\begin{aligned}
   & T^r_1(x,\mu,z''), &l=0,\\
   & R^r(z'',z), &l=1.
\end{aligned}
\right.
    \]
 It is not hard to show in this example that $T_0(\mu,z)=(1-\epsilon)T_1(\mu)+\epsilon z$ with $T_1(\mu)=T^r_1(\cdot,\mu,\cdot)_\# (\mu\times U_{[0,1]})$. Thus the main result in Example \ref{exm1} follows.   
\end{example}
For an uncountable space $S$, it is very challenging to provide some general sufficient conditions such that Assumption \ref{dominate} holds, which will be left as our future research. The main obstacle lies in the fact that there is no canonical measure on the space $\barS=\cP(S)$ possessing certain invariance properties. This is the case even for $S=[0,1]$. Here, we only provide a concrete example when Assumption \ref{dominate} holds.
\begin{example}\label{dominateexm}
    This example is inspired by \cite{jaskiewicz2021markov}; see Example 4.8 therein. We consider a mean-field counterpart of it. Suppose that $S$ is a general compact metric space and let us fix two measures $\mu_1,\mu_2\in \cP(S)$. We work with $\cZ=[0,1]$ and $\cL(\cZ)=U_{[0,1]}$. For a function $\alpha_0:\barS\to [0,1]$, let us consider $T_0(\mu,z)=\alpha_\epsilon(\mu,z)\mu_1+(1-\alpha_\epsilon(\mu,z))\mu_2$, where $\alpha_\epsilon(\mu,z)=(1-\epsilon)\alpha_0(\mu)+\epsilon z$. This model can be constructed from ({\bf Limit-MDP}) in a similar way as in Example \ref{exm1construct} and we hence omit the details. Moreover, it shall be clear from the construction that we can take $\alpha_0(\mu)=\int_S \alpha(x,\mu)\mu(\md x)$ for some $\alpha$, and Assumption \ref{Tr-f-Lip} is satisfied if $\alpha$ is Lipschitz in $(x,\mu)$. 
    
    Now, for each test function $f\in C(\barS)$, we have
    \begin{align*}
        \int_{\barS} f(\nu) T_0(\mu,\cdot)_\# U_{[0,1]}(\md \nu)&=\int_0^1 f(T_0(\mu,z))\md z\\
        &=\int_0^1 f([(1-\epsilon)\alpha_0(\mu)+\epsilon z]\mu_1+[1-(1-\epsilon)\alpha_0(\mu)-\epsilon z]\mu_2 )\md z\\
        &=\int_0^1 f(z'\mu_1+(1-z')\mu_2)\frac{1}{\epsilon}\ind_{[(1-\epsilon)\alpha_0(\mu),(1-\epsilon)\alpha_0(\mu)+\epsilon]}(z')\md z'.
    \end{align*}
Let us introduce $L(z'):= z'\mu_1+(1-z')\mu_2$, $\sL_0:= L_\#U_{[0,1]}$, $I_\epsilon(\mu):= [(1-\epsilon)\alpha_0(\mu),(1-\epsilon)\alpha_0(\mu)+\epsilon]$, $B_\epsilon(\mu):= \{L(z'):z'\in I_\epsilon(\mu)   \}$. With these notations, we have
\begin{align*}
\int_{\barS} f(\nu)T_0(\mu,\cdot)_\#U_{[0.1]}(\md \nu)&=\int_0^1 f(L(z'))\frac{1}{\epsilon}\ind_{I_\epsilon(\mu)}(z')\md z'\\
&=\int_{\barS} f(\nu) \frac{1}{\epsilon}\ind_{B_\epsilon(\mu)}(\nu)\sL_0(\md \nu).
\end{align*}
In other words, for any $\mu\in \barS$, $T_0(\mu,\cdot)_\# U_{[0,1]}$ has the density function $f^0(\mu,\cdot)= \frac{1}{\epsilon}\ind_{B_\epsilon(\mu)}$ with respect to $\sL_0$, a fixed probability measure on $\barS$. Moreover, it holds that
\begin{align*}
    \int_{\barS} |f^0(\mu,\nu)-f^0(\mu',\nu)| \sL_0(\md \nu)&=\frac{1}{\epsilon}\int_{\barS}|\ind_{B_\epsilon(\mu)}(\nu)-\ind_{B_\epsilon(\mu')}(\nu)|\sL_0(\md \nu)\\
    &=\frac{1}{\epsilon}\int_0^1 |\ind_{I_\epsilon(\mu)}-\ind_{I_\epsilon(\mu')}|\md z\\
    &\leq \frac{2(1-\epsilon)}{\epsilon}|\alpha_0(\mu)-\alpha_0(\mu')|.
\end{align*}
Consequently, Assumption \ref{dominate} is satisfied if $\alpha_0$ is Lipschitz, which is clear from the Lipschitz continuity of $\alpha$ in ({\bf Limit-MDP}). This example can be readily generalized to other models where $T_0(\mu,z)$ can be parameterized by a finite-dimensional parameter space.
\end{example}

\subsection{Approximated Equilibrium of ({\bf N-MDP})}\label{epsilon-eq-finite}
In this section, we combine results from Section \ref{exisresults} and Subsection \ref{convergencesubsection} to obtain an approximated equilibrium for the multi-agent MDP problem ({\bf N-MDP}). We first define the approximated equilibrium.

\begin{definition}\label{Nepsilondef}
    For fixed $\epsilon>0$ and $N\in \mN$, $\phi^*\in \F_S$ is said to be an $\epsilon$-equilibrium of ({\bf N-MDP}) if for any $\psi\in [0,1]$, $\nu_0\in \barS$ and i.i.d. random variables $\{\xi^i\}_{i\in [N]}$ with the distribution $\nu_0$, we have
    \[
    \frac{1}{N}\sum_{i\in [N]} J^{i,N,\psi\oplus_1\phi^*}(\xi^i)\leq \frac{1}{N}\sum_{i\in [N]} J^{i,N,\phi^*}(\xi^i)+\epsilon.
    \]
    Here, we recall that $\psi\oplus_1\phi^* \in \F$ is given by the sequence $\{\phi_k\}_{k\in \mT}$ with $\phi_0=\psi$, $\phi_k=\phi$, $k\geq 1$.
\end{definition}
\begin{remark}
    Definition \ref{Nepsilondef} has a similar interpretation as Definition \ref{originalequilibrium}. Indeed, assessed by the average awards of all agents, the social planner's incentive to deviate from the approximated equilibria can be arbitrarily small. This is in line with the classical definition of equilibria in single-agent problems. We remark that in the multi-agent problem, we only consider strategies induced by $\phi\in \F_S$, i.e., mean-field strategies. In addition to technical consideration, we think this setting is also realistic because mean-field strategies have anonymity: the social planner needs not to know the precise states of every single agent but only how they are distributed. Such information is often accessible via statistical methods. 
\end{remark}
We now present the main result of this subsection.
\begin{theorem}\label{epsiloneq-N}
   For any $\epsilon>0$, $\phi_\lambda$ given by Corollary \ref{extregulareq} is an $\epsilon$-equilibrium for ({\bf N-MDP}) with $N$ agents, provided that $\lambda$ is sufficiently small and $N$ is sufficiently large.
\end{theorem}
\begin{proof}
    We first choose $\lambda>0$ sufficiently small such that
   \begin{align}
    & J^{\psi\oplus_1\phi_\lambda}_\lambda(\nu_0)\leq J^{\phi_\lambda}_\lambda(\nu_0)+\epsilon,\forall \psi\in [0,1], \label{epsiloneqMF}  \\
    &  \frac{1}{N}\sum_{i\in [N]} |J^{i,N,\psi\oplus_1\phi_\lambda}(\xi^i)-J^{i,N,\psi\oplus_1\phi_\lambda}_\lambda(\xi^i)| <\epsilon, \label{approxJa} \\
\text{and}\  \ \ & \frac{1}{N}\sum_{i\in [N]} |J^{i,N,\phi_\lambda}(\xi^i)-J^{i,N,\phi_\lambda}_\lambda(\xi^i)| <\epsilon. \label{approxJphi}
   \end{align}
Thanks to Assumptions \ref{as1} and \ref{Tr-f-Lip}, the choice of $\lambda$ only depends on $L_0$ and $L_7$. In particular, it is independent of $N$, $\psi$, and $\nu_0$. On the other hand, we write
{\small
\begin{align*}
J^{i,N,\psi\oplus \phi_\lambda}(\xi^i)=\mE f\left(\xi^i,\frac{1}{N}\sum_{j\in [N]} \delta_{\xi^j}\right)\psi+\sum_{k=1}^\infty \delta(k)\mE \left[f\left(X^{i,N,\phi}_k,\frac{1}{N}\sum_{j\in [N]}\delta_{X^{i,N,\phi_\lambda}_k}\right)\phi_\lambda\left( \frac{1}{N}\sum_{j\in [N]}\delta_{X^{i,N,\phi_\lambda}_k}  \right)\right].
\end{align*}}After some minor modifications in the proof of Theorem \ref{finiteconvergence}, we can obtain that there exists a constant $C_3(\lambda)>0$ (independent of $N$, $a$ and $\nu_0$) such that 
\begin{equation}\label{approxJaN}
    \frac{1}{N}\sum_{i\in [N]} | J^{i,N,\psi\oplus \phi_\lambda}_\lambda(\xi^i)-J^{i,\psi\oplus \phi_\lambda}_\lambda(\xi^i) |\leq C_3 M_N.
\end{equation}
Thanks to Lemma \ref{MNrate}, we can choose $N$ large enough such that $C_2(\lambda)M_N<\epsilon$ and $C_3(\lambda)M_N<\epsilon$. We thus deduce from \eqref{approxJaN} and \eqref{valueconvergence} that 
\begin{align}
    &  \frac{1}{N}\sum_{i\in [N]} | J^{i,N,\psi\oplus_1 \phi_\lambda}_\lambda(\xi^i)-J^{i,\psi\oplus_1 \phi_\lambda}_\lambda(\xi^i) | <\epsilon,  \label{approxJaNepsilon}  \\
 \text{and}\ \  &   \frac{1}{N}\sum_{i\in [N]} | J^{i,N,\phi_\lambda}_\lambda(\xi^i)-J^{i,\phi_\lambda}_\lambda(\xi^i) | <\epsilon. \label{approxJphiNepsilon}
\end{align}
Plugging \eqref{approxJa}, \eqref{approxJphi}, \eqref{approxJaNepsilon} and \eqref{approxJphiNepsilon} into \eqref{epsiloneqMF}, and using the relationship in \eqref{valuerelation}, we readily arrive at
\[
\frac{1}{N}\sum_{i\in [N]} J^{i,N,\psi\oplus_1\phi_\lambda}(\xi^i)\leq \frac{1}{N}\sum_{i\in [N]} J^{i,N,\phi_\lambda}(\xi^i)+5\epsilon,
\]
which completes the proof.
\end{proof}
\begin{remark}
    In some special cases, the order of $\lambda$ and $N$ to achieve an $\epsilon$-equilibrium can be obtained explicitly. For example, if $L_0':= \sum_{k=0}^\infty \delta(k)k^2<\infty$, we have
    \begin{align*}
    \sum_{k=0}^\infty \delta(k)\left(1-\left( \frac{1}{1+\lambda}   \right)^{k^2}    \right)&\leq \sum_{k=0}^\infty\delta(k)k^2 \frac{\lambda}{1+\lambda}\leq L_0' \lambda. 
    \end{align*}
    Therefore, from the proof of Theorem \ref{epsiloneq-MF}, it follows that \eqref{epsiloneqMF} holds if we choose
    \[
    \lambda<   \frac{\epsilon}{2(L_0'+e^{-1})}.
    \]
    Similarly, \eqref{approxJa} and \eqref{approxJphi} hold with $\lambda<O(\epsilon)$. With the convergence rate of $M_N$ explicitly given in Lemma \ref{MNrate}, the order of $N$ can be determined by $M_N<\frac{\epsilon}{C_2(\lambda)+C_3(\lambda)}$. By the proof of Theorem \ref{finiteconvergence}, both $C_2(\lambda)$ and $C_3(\lambda)$ tend to infinity as $\lambda\to 0$, and we have the order
    \[
    O\left(\sum_{k}\delta(k)\left( \frac{1}{1+\lambda}   \right)^{k^2}\left(1+\frac{1}{\lambda}\exp\left(\frac{1}{\lambda}\right)     \right)^k      \right),
    \]
    where we have used Lemma \ref{phiestuniformlemma} to estimate $\|\phi\|_{\rm Lip}$.
\end{remark}

\section{Examples}\label{sec:exm}
\subsection{An Example with Explicit Solution: R\&D Project Announcement}\label{Illexample}

We now give in this subsection an illustrative example, in which we can prove that the relaxed equilibrium is unique and can be explicitly characterized. If model parameters satisfy some extra conditions, we can further show that the pure strategy equilibrium does not exist. In addition, we will also numerically illustrate the convergence of regularized equilibrium to relaxed equilibrium as the parameter $\lambda\rightarrow 0$.

\begin{example}\label{nopure}
Among the large population, we assume that each representative agent only has two states $S=\{A,B\}$. At time $k\in \mT$, each agent at state A has the probability $p$ to remain at the same state and the probability $1-p$ to move to the other state $B$, while the state $B$ is the absorbing state. Moreover, it is assumed that the transition probability $p$ itself is a uniform random variable that $p\sim U_{[0,1]}$, which can be interpreted as common noise in the mean-field model. The goal of the social planner is to stop the large system in a centralized manner by achieving a higher proportion of the population that attains state B by the time of stopping. However, due to the waiting cost described by the discount factor, it is not optimal to wait forever. Instead, the social planner needs to make a strategic decision to stop under a non-exponential discount described below. In particular, let us denote by $\mu_k$ the proportion of agents in the population that still remain in the state $A$ at the time step $k$. The underlying model can be described by the following two-state Markov chain that


\begin{center}
	\begin{tikzpicture}[->, >=stealth', auto, semithick, node distance=3cm]
	\tikzstyle{every state}=[fill=white,draw=black,thick,text=black,scale=1]
	\node[state]    (A)                     {$A$};
	\node[state]    (B)[right of=A]   {$B$};
	\path
	(A) edge[loop left]			node{$p\sim U_{[0,1]}$}	(A)
	(A) edge[bend right,above]	node{$1-p$}	(B)
	(B) edge[loop right]		node{$1$}	(B);
	\end{tikzpicture}
\end{center}

 In the context of research and development (R$\&$D), our example here can be viewed as a toy mean-field model when the social planner such as the government or the company manager needs to decide the timing to announce the new product or technology on behalf of a large population of agents. Each agent at state A indicates that she is still in the developing mode while at state B means that the task is completed. The random variable $p$ depicts the randomness for the agent to complete the task, i.e., moving from state A to state B. Ideally, the social planner would hope to wait until all agents finish their task and make the announcement of the completion of the whole project. However, in real-life situations, the waiting cost might be very high due to well-observed severe competition in R$\&$D among large teams. Therefore, for the overall welfare of the large population, a strategic stopping policy by the social planner is necessary in the face of competition.  
 
 Let us consider the reward function $r(\mu)=1-\mu$. The transition rule of the mean-field MDP is given by $\mu_{k+1}=\mu_k U_{k+1}$, with $\{U_k\}_{k\geq 1}$ is a family of i.i.d. uniform random variables on $[0,1]$. In other words, we take $T_0(\mu,z)=\mu z$, $\cZ=[0,1]$, $\cL(\cZ)=U_{[0,1]}$. For some $\beta\in (0,1)$, $K\in (1,1/\beta)$, we consider the discount function defined as follows\footnote{\eqref{exmdeltadef} is a generalization of the well-known quasi-hyperbolic discount function; see \cite{Laibson1997}. In the quasi-hyperbolic discount function, it is required that $0<K<1$. In contrast, we only impose a minimal assumption, i.e., $\delta(1)=K\beta<1$, and allow $K>1$, accommodating the increasing impatience. In Remark \ref{rmkassumptionK}, we demonstrate that it is precisely this generalization that leads to the emergence of a (mixed-strategy) relaxed equilibrium. }:
\begin{equation}\label{exmdeltadef}
 \delta(k)=\left\{
    \begin{aligned}
     &1,  &k=0,\\
     &K\beta^k,&k\geq 1.
    \end{aligned}
    \right.
\end{equation}
\end{example}

We first prove that, in this example, relaxed equilibria must satisfy some specific forms.
\begin{definition}\label{thresholdtype}
    In Example \ref{nopure}, for a fixed strategy $\phi_0\in \F_S$, denote $\tilde{V}(\mu):= \mE^0\tilde{J}^{\phi_0}(T_0(\mu,Z^0)$. $\phi_0$ is said to be of "two-threshold" type if there exists $0\leq \mu_*\leq \mu^*\leq 1$ such that $\tilde{V}<r$ (hence $\phi_0=1$) on $[0,\mu_*]$, $\tilde{V}=r$ on $(\mu_*,\mu^*)$ and $\tilde{V}>r$ (hence $\phi_0=0$) on $[\mu_*,1]$. By setting the interval to be empty set if two endpoints coincide, we also embed the pure strategies (including ``always stopping" and ``always continuing" strategies) with one threshold into the type of two-threshold.
\end{definition}

\begin{proposition}\label{noclassicaleq}
In the context of Example \ref{nopure}, any relaxed equilibrium, if exists, must be of the two-threshold type.
\end{proposition}
\begin{proof}
 Denote $a:= \frac{1-K\beta}{1-K\beta/2}$ and $b:= \frac{1-\beta}{2-\beta}$. Suppose that $\phi_0$ is an equilibrium. Recall that,
 \begin{align*}
 \tilde{V}(\mu)&:= \mE^0\tilde{J}^{\phi_0}(T_0(\mu,Z^0))\\
 &=\sum_{k=1}^\infty\delta(k)\tilde{\mE}^\mu r(\mu_k)\phi_0(\mu_k)\prod_{j=1}^{k-1}(1-\phi_0(\mu_j))\\
 &=: K\bar{V}(\mu),
 \end{align*}
 where
 \begin{equation}\label{barVdef}
 \bar{V}(\mu):= \sum_{k=1}^\infty\beta^k\tilde{\mE}^\mu r(\mu_k)\phi_0(\mu_k)\prod_{j=1}^{k-1}(1-\phi_0(\mu_j)).
 \end{equation}
We note that, for any $\mu\in [0,1]$,
\begin{equation}\label{barViiteration}
\begin{aligned}
\bar{V}(\mu)&=\beta \tilde{\mE}^\mu r(\mu_1)\phi_0(\mu_1)+\sum_{k=2}^\infty\beta^k\tilde{\mE}^\mu\tilde{\mE}^{\mu_1}r(\mu_{k-1})\phi_0(\mu_{k-1})\prod_{j=0}^{k-2}(1-\phi_0(\mu_j))\\
&=\beta \mE^0 r(\mu U)\phi_0(\mu U) + \beta \sum_{k=1}^{\infty} \beta^k \mE^0 \tilde{\mE}^{\mu U}r(\mu_k)\phi_0(\mu_k)\prod_{j=0}^{k-1}(1-\phi_0(\mu_j))\\
&=\beta \mE^0 r(\mu U)\phi_0(\mu U) + \beta \mE^0 (1-\phi_0 (\mu U))\bar{V}(\mu U)\\
&=\frac{\beta}{\mu}\int_0^{\mu} [(1-\phi_0(\nu))\bar{V}(\nu)+\phi_0(\nu)r(\nu)]\md \nu.
\end{aligned}
\end{equation}
We first observe from \eqref{barViiteration} that $\bar{V}\in C(0,1)$, which will be used implicitly. We split the rest of the proof into three steps.

{\bf Step 1:} $\tilde{V}(\mu)=K\beta(1-\frac{\mu}{2})$ for $\mu\in [0,a]$. 

Noting that $\mu_k=0$ for any $k\in \mT$, $\tilde{\mP}^{0}$-a.s., we obtain from definition that $\tilde{V}(0)=\sum_{k=1}^\infty\delta(k)\phi_0(0)(1-\phi_0(0))^{k-1}$. If $\phi_0(0)=1$, $\tilde{V}(0)=\delta(1)=K\beta<1$. If $\phi_0(0)=0$, $\tilde{V}(0)=0$. If $0<\phi_0(0)<1$, $\tilde{V}(0)<\sum_{k=1}^{\infty}\phi_0(0)(1-\phi_0(0))^{k-1}=1$. In all cases, $\tilde{V}(0)<1=r(0)$. Therefore, for a sufficiently small $\mu$, we must have $\tilde{V}<r$. Because $\phi_0$ is an equilibrium, we have $\phi_0=1$ for sufficiently small $\mu$. Denote $\bar{a}:= \sup\{\bar{\mu}\in (0,1):\phi_0=1\mathrm{\ on\ } [0,\bar{\mu}]     \}$. We claim that $\bar{a}\geq a= \frac{1-K\beta}{1-K\beta/2}$. Otherwise, suppose $\bar{a}<a$. Because $\phi_0=1$ on $[0,\bar{a})$, by \eqref{barViiteration} we have for $\mu\in [0,\bar{a}]$, $\bar{V}(\mu)=\beta(1-\frac{\mu}{2})$. Consequently, $\tilde{V}(\bar{a})=K\beta(1-\frac{\bar{a}}{2})<1-\bar{a}=r(\bar{a})$. By continuity, we can choose a small $\eta>0$ such that $\tilde{V}<r$ on $[\bar{a},\bar{a}+\eta]$, which implies that $\phi_0=1$ on $[0,\bar{a}+\eta]$ by the definition of equilibrium. This is a contradiction to the definition of $\bar{a}$ and hence it holds that $\bar{a}\geq a$. As a result, $\bar{V}_0(\mu)=\beta(1-\frac{\mu}{2})$ for $\mu\in [0,a]$, and $\tilde{V}(a)=1-a$. See Figure \ref{proofdemo1} for the graphical demonstration of {\bf Step 1}.

{\bf Step 2:} $\tilde{V}(\mu)\leq r(\mu) $ for $\mu\in [0,b]$. 

If $b\leq a$, the claim follows from {\bf Step 1}. Suppose now $a<b$ and $\tilde{V}(\mu_0)>r(\mu_0)$ for some $\mu_0\in (a,b)$. By the continuity, $\tilde{V}>r$ holds in a small neighborhood of $\mu_0$. Consider $a_*:= \inf\{\bar{\mu}\in [a,\mu_0]: \tilde{V}>r \mathrm{\ on\ }[\bar{\mu},\mu_0]    \}$. Clearly, $\phi_0=0$ on $(a_*,\mu_0]$. Thus, for $\mu\in [a_*,\mu_0]$, we have
\[
\mu\bar{V}(\mu)=\mu_0\bar{V}(\mu_0)-\beta\int_\mu^{\mu_0} \bar{V}(\nu)\md \nu.
\]
 Solving this integral equation on $[a_*,\mu_0]$ gives $\bar{V}(\mu)=\bar{V}(\mu_0)\left(   \frac{\mu}{\mu_0}\right)^{\beta-1}$. Note that $\tilde{V}'(\mu)\leq \tilde{V}'(\mu_0)=\frac{\tilde{V}(\mu_0)(\beta-1)}{\mu_0}$. Because $\tilde{V}(\mu_0)>r(\mu_0)$ and $\mu_0<b$, we have $\tilde{V}'(\mu)<\left(\frac{1}{b}-1\right)(\beta-1)=-1$. Therefore, by the convexity of $\tilde{V}$, $\tilde{V}(a_*)>r(a_*)$. We thus conclude $a_*=a$, or otherwise, by the continuity, we are led to a contradiction to the definition of $a_*$. However, we now have $\tilde{V}(a)>r(a)$, contradicting to the fact that $\tilde{V}(a)=r(a)$. See Figure \ref{proofdemo2} for the graphical demonstration of {\bf Step 2}.

{\bf Step 3:} $\tilde{V}(\mu)\geq r(\mu) $ for $\mu\in (a,1)$. 

Suppose $\tilde{V}(\mu_1)<r(\mu_1)$ for $\mu_1\in (a,1)$. We consider $a_*:= \inf\{\bar{\mu}\in [a,\mu_1]: \tilde{V}>r \mathrm{\ on\ }[\bar{\mu},\mu_1]    \}$ so that $\tilde{V}(\mu)<r(\mu)$ on $(a_{**},\mu_1]$. By the equilibrium condition, $\phi_0=1$ on $(a_{**},\mu_1]$. Thus, for $\mu\in [a_{**},\mu_1]$,
\[
\mu\bar{V}(\mu)=\mu_1\bar{V}(\mu_1)+\beta \int_{\mu_1}^\mu(1-\nu)\md \nu.
\]
Consequently, for $\mu\in (a_*,\mu_1]$
\[
\tilde{V}(\mu)=K\beta+\frac{\mu_1\tilde{V}(\mu_1)}{\mu}-\frac{K\beta \mu}{2}-\mu_1\left( 1-\frac{\mu_1}{2}  \right)\frac{K\beta}{\mu},
\]
and
\begin{align*}
\tilde{V}'(\mu)&=-\frac{\mu_1\tilde{V}(\mu_1)}{\mu^2}-\frac{\beta K}{2}+\frac{K\beta\mu_1\left(1-\frac{\mu_1}{2}\right)}{\mu^2}\\
&> \frac{-\mu_1(1-\mu_1)+K\beta \mu_1 \left(1-\frac{\mu_1}{2}\right)}{\mu^2}-\frac{ K\beta}{2}\\
&=\frac{\left( 1-\frac{K\beta}{2}\right)\mu_1(\mu_1-a)  }{\mu^2}-\frac{K\beta}{2}\\
&>-\frac{K\beta}{2}>-1.
\end{align*}
Similarly to {\bf Step 2}, we can show that $a_{**}=a$ and $\tilde{V}(a_{**})<r(a_{**})$, which is a contradiction.
 
{\bf Step 4:} Conclusion. 

Define $\mu^*=\inf\{\mu \in (a,1): \tilde{V}>r  {\rm\ on\ } [\mu^*,1]\}$, and we suppose $\mu^*=1$ if the right hand side is empty. By {\bf Step 1} and {\bf Step 2}, $\mu^*\geq c:= a\vee b$. If $\mu^*=c$, then the proof is completed because $\tilde{V}<r$ on $[0,a)$, $\tilde{V}=r$ on $(a,c)$ (empty if $b\leq a$), and $\tilde{V}>r$ on $(\mu^*,1]$. If $\mu^*>c$, we claim that $\tilde{V}=r$ on $[c,\mu^*]$. Otherwise, {\bf Step 3} implies that there exists $\mu_2\in (c, \mu^*)$, such that $\tilde{V}(\mu_2)>r(\mu_2)$. Consider $c_*=\inf \{ \mu\in (c,\mu_2): \tilde{V}> r {\rm\ on\ } (\mu,\mu_2)   \}\geq c$. By the continuity of $\tilde{V}$, we have $\tilde{V}(c_*)=r(c_*)$. Using computations in {\bf Step 2} yields that $\tilde{V}'(c_*+)=\frac{(1-c_*)(\beta-1)}{c_*}\leq \left(\frac{1}{b}-1\right)(\beta-1)=-1$, and $\tilde{V}(\mu)=\tilde{V}(\mu_2)\left(\mu/\mu_2  \right)^{\beta-1}$ for $\mu\in (c_*,\mu_2)$. As a result, $\tilde{V}'(\mu_2)>-1$ and by convexity, $\tilde{V}(\mu)=\tilde{V}(\mu_2)\left(\mu/\mu_2  \right)^{\beta-1}$ for $\mu\in (c_*,1]$. This contradicts the definition of $\mu^*$ because $c_*<\mu^*$. See Figure \ref{proofdemo4} for the graphical demonstration of {\bf Step 4}. In conclusion, we have shown that $\tilde{V}=r$ on $[c,\mu^*]$ and by the definition of $\mu^*$, $\tilde{V}>r$ on $(\mu^*,1]$. Therefore, it is verified that the equilibrium $\phi_0$ must satisfy the two-thresholds type. 
\end{proof}

\begin{figure}
     \centering
     \begin{subfigure}[b]{0.3\textwidth}
         \centering
        \resizebox{1\textwidth}{!}{%
\begin{circuitikz}
\tikzstyle{every node}=[font=\normalsize]
\draw [ line width=0.5pt, -Stealth] (7.75,10) -- (15.5,10);
\draw [ line width=0.5pt, -Stealth] (7.75,10) -- (7.75,16.5);
\draw [line width=0.5pt, short] (7.75,15) .. controls (10.5,12.5) and (10.5,12.5) .. (13.25,10);
\draw [line width=0.5pt, dashed] (7.75,13.5) -- (10.75,12.25);
\draw [line width=0.2pt, dashed] (9,13) -- (9,10);
\draw [line width=0.2pt, dashed] (10.75,12.25) -- (10.75,10);
\node [font=\large] at (10.75,9.75) {$a$};
\node [font=\large] at (8.75,9.75) {$\bar{a}$};
\node [font=\large] at (13.25,9.75) {1};
\node [font=\large] at (7.5,15) {1};
\node [font=\normalsize] at (7.25,13.5) {$K\beta$};
\draw [line width=0.2pt, dashed] (9.5,12.75) -- (9.5,10);
\node [font=\large] at (9.75,9.75) {$\bar{a}+\eta$};
\node [font=\normalsize] at (10.75,15.5) {$\tilde{V}(\mu)=K\beta\left(1-\frac{1}{2}\mu\right)$};
\draw [ line width=0.2pt, -Stealth] (10.75,15.25) -- (8.5,13.25);
\node [font=\normalsize] at (7.5,9.75) {0};
\end{circuitikz}
}%
         \caption{The proof of $\tilde{V}(\mu)=K\beta\left(1-\frac{\mu}{2} \right)$ for $\mu\in [0,a]$.}
         \label{proofdemo1}
     \end{subfigure}
     \hfill
\begin{subfigure}[b]{0.3\textwidth}
         \centering
         \resizebox{1\textwidth}{!}{%
\begin{circuitikz}
\tikzstyle{every node}=[font=\large]
\draw [ line width=0.5pt, -Stealth] (7.75,10) -- (15.5,10);
\draw [ line width=0.5pt, -Stealth] (7.75,10) -- (7.75,16.5);
\draw [line width=0.5pt, short] (7.75,15) .. controls (10.5,12.5) and (10.5,12.5) .. (13.25,10);
\draw [line width=0.5pt, dashed] (7.75,13.5) -- (10.75,12.25);
\draw [line width=0.2pt, dashed] (9,13) -- (9,10);
\draw [line width=0.2pt, dashed] (10.75,14.5) -- (10.75,10);
\node [font=\large] at (10.75,9.75) {$a$};
\node [font=\large] at (8.75,9.75) {$\bar{a}$};
\node [font=\large] at (13.25,9.75) {1};
\node [font=\large] at (7.5,15) {1};
\node [font=\normalsize] at (7.25,13.5) {$K\beta$};
\draw [line width=0.2pt, dashed] (9.5,12.75) -- (9.5,10);
\node [font=\large] at (9.75,9.75) {$\bar{a}+\eta$};
\node [font=\normalsize] at (7.5,9.75) {0};
\draw [line width=0.2pt, dashed] (11.5,10) -- (11.5,12.25);
\draw [line width=0.2pt, dashed] (12,10) -- (12,11.75);
\node [font=\large] at (12,9.75) {$b$};
\node [font=\large] at (11.5,9.75) {$\mu_0$};
\draw [line width=0.5pt, short] (12,11.75) .. controls (10.75,13.25) and (10.75,13.25) .. (10.75,14.5);
\node [font=\large] at (14,14.25) {$\tilde{V}(\mu)=\tilde{V}(\mu_0)\left( \frac{\mu}{\mu_0}\right)^{\beta-1}$};
\draw [line width=0.2pt, dashed] (11,10) -- (11,13);
\node [font=\large] at (11,9.55) {$a_*$};
\end{circuitikz}
}
         \caption{The proof of $\tilde{V}\leq r$ on $(a,b)$.}
         \label{proofdemo2}
     \end{subfigure}
\hfill
  \begin{subfigure}[b]{0.3\textwidth}
         \centering
         \resizebox{1\textwidth}{!}{%
\begin{tikzpicture}[x=0.75pt,y=0.75pt,yscale=-1,xscale=1]

\draw    (112,253) -- (111.01,39) ;
\draw [shift={(111,36)}, rotate = 89.74] [fill={rgb, 255:red, 0; green, 0; blue, 0 }  ][line width=0.08]  [draw opacity=0] (10.72,-5.15) -- (0,0) -- (10.72,5.15) -- (7.12,0) -- cycle    ;
\draw    (112,253) -- (400,253.99) ;
\draw [shift={(403,254)}, rotate = 180.2] [fill={rgb, 255:red, 0; green, 0; blue, 0 }  ][line width=0.08]  [draw opacity=0] (10.72,-5.15) -- (0,0) -- (10.72,5.15) -- (7.12,0) -- cycle    ;
\draw    (111,92) -- (343,254) ;
\draw  [dash pattern={on 0.84pt off 2.51pt}]  (164,129) -- (165,253) ;
\draw  [dash pattern={on 0.84pt off 2.51pt}]  (215,164) -- (215,257) ;
\draw  [dash pattern={on 0.84pt off 2.51pt}]  (312,232) -- (312,253) ;
\draw  [dash pattern={on 0.84pt off 2.51pt}]  (276,207) -- (276,256) ;
\draw  [dash pattern={on 0.84pt off 2.51pt}]  (243,185) -- (243,254) ;
\draw  [dash pattern={on 4.5pt off 4.5pt}]  (243,185) .. controls (320,225) and (322,227) .. (343,226) ;
\draw  [dash pattern={on 0.84pt off 2.51pt}]  (343,226) -- (343,254) ;
\draw    (356,153) -- (303.47,201.64) ;
\draw [shift={(302,203)}, rotate = 317.2] [color={rgb, 255:red, 0; green, 0; blue, 0 }  ][line width=0.75]    (10.93,-3.29) .. controls (6.95,-1.4) and (3.31,-0.3) .. (0,0) .. controls (3.31,0.3) and (6.95,1.4) .. (10.93,3.29)   ;

\draw (100,252) node [anchor=north west][inner sep=0.75pt]   [align=left] {0};
\draw (159,251) node [anchor=north west][inner sep=0.75pt]   [align=left] {$\displaystyle a$};
\draw (212,251) node [anchor=north west][inner sep=0.75pt]   [align=left] {$\displaystyle c$};
\draw (234,253) node [anchor=north west][inner sep=0.75pt]   [align=left] {$\displaystyle c_{*}$};
\draw (266,254) node [anchor=north west][inner sep=0.75pt]   [align=left] {$\displaystyle \mu _{2}$};
\draw (122,127) node [anchor=north west][inner sep=0.75pt]   [align=left] {\textsuperscript{}};
\draw (300,255) node [anchor=north west][inner sep=0.75pt]   [align=left] {$\displaystyle \mu ^{*}$};
\draw (339,256) node [anchor=north west][inner sep=0.75pt]   [align=left] {1};
\draw (299,111.4) node [anchor=north west][inner sep=0.75pt]    {$\tilde{V}( \mu ) =\tilde{V}( \mu _{2})\left(\frac{\mu }{\mu _{2}}\right)^{\beta -1}$};
\end{tikzpicture}
}
         \caption{The proof of {\bf Step 4}.}
         \label{proofdemo4}
     \end{subfigure}
     \caption{}
\end{figure}
We next provide a characterization of the unique (mixed strategy) relaxed equilibrium. 

\begin{proposition}\label{mixequi}
    Suppose $K\in \left(\frac{2}{(3-\beta)\beta},\frac{1}{\beta}\right)$. Denote $a:= \frac{1-K\beta}{1-K\beta/2}$, $b:= \frac{1-\beta}{2-\beta}$. Then, there exists a unique relaxed equilibrium in Example \ref{nopure}, which is given by 
    \begin{equation}\label{relaxedeq}
    \phi_0(\mu)=\left\{
    \begin{aligned}
    &1, &0\leq \mu\leq a,\\
    &\frac{1-\beta-(2-\beta)\mu}{\beta(K-1)(1-\mu)}, &a<\mu\leq b,\\
    &0, & b<\mu\leq 1.
    \end{aligned}
    \right.
    \end{equation}
\end{proposition}
\begin{proof}
    By direct computations, $K>\frac{2}{(3-\beta)\beta}$ implies $a<b$. We now show that any equilibrium of the two-thresholds type can be uniquely constructed. Denote by $\mu_*$ and $\mu^*$ two thresholds (see Definition \ref{thresholdtype}). From the proof of Proposition \ref{noclassicaleq}, we know $\mu_*=a$. Because $\tilde{V}(\mu)=r(\mu)=1-\mu$ for $\mu\in (a,\mu^*)$, we derive from \eqref{barViiteration} that, for $\mu\in (a,\mu^*)$ 
    \[
    \frac{\mu(1-\mu)}{K}=\beta\left(a-\frac{a^2}{2}\right)+\beta\int_a^\mu \left[\frac{1-\phi_0(\nu)}{K}+\phi_0(\nu)\right](1-\nu)\md \nu,
    \]
    which is equivalent to (by taking derivative or integrating on both sides) 
    \[
    \phi_0(\mu)=\frac{1-\beta-(2-\beta)\mu}{\beta(K-1)(1-\mu)}.
    \]
    It is not hard to show that $\phi_0$ is decreasing on $(a,\mu^*)$ and $\phi_0(a+)=\frac{(1-\beta)K\beta-2+2K\beta}{K\beta^2(K-1)}<1$ in view that $K\beta<1$. Thus, to ensure $0<\phi_0<1$ on $(a,\mu^*)$, it is sufficient and necessary to require $\phi_0(\mu^*-)\geq 0$, which is equivalent to $\mu^*\leq \frac{1-\beta}{2-\beta}=b$. On $(\mu^*,1]$, we use \eqref{barVdef} again to obtain $\tilde{V}(\mu)=(1-\mu^*)\left(\mu/\mu^*  \right)^{\beta-1}$. Due to the strict convexity of $\tilde{V}$, to satisfy $\tilde{V}>r$ on $(\mu^*,1]$, it is sufficient and necessary to require that $\tilde{V}'(\mu^*-)\geq -1$, which gives $\mu^*\geq \frac{1-\beta}{2-\beta}=b$. Hence, it follows that $\mu^*=b$. To conclude, \eqref{relaxedeq} is the unique equilibrium of the two-thresholds type. Moreover, using Lemma \ref{noclassicaleq}, it is the unique equilibrium even without confining the two-thresholds type.
\end{proof}

\begin{remark}\label{rmkassumptionK}
    From Propositions \ref{noclassicaleq} and \ref{mixequi}, when $0<K\leq\frac{2}{(3-\beta)\beta}$, there exists a unique pure strategy equilibrium (i.e., the equilibrium such that $\phi_0(\mu)\in \{0,1\}$ for any $\mu\in [0,1]$) with stopping region $[0,a]$. In particular, if we are considering traditional quasi-hyperbolic discount ($0<K<1$), we obtain a unique pure strategy equilibrium. We also remark that, because $\beta$ is very close to 1 in reality, the assumption $K>\frac{2}{(3-\beta)\beta}$ is mild. For example, in an empirical study, \cite{AS12} estimates $K$ in the range 1.004$\sim$1.028, and the annual discount rate is approximately 0.3, i.e., $\beta \approx 1/(1.3)^{1/365}\approx 0.9993$. Thus, $\frac{2}{(3-\beta)\beta}\approx 1.0004$ and $K>\frac{2}{(3-\beta)\beta}$ is certainly satisfied.
\end{remark}

\begin{remark}\label{rmk:nonexit-pure}
Because the relaxed equilibrium is unique in Example \ref{nopure}, it does not admit any pure strategy equilibrium if assumptions in Proposition \ref{mixequi} hold. In particular, this illustrative example shows that the mild pure strategy equilibrium given by the iteration approach in \cite{Huang2019} does not exist in this model setting. The (mixed strategy) relaxed equilibrium per Definition \ref{originalequilibrium}, on the other hand, exists in general.
\end{remark}

\begin{figure}[t]
     \centering
     \begin{subfigure}[b]{0.45\textwidth}
         \centering
       \includegraphics[width=1\textwidth]{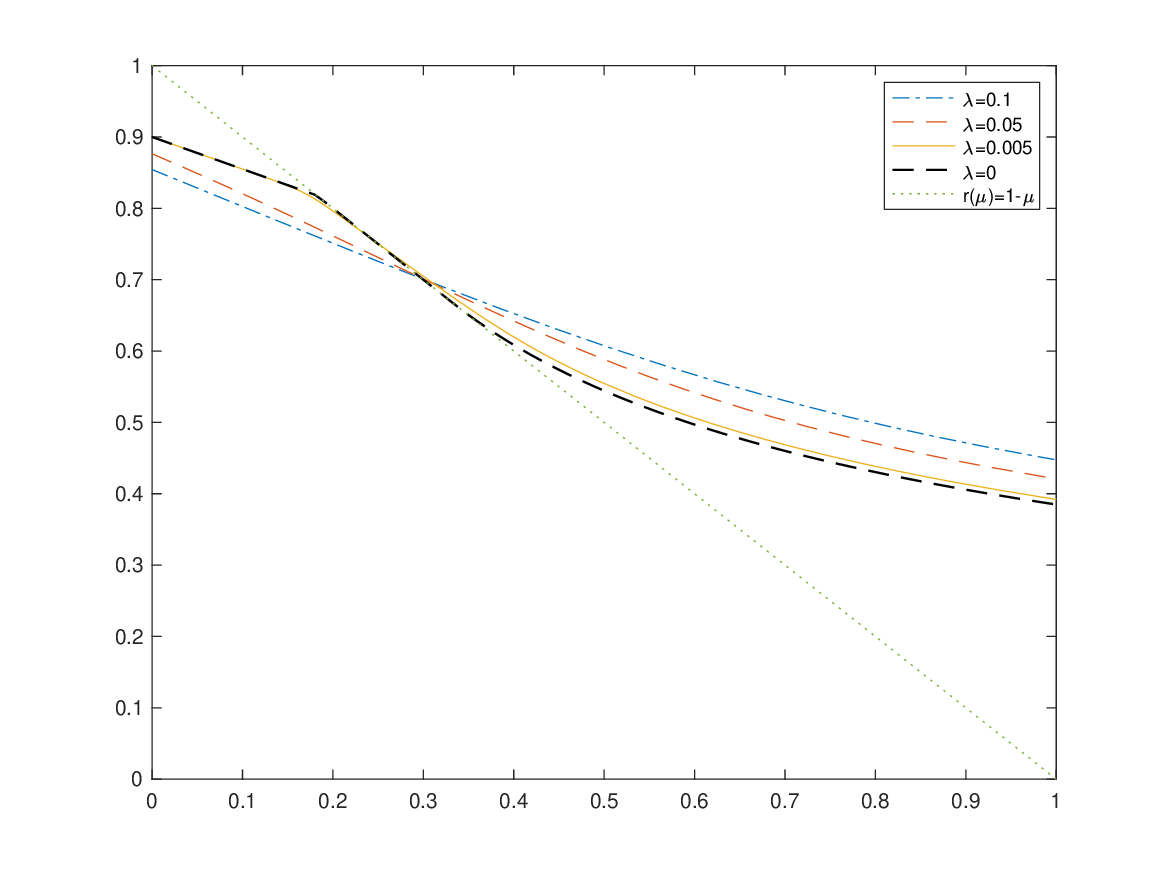}
         \caption{Graphs of functions $\tilde{V}_\lambda$ and $r$.}
         \label{solutiondemovalue}
     \end{subfigure}
     \hfill
     \begin{subfigure}[b]{0.45\textwidth}
         \centering
         \includegraphics[width=1\textwidth]{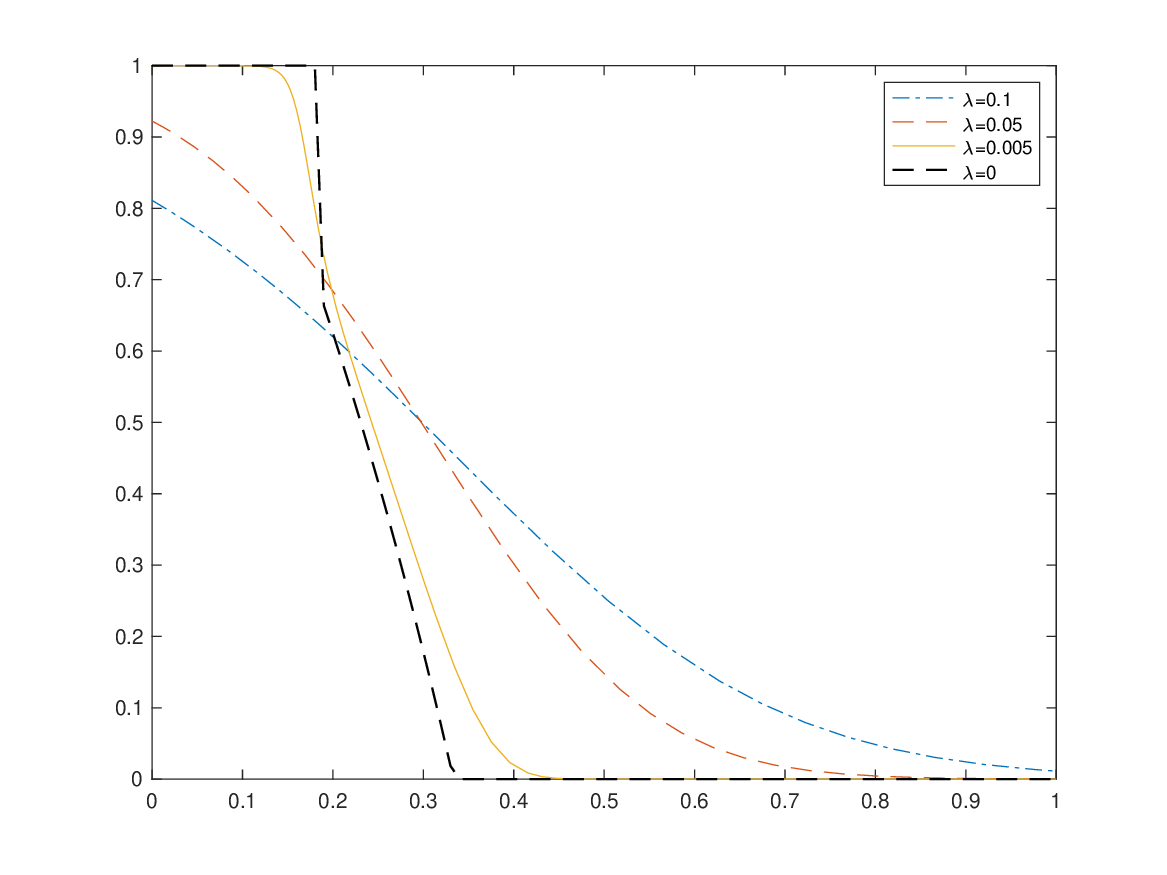}
         \caption{Graphs of equilibrium stopping strategies $\phi_\lambda$.}
         \label{solutiondemostrategy}
     \end{subfigure}
     \caption{}
\end{figure}

We next present some numerical plots to illustrate the regularized equilibrium with the parameter $\lambda>0$ and show the convergence to the unique relaxed equilibrium in \eqref{mixequi} as $\lambda\rightarrow 0$. For $\lambda>0$, let $\phi_\lambda$ denote a regularized equilibrium and $\tilde{V}_\lambda(\mu):= \mE^0\tilde{J}^{\phi_\lambda}_\lambda(T_0(\mu,Z^0))$. Recall the entropy term $\E(\phi)=-\phi \log \phi -(1-\phi)\log (1-\phi)$ for $\phi\in (0,1)$. Here we further introduce the notation $\Phi_\lambda(\mu,v):= \frac{1}{1+\exp\left(\frac{v-1+\mu}{\lambda}   \right)}$. It then follows that $\phi_\lambda(\mu)=\Phi(\mu,\tilde{V}_\lambda(\mu))$, and similarly to \eqref{barVdef} and \eqref{barViiteration}, we have
\begin{align*}
    \bar{V}_\lambda(\mu)&:= \sum_{k=1}^{\infty}\beta^k \tilde{\mE}^\mu\left[r(\mu_k)\phi_\lambda(\mu_k)+\lambda\E(\phi_\lambda(\mu_k))\right]\prod_{j=0}^{k-1}(1-\phi_\lambda(\mu_j))\\
    &=\frac{\beta}{\mu}\int_0^\mu [(1-\phi_\lambda(\nu))\bar{V}_\lambda(\nu)+\phi_\lambda(\nu)r(\nu)+\lambda \E(\phi_\lambda(\nu))]\md \nu\\
    &=\frac{\beta}{\mu}\int_0^\mu\left[(1-\Phi_\lambda(\nu,\tilde{V}_\lambda(\nu)))\bar{V}_\lambda(\nu) +\Phi_\lambda(\nu,\tilde{V}_\lambda(\nu))(1-\nu)+\lambda \E(\Phi_\lambda(\nu,\tilde{V}_\lambda(\nu)))   \right]\md \nu,
\end{align*}
for any $\mu\in (0,1]$, and $\tilde{V}(\mu)=K\bar{V}(\mu)$. Taking $\mu\to 0$ gives
\begin{equation}\label{regularizedinitial}
\tilde{V}_\lambda(0)=\beta\left[ (1-\Phi_\lambda(0,\tilde{V}_\lambda(0)))\tilde{V}_\lambda(0)+K\Phi_\lambda(0,\tilde{V}_\lambda)+K\lambda\E(\Phi_\lambda(0,\tilde{V}_\lambda(0))   \right].
\end{equation}
As a result,  $\phi_\lambda$ is a regularized equilibrium if and only if $\tilde{V}_\lambda$ solves the following ODE:
\begin{equation}\label{regularizedODE}
   \tilde{V}_\lambda'=\frac{1}{\mu}\left[(\beta-1)\tilde{V}_\lambda-\beta \Phi_\lambda(\mu,\tilde{V}_\lambda)\tilde{V}_\lambda+K\beta \Phi_\lambda(\mu,\tilde{V}_\lambda)(1-\mu)+\lambda K\beta\E(\Phi_\lambda(\mu,\tilde{V}_\lambda)) \right],
\end{equation}
with the initial condition given by \eqref{regularizedinitial}.

Note that \eqref{regularizedODE} and \eqref{regularizedinitial} can be solved numerically, and $\phi_\lambda$ is then obtained via $\phi_\lambda(\mu)=\Phi_\lambda(\mu,\tilde{V}_\lambda(\mu))$. In our numerical experiment, we take $\beta=0.5$, $K=1.8$ and we plot the graphs of $\phi_\lambda$ and $\tilde{V}_\lambda$ for $\lambda=0.1,\ 0.05,\ 0.01$, respectively, in Figures \ref{solutiondemovalue} and \ref{solutiondemostrategy}. We also plot the graph of the relaxed equilibrium without entropy regularization obtained in Proposition \ref{mixequi} ($\lambda=0$ in Figures \ref{solutiondemovalue} and \ref{solutiondemostrategy}). We can clearly observe that the regularized equilibrium can be a good approximation of the relaxed equilibrium of the original problem as $\lambda\to 0$.

\subsection{  An Example without Explicit Solution: Exercise of American ETF Put}\label{exm:ETF:option}

{  

During the past decades, the size of exchange-traded funds (ETFs) has grown to USD 7.1 trillion in the U.S. in terms of assets under management (AUM)\footnote{\href{https://www.ishares.com/us/insights/global-etf-facts-q1-2024}{https://www.ishares.com/us/insights/global-etf-facts-q1-2024}.}. Most ETFs in the market belong to the type of passive index tracking funds, which use replications to track the performance of some indices, e.g., S\&P 500 or Dow Jones Industrial Average. It has also been seen surging popularity in trading ETF {\it options}, a type of derivatives that are written directly on ETFs. For example, ETF options accounted for 41.5\% of all traded options on NYSE in June 2022\footnote{\href{https://www.nyse.com/data-insights/etf-options-break-new-records}{https://www.nyse.com/data-insights/etf-options-break-new-records}.}. 

{  Mean-field model with common noise is suitable for the study of ETF options, which provides a tractable and satisfactory approximation in analyzing the hedging of {\it systematic risks} with large but finite number of underlying stocks. Just as individual randomnesses can be averaged out in the mean-field model, the idiosyncratic risks are assumed to be eliminated through diversification in the context of finance. In particular, exercising an American ETF put aligns with our formulation of the centralized mean-field optimal stopping because the option holder cannot modify the index construction or trade individual stocks. Additionally, this leads to the feature that all ``agents" (stocks in this context) will be terminated simultaneously, and partial stopping is not permitted.

}

	{  To illustrate the application of our results to ETF options}, we start from the price dynamics of each component stock of an index. By using the previous relationship between ({\bf N-MDP}) and ({\bf MF-MDP}), we derive dynamics of ETF prices in the limiting model when the number of stocks goes to infinity, leading to reduced computation complexity. Further, using our regularized approach, we devise a model-free algorithm to obtain one approximated equilibrium stopping policy under a quasi-hyperbolic discount. 
	
	For an index with $N$ constituents, suppose that the price process of each individual stock, $i\in [N]$, is described by
	\begin{equation}\label{exm:stockprice}
	P^i_{t+1}=P^i_t(1+r^i_{t+1}),
	\end{equation}
	where $\{r^i_{t}\}_{t\in \mT}$ are dynamic returns of stock $i$. For simplicity, it is assumed that
	\begin{equation}\label{exm:return}
	r^i_{t+1} = \theta^i_t +\sigma^i_t \epsilon^i_{t+1}+\sigma^{0i}_t\epsilon^0_{t+1}.
	\end{equation}
	Here $\{\epsilon^i_t\}_{i\in [N],t\in \mT}$ and $\{\epsilon^0_t\}_{t\in \mT}$ are normalized idiosyncratic and common noises, respectively, both with mean 0 and variance 1. To model the mean-field interaction, we assume that
	\begin{align}\label{exm:return:model}
	&\theta^i_t = f^i(t,\bar r^N_t,\bar v^N_t),\quad \sigma^{0i}_t = g^i(t, \bar r^N_t,\bar v^N_t),\quad \sigma^i_t = h^i(t, \bar r^N_t,\bar v^N_t)\\
	 &\bar r^N_t :=\frac{1}{N}\sum_i r^i_t,\quad \bar v^N_t:=\frac{1}{N}\sum_i |r^i_t|^2.
	\end{align}
 Here, $f^i$, $g^i$, and $h^i$ are sampled from some ``type distribution", independently from all other randomness in the model. We denote by $\bar f$, $\bar g$, $\bar h$ their expectations, and $\overline{f^2}$, $\overline{g^2}$, $\overline{h^2}$ their second moments. Note that the type distributions are used to model heterogeneity among different stocks or sectors. We assume for simplicity that $f^i$, $g^i$, and $h^i$ are also independent, hence when considering ETFs or indices, the heterogeneity is aggregated and only their moments matter. This example can also be generalized to the case where $f^i$, $g^i$, and $h^i$ are correlated.
 
  We use $\bar r^N_t$, the average return rates of all stocks, and $\bar v^N_t$, the average of squared return rates, to model the overall performance of the market. Consequently, \eqref{exm:return:model} implies that individual stock returns exhibit distinct {\it alphas} and {\it betas} across various market conditions. {  Alphas and betas are fundamental concepts in finance, originating from CAPM (\cite{CAPM64}), and are regarded as pivotal factors in asset pricing, portfolio selection, and risk management. Furthermore, there are renowned variations of CAPM that permit different alphas and betas in different time periods and market states, such as conditional CAPM (\cite{JW96}). This motivates us to incorporate mean-field interaction in \eqref{exm:return:model} to better capture the mutual influence.   
  
 }
 
With \eqref{exm:stockprice} and \eqref{exm:return} in mind, we use the average price $\bar P^N_t:=\frac{1}{N}\sum_i P^i_t$ to represent the ETF price or the index level, and aim to study the optimal exercising time of an American ETF put with strike $P_0>0$. That is, we want to solve
\begin{equation}\label{exm:finiteproblem}
	v(\boldsymbol{p},\boldsymbol{r})=\sup_\tau \mE^{\boldsymbol{p},\boldsymbol{r}} \delta(\tau)\Big(P_0 - \bar P^N_\tau\Big)_+,
\end{equation}
with observations of prices $\boldsymbol{p}=(p_1,\cdots,p_N)$ and returns $\boldsymbol{r}=(r_1,\cdots,r_N)$.

We now formally derive the mean-field limit. We emphasize that due to the dependence of $\bar v^N_t$ in \eqref{exm:return:model}, we need the distribution of squared return in addition to the average return, which is not available from the market (macro) data (e.g., index data). Therefore, using the connection between ({\bf N-MDP}) and ({\bf MF-MDP})  in Section \ref{finiteconvergencesection}, we can focus on some reduced {\bf single-agent} problems with less computation complexity.

 Denote 
\begin{equation}
	\bar P_t:= \lim_{N\to \infty}\bar P^N_t,\quad \bar r_t:=\lim_{N\to \infty} \bar r^N_t,\quad \bar v_t =\lim_{N\to \infty} \bar v^N_t.
\end{equation}
 We obtain from \eqref{exm:stockprice} that
\begin{equation}\label{exm:return:square}
	|r^i_{t+1}|^2=|\theta^i_t|^2+|\sigma^i_t|^2|\epsilon^i_{t+1}|^2+|\sigma^{0i}_t|^2|\epsilon^0_{t+1}|^2+2\theta^i_t\sigma^i_t\epsilon^i_{t+1}+2\theta^i_t\sigma^{0i}_t\epsilon^0_{t+1}+2\sigma^i_t\sigma^{0i}_t\epsilon^i_{t+1}\epsilon^0_{t+1}
\end{equation}

Taking average on both sides of \eqref{exm:stockprice}, and using the law of large numbers, we have
\begin{align}
	\frac{1}{N}\sum_i P^i_{t+1}=&\frac{1}{N}\sum_i P^i_t+\frac{1}{N}P^i_t \Big(\theta^i_t +\sigma^i_t \epsilon^i_{t+1}+\sigma^{0i}_t\epsilon^0_{t+1}\Big)\\
	\to &\bar P_t \Big(1+\bar f(t,\bar r_t,\bar v_t)+\bar g(t,\bar r_t,\bar v_t)\epsilon^0_{t+1} \Big)\\
	=&\bar P_t (1+\bar r_{t+1}).
\end{align}		
Here, we take the average on both sides of \eqref{exm:return} and \eqref{exm:return:square}, and then take the limit as $N\to \infty$ to obtain 
\begin{align}
	&\bar r_{t+1} =\bar f(t,\bar r_t,\bar v_t)+\bar g(t,\bar r_t,\bar v_t)\epsilon^0_{t+1},\quad \text{and} \label{exm:marketreturn}\\
	&\bar v_{t+1} = \overline{f^2}(t,\bar r_t,\bar v_t)+\overline{h^2}(t,\bar r_t,\bar v_t)+\overline{ g^2}(t,\bar r_t,\bar v_t)|\epsilon^0_{t+1}|^2+2\bar f(t,\bar r_t,\bar v_t)\bar g (t,\bar r_t,\bar v_t)\epsilon^0_{t+1}.\label{exm:average_sq_return}
\end{align}
Therefore, using the approximation results and the relationship between ({\bf N-MDP}) and ({\bf MF-MDP}) in Section \ref{finiteconvergencesection}, we can focus on the (approximated) market return \eqref{exm:marketreturn} and ETF price $\{\bar P_t\}_{t=0}^\infty$  satisfying
\begin{equation}\label{exm:index}
	\bar P_{t+1}=\bar P_t(1+\bar r_{t+1}).
\end{equation}

We next work on the numerical implementation to find the regularized equilibrium. Let us take the state space $\bar S=[0,K]\times [-1,1]$ in the mean-field model, and consider the state dynamics satisfying \eqref{exm:marketreturn}, \eqref{exm:average_sq_return} and \eqref{exm:index}, and the reward function $R(p,r) = (P_0-p)_+$\footnote{In this example, $r$ has been used to denote ``return", hence we use $R$ for reward function, as opposed to $r$ in the main context.}. For simplicity, we choose the same discount function as in Subsection \ref{Illexample} that
\begin{equation}\label{exm:discount}
	\delta(k)=\left\{
    \begin{aligned}
     &1,  &k=0,\\
     &K\beta^k,&k\geq 1.
    \end{aligned}
    \right.
\end{equation}

To ease the presentation, in what follows, we assume that 
\begin{align}\label{exm:return:linear}
	f^i(t,r,v) = a^i+b^i r,\quad g^i(t,r,v) = c^i+d^i r.
\end{align}
Here, for $u\in \{a,b,c,d\}$, $u^j$ are sampled from some distribution with mean $\bar u$. {  Because $f$ and $g$ do not depend on $v$ in \eqref{exm:return:linear}, the form of $h$, only appearing in the dynamics of $v$ \eqref{exm:average_sq_return}, becomes irrelavent.} Moreover, we simply write $P_k$, $r_k$ instead of $\bar P_k$, $\bar r_k$. 

\begin{algorithm}
\caption{Policy Iteration Algorithm for ETF put}

\textbf{Inputs}: a neural network $\{\hat F^\theta:\bar S\to \mR\}_{\theta\in \Theta}$, parameters for discount function $K$ and $\beta$ in \eqref{exm:discount}, regularization parameter $\lambda$, number of iterations $L$.

\textbf{Initialization}: $\theta_0$, $\phi^0: = \frac{1}{1+\exp\Big(\frac{1}{\lambda}(K\hat F^{\theta_0}-R  ) \Big)}$.

\begin{algorithmic}[1]
\FOR{Interation $l=0$ \TO $L-1$}

\STATE{
Use Algorithm \ref{exm:algo:neuralTD} to update $\theta_{l+1}$ such that:  
\begin{align}\label{exm:eval}
\hat F^{\theta_{l+1}}(p,r) \approx  \sum_{k=1}^\infty \beta^k \mE^{p,r,0\oplus_1 \phi^l}\Big[(P_0- P_k)_+\phi^l(P_k,r_k)+\lambda \Epsilon(\phi^l( P_k, r_k))\Big].
\end{align}
}

\STATE{
Update $\phi^{l+1} \leftarrow  \frac{1}{1+\exp\Big(\frac{1}{\lambda}(K\hat F^{\theta_{l+1}}-R )  \Big)}$.
}

\ENDFOR
\end{algorithmic}

\textbf{Output}: A policy $\phi^{L}$.

\label{exm:algo:policyiteration}
\end{algorithm}

\begin{algorithm}
\caption{Neural TD Algorithm for Policy Evaluation}

\textbf{Inputs}: simulator $\M$ of the model \eqref{exm:marketreturn}, \eqref{exm:average_sq_return} and \eqref{exm:index}, a neural network $\{\hat F^\theta:\bar S\to \mR\}_{\theta\in \Theta}$ for evaluation, a neural network $\{\hat G^\psi:\bar S\to \mR\}_{\psi \in \Psi}$ for approximating the conditional expectation, parameters for discount function $K$ and $\beta$ in \eqref{exm:discount}, regularization parameter $\lambda$, batch size $M$, length of maximal evaluated horizon $T_m$, the learning rate $\eta$, a fixed policy $\phi$.

\textbf{Initialization}: $\theta_{0}$, $P_{0,0}$, $r_{0,0}$, 

\begin{algorithmic}[1]

\STATE{
Sample a batch of paths of ETF prices and returns $\{(P_{k,j},r_{k,j})\}_{0\leq k\leq T_m-1,0\leq j\leq M-1}$.
}

\STATE{
Compute the rewards (with entropy term) and the stopping decision along the sampled path:
\begin{align}
&\phi_{k,j} = \phi(P_{k,j},r_{k,j}),\\
&Rewards_{k,j}= (P_0-P_{k,j})_+\phi_{k,j} + \lambda \E(\phi_{k,j})
\end{align}
}

\FOR{Timestep $k=0$ \TO $T_m-2$}

\STATE{
Compute TD loss: 
\begin{align}
TDloss = \frac{1}{M}\sum_{j=1}^M |\hat F^{\theta}(P_{k,j},r_{k,j})-\beta \hat G^\psi(P_{k,j},r_{k,j})|^2.
\end{align}
}

\STATE{
Compute CE loss:
\begin{align}
CEloss = \frac{1}{M}\sum_{j=1}^M |\hat G^{\psi}(P_{k,j},r_{k,j})-(1-\phi_{k+1,j})\hat F^\theta(P_{k+1,j},r_{k+1,j})-Rewards_{k+1,j}|^2.
\end{align}

}

\STATE{
Update $\theta$ and $\psi$ by minimizing $TDloss+CEloss$.
}

%

\ENDFOR

\end{algorithmic}

\textbf{Output}: $\hat F^{\theta_{T_m-1}}$.

\label{exm:algo:neuralTD}
\end{algorithm}

By Theorems \ref{fixedpoint} and \ref{epsilon-eq-finite}, we propose a policy iteration algorithm to compute a regularized equilibrium policy for this mean-field stopping problem, which is also an approximated relaxed equilibrium policy for the original problem \eqref{exm:finiteproblem} with finite stocks; see Algorithm \ref{exm:algo:policyiteration}. In particular, the Neural TD algorithm in the evaluation step (Line 2 of Algorithm \ref{exm:algo:policyiteration}) is detailed in Algorithm \ref{exm:algo:neuralTD}. The convergence of Neural TD is theoretically guaranteed, see \cite{CYLW23}. To approximate the function in \eqref{exm:eval}, we use the Markov property to obtain
\begin{align}\label{exm:eval:relation}
	\hat{F}^{\theta}(P_k,r_k) =\beta\mE\Big[ Rewards_{k+1}+(1-\phi^l(P_{k+1},r_{k+1}))\hat F^{\theta}(P_{k+1},r_{k+1})\Big|P_k,r_k \Big],
\end{align}
where $Rewards_k =(P_0-P_k)_+\phi(P_k,r_k)+\lambda\Epsilon(\phi(P_k,r_k))$, and $\{(P_k,r_k)\}_{k\geq 0}$ is any sample path of ETF prices and returns, either by simulation or from the real data. 

Note that in \eqref{exm:eval:relation} the temporal difference term is not the current reward $Rewards_k$ as in the usual TD learning task. The reason is that in our evaluation step, we need to compute the value function if we decide to continue, and compare it to the instantaneous reward by stopping. For the same reason, the term $(1-\phi^l(P_{k+1},r_{k+1}))$ appears in \eqref{exm:eval:relation}. Inspired by \eqref{exm:eval:relation}, we update the parameters by minimizing the $TDloss$ defined by 
\begin{align}
\mE \Big|	\hat{F}^{\theta}(P_k,r_k) -  \beta\mE\big[ Rewards_{k+1}+(1-\phi^l(P_{k+1},r_{k+1}))\hat F^{\theta}(P_{k+1},r_{k+1})\big|P_k,r_k \big]             \Big|^2.
\end{align}
However, we still need to compute a conditional expectation. Leveraging on the projection property of conditional expectation, we use another neural network $\hat G^\psi$ to approximate the right hand side of \eqref{exm:eval:relation} by minimizing the $CEloss$ defined by 
\begin{align}
	\mE \Big|\hat G^\psi(P_k,r_k) - Rewards_{k+1}-(1-\phi^l(P_{k+1},r_{k+1}))\hat F^{\theta}(P_{k+1},r_{k+1})      \Big|^2.
\end{align} 
See Algorithm \ref{exm:algo:neuralTD} for more details.

  We select model parameters as outlined in Table \ref{exm:modelpara}\footnote{  The values of $\bar c$ and $\bar d$ are set to be relatively small, primarily to match the simulated data with the real-world data (as depicted in Figure \ref{fig:comparison}). Increasing $\bar c$ and $\bar d$ will not compromise the convergence and stability of our algorithms, but it will result in an unreasonable scale for the simulated data.}. Additionally, $\epsilon^0$ is distributed according to a standard student-$t$ distribution with degrees of freedom set to 10. }, and $\epsilon^0$ that has a standard student-$t$ distribution with degrees of freedom 10. In Figure \ref{fig:comparison}, we simulate a path of index value and construct the histogram of its daily returns, and compare them with the real data from April 19, 2023 to April 17, 2024\footnote{Data source: \href{https://www.spglobal.com/spdji/en/indices/equity/sp-500}{https://www.spglobal.com/spdji/en/indices/equity/sp-500}.}. Because this subsection only aims to illustrate the applicability of our results, we do not calibrate model parameters. However, from Figure \ref{fig:comparison}, we can see that our choices of parameters are reasonable. In the implementation of Algorithms \ref{exm:algo:policyiteration} and \ref{exm:algo:neuralTD}, we use two three-layer fully connected neural networks with Tanh and ReLU activation functions. The optimizer to update the network parameters is Adam, with a learning rate 0.001. The time horizon for each policy evaluation task is $T_m=500$. The batch size is $M=200$. After $L=100$ iterations of policy update, the stopping policy is displayed in Figure \ref{exm:fig:policy}. During the training process, we document TD loss (see line 4 of Algorithm \ref{exm:algo:neuralTD}) for both the outer iteration (i.e., for $l=0,1,2,\cdots,100$ in Algorithm \ref{exm:algo:policyiteration}) and the inner iteration of the last step (i.e., for $l=100$); see Figure \ref{exm:fig:loss}.

\begin{table}
\centering 
\begin{tabular}{@{}llllllll@{}}
\toprule
$\bar a$         & $\bar b$ & $\bar c$ & $\bar d$ & $P_0$ & $\beta$ & $K$  & $\lambda$ \\ \midrule
$1.07^{1/252}-1$ & 0.5      & 0.006    & 0.006    & 100   & 0.7     & 1.01 & 0.1      \\ \bottomrule
\end{tabular}
\caption{Model parameters in numerical experiments}
\label{exm:modelpara}
\end{table}

\begin{figure}
     \centering
         \includegraphics[width=1\textwidth]{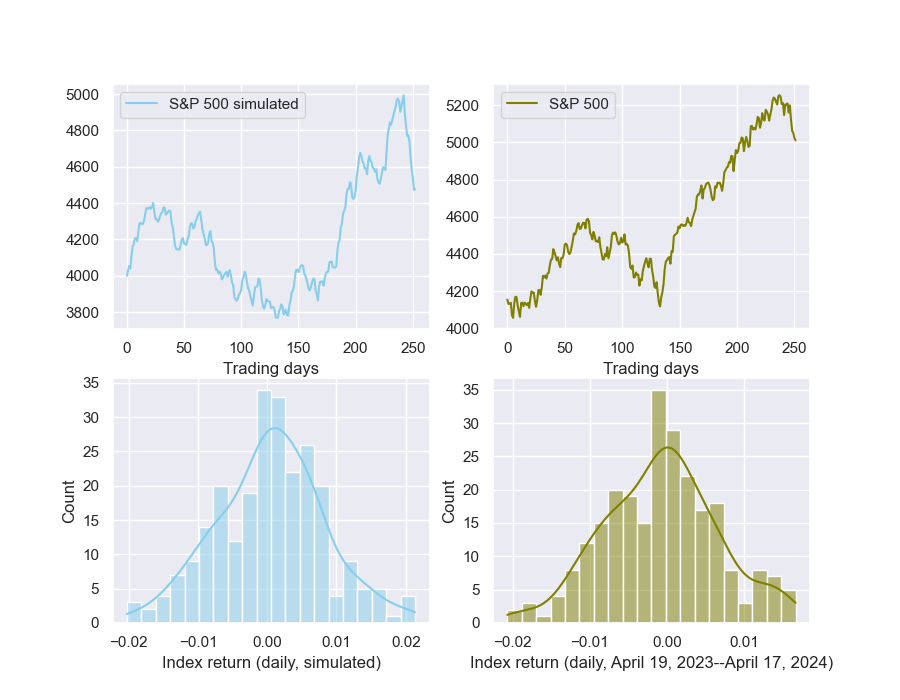}
         \caption{Simulated data from \eqref{exm:marketreturn} and \eqref{exm:index} v.s. Real-world data}
         \label{fig:comparison}
\end{figure}

\begin{figure}
     \centering
     \begin{subfigure}{0.7\textwidth}
         \centering
       \includegraphics[width=1\textwidth]{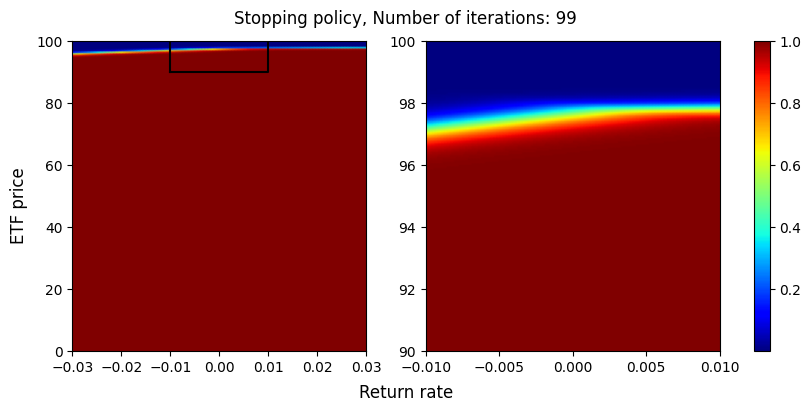}
     \end{subfigure}
     \hfill
     \begin{subfigure}{0.7\textwidth}
         \centering
         \includegraphics[width=1\textwidth]{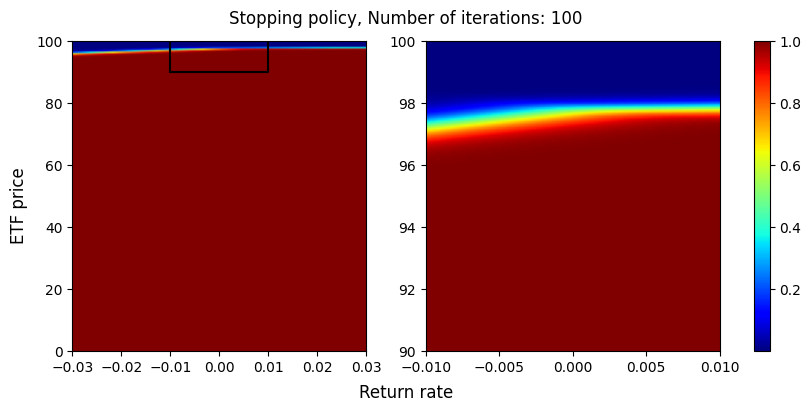}
     \end{subfigure}
     \caption{The output policy of Algorithm \ref{exm:algo:policyiteration}. The right panel contains zoom-in figures of the region indicated by black rectangles in the left panel.}
     \label{exm:fig:policy}
\end{figure}
 
 \begin{figure}
     \centering
     \begin{subfigure}[b]{0.45\textwidth}
         \centering
       \includegraphics[width=1\textwidth]{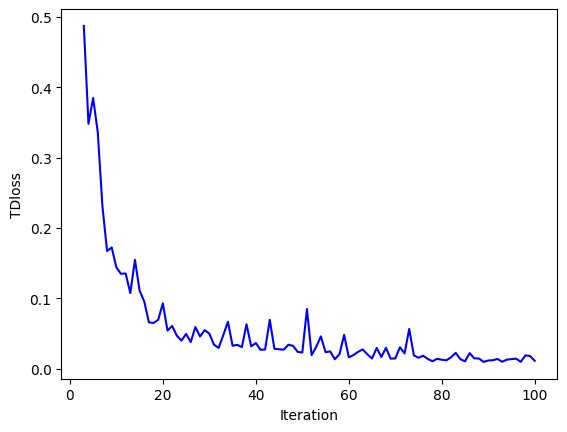}
         \caption{}
     \end{subfigure}
     \hfill
     \begin{subfigure}[b]{0.45\textwidth}
         \centering
         \includegraphics[width=1\textwidth]{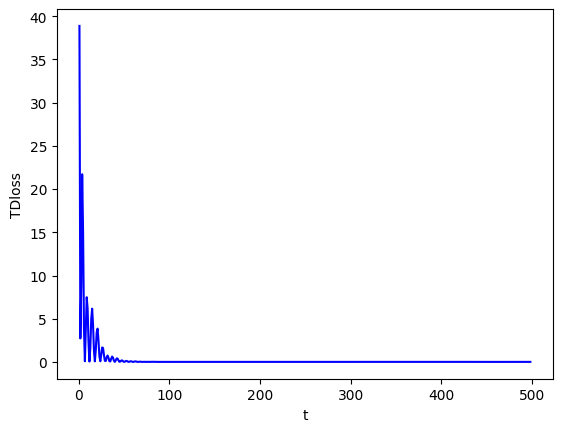}
         \caption{}
     \end{subfigure}
     \caption{(a): $TDloss$ after each outer iteration in Algorithm \ref{exm:algo:policyiteration}; (b): $TDloss$ in an inner iteration of TD learning when $l=100$.}
     \label{exm:fig:loss}
\end{figure}

It is clear from Figure \ref{exm:fig:policy} that the policy iteration becomes almost stationary after $l=100$. Therefore, the output policy after 100 iterations (in the second row of Figure \ref{exm:fig:policy}) can be accepted as an equilibrium policy. The stopping (exercise) region is red, and the continuation (holding) region is blue. There also exists a mixed region in which the stopping policy takes values in $(0,1)$ because we fix $\lambda=0.1$. It is not clear from the algorithm whether the mixed region will converge to a boundary when $\lambda\to 0$ because the algorithm becomes unstable for sufficiently small $\lambda$. Nevertheless, we have several observations from Figure \ref{exm:fig:policy}. First, the holding region is small. This is because day-to-day fluctuations of ETF prices are usually not too large (no larger than $\pm$2\% with high probability). Thus, if the current price is not close to the strike, rewards from a further decline of ETF prices do not exceed waiting costs (described by discount) plus upward risks. On the contrary, if the price is very close to the strike, holding may be profitable because the strike caps upward risks. Second, the holding region is larger when the return is negative. Such a result is intuitively correct, which is caused by the mean-field interactions in \eqref{exm:return:linear}. That is, a negative ETF return indicates a bad market performance, hence inducing a higher probability of further decline. Therefore, a lower exercise threshold brings more profit from selling the ETF at strike. 

\begin{remark}
	To simplify the illustration, we assume that a simulator of the mean-field dynamics is available. Therefore, we can focus on the social planner's problem in a mean-field model, which looks like a single-agent decision making. However, this is not the case in practice. First, there may exist more mean-field interaction terms other than the linear interaction as in \eqref{exm:return:linear}, hence loading only ETF or index data is not sufficient. We need to utilize original stock price data to compute state variables such as $\bar v_t$ in \eqref{exm:marketreturn} and \eqref{exm:average_sq_return}. Second, the simulator is not available in reality and scenario generation is known to be challenging in finance. Inspired by our results, we may first use stock price data, whose volume is large, to establish a prediction model of ETF data, and then turn to the regularized mean-field model, which is known to be solvable by Algorithms \ref{exm:algo:policyiteration} and \ref{exm:algo:neuralTD}. See the flow chart in Figure \ref{exm:fig:app} for an illustration of this procedure. Implementation of our method to real-world data will be left for future studies.
	\end{remark}

\tikzset{every picture/.style={line width=0.75pt}} 
\begin{figure}
\centering
\begin{tikzpicture}[x=0.75pt,y=0.75pt,yscale=-1,xscale=1]

\draw   (145,34) -- (260.87,34) -- (260.87,77.06) -- (145,77.06) -- cycle ;
\draw    (206.25,77.06) -- (206.25,104.12) ;
\draw [shift={(206.25,106.12)}, rotate = 270] [color={rgb, 255:red, 0; green, 0; blue, 0 }  ][line width=0.75]    (10.93,-3.29) .. controls (6.95,-1.4) and (3.31,-0.3) .. (0,0) .. controls (3.31,0.3) and (6.95,1.4) .. (10.93,3.29)   ;
\draw    (260.87,54.45) -- (428.06,54.45) ;
\draw    (428.06,54.45) -- (428.06,182.69) ;
\draw [shift={(428.06,184.69)}, rotate = 270] [color={rgb, 255:red, 0; green, 0; blue, 0 }  ][line width=0.75]    (10.93,-3.29) .. controls (6.95,-1.4) and (3.31,-0.3) .. (0,0) .. controls (3.31,0.3) and (6.95,1.4) .. (10.93,3.29)   ;
\draw   (146.66,189) -- (264.18,189) -- (264.18,227.75) -- (146.66,227.75) -- cycle ;
\draw   (148.31,113.65) -- (255.91,113.65) -- (255.91,153.48) -- (148.31,153.48) -- cycle ;
\draw    (205.94,155.63) -- (205.94,181.62) ;
\draw [shift={(205.94,183.62)}, rotate = 270] [color={rgb, 255:red, 0; green, 0; blue, 0 }  ][line width=0.75]    (10.93,-3.29) .. controls (6.95,-1.4) and (3.31,-0.3) .. (0,0) .. controls (3.31,0.3) and (6.95,1.4) .. (10.93,3.29)   ;
\draw   (370.13,186.85) -- (484.34,186.85) -- (484.34,227.75) -- (370.13,227.75) -- cycle ;
\draw    (272.46,210.53) -- (363.16,210.53) ;
\draw [shift={(365.16,210.53)}, rotate = 180] [color={rgb, 255:red, 0; green, 0; blue, 0 }  ][line width=0.75]    (10.93,-3.29) .. controls (6.95,-1.4) and (3.31,-0.3) .. (0,0) .. controls (3.31,0.3) and (6.95,1.4) .. (10.93,3.29)   ;
\draw    (204.59,233.13) -- (204.59,269.88) ;
\draw [shift={(204.59,271.88)}, rotate = 270] [color={rgb, 255:red, 0; green, 0; blue, 0 }  ][line width=0.75]    (10.93,-3.29) .. controls (6.95,-1.4) and (3.31,-0.3) .. (0,0) .. controls (3.31,0.3) and (6.95,1.4) .. (10.93,3.29)   ;
\draw   (146.66,278.34) -- (262.53,278.34) -- (262.53,314.94) -- (146.66,314.94) -- cycle ;
\draw    (204.59,316.01) -- (204.59,344) ;
\draw    (204.59,344) -- (429.72,342.92) ;
\draw    (429.72,342.92) -- (429.98,321) ;
\draw [shift={(430,319)}, rotate = 90.67] [color={rgb, 255:red, 0; green, 0; blue, 0 }  ][line width=0.75]    (10.93,-3.29) .. controls (6.95,-1.4) and (3.31,-0.3) .. (0,0) .. controls (3.31,0.3) and (6.95,1.4) .. (10.93,3.29)   ;
\draw   (371.78,277.26) -- (486,277.26) -- (486,313.86) -- (371.78,313.86) -- cycle ;
\draw    (429.72,276.19) -- (429.99,233) ;
\draw [shift={(430,231)}, rotate = 90.36] [color={rgb, 255:red, 0; green, 0; blue, 0 }  ][line width=0.75]    (10.93,-3.29) .. controls (6.95,-1.4) and (3.31,-0.3) .. (0,0) .. controls (3.31,0.3) and (6.95,1.4) .. (10.93,3.29)   ;

\draw (162,37) node [anchor=north west][inner sep=0.75pt]   [align=left] {Stock prices \\ \ \ \ \ \ \ data};
\draw (150.31,116.65) node [anchor=north west][inner sep=0.75pt]   [align=left] {\small Finite-stock \\model:\eqref{exm:stockprice}-\eqref{exm:return:model} };
\draw (168,197) node [anchor=north west][inner sep=0.75pt]   [align=left] {Simulator};
\draw (154.66,281.34) node [anchor=north west][inner sep=0.75pt]  [font=\small] [align=left] {Output exercise \\ \ \ \ \ \ policy};
\draw (394,198) node [anchor=north west][inner sep=0.75pt]   [align=left] {ETF data};
\draw (394,288) node [anchor=north west][inner sep=0.75pt]   [align=left] {ETF put};
\draw (441,243) node [anchor=north west][inner sep=0.75pt]   [align=left] {written on};
\draw (293,28) node [anchor=north west][inner sep=0.75pt]   [align=left] {ETF management company};
\draw (32,79) node [anchor=north west][inner sep=0.75pt]  [font=\small] [align=left] {Regression/non-parametric \\learning};
\draw (32,162) node [anchor=north west][inner sep=0.75pt]   [align=left] {Mean-field limit};
\draw (32,244) node [anchor=north west][inner sep=0.75pt]   [align=left] {Algorithms \ref{exm:algo:policyiteration} \& \ref{exm:algo:neuralTD}};
\draw (210,353) node [anchor=north west][inner sep=0.75pt]   [align=left] {Approximated equilibrium policy};
\draw (293,187) node [anchor=north west][inner sep=0.75pt]   [align=left] {Predict};
\end{tikzpicture}
\caption{Application of Subsection \ref{exm:ETF:option} to real-world ETF option exercise.}
\label{exm:fig:app}
\end{figure}
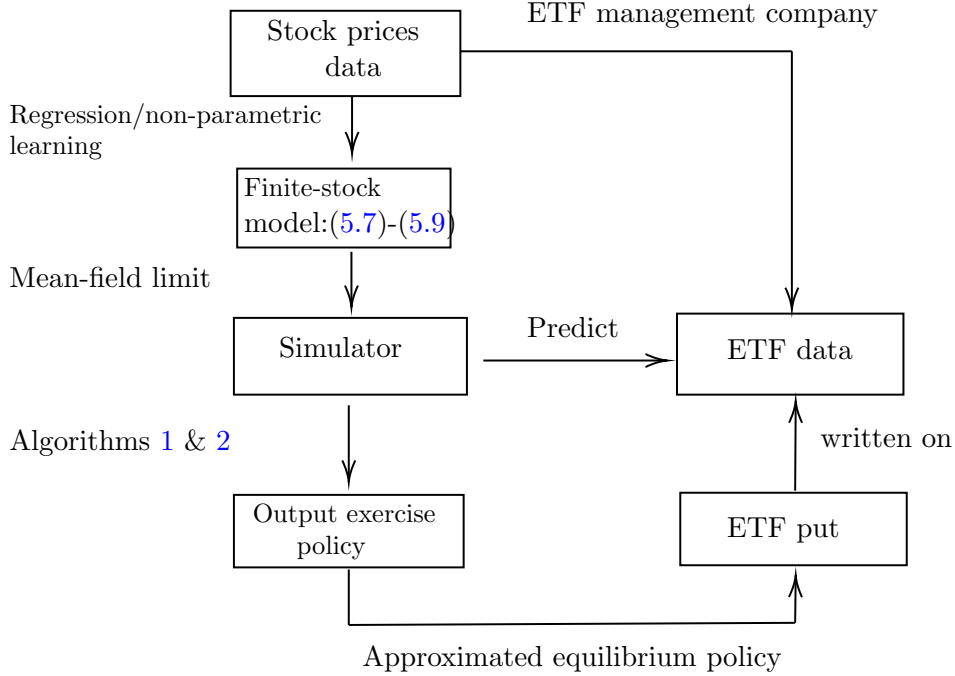

\section{  Extension to Asynchronous Stopping}\label{sec:extension}
{  
In the main body of the paper, we focus on the case of synchronous stopping, i.e., the social planner makes a decision to stop the entire system simultaneously. Such a formulation simplifies notations and some technical proofs. In Section \ref{sec:exm}, we also provide some financial applications that fit into this synchronous stopping formulation. In this section, we extend the previous formulation to allow for asynchronous stopping, i.e., only part of the system will be stopped. We show that when the system only depends on the distribution of {\it unstopped} ({\it surviving}) agents and does not depend on the survival ratio, most of our previous techniques and proofs remain valid. More importantly, we can still express a regularized equilibria as a fixed point in explicit form, so that the policy iteration algorithms in Subsection \ref{exm:ETF:option} are still applicable after mild modifications.

Consider an extended state space $S\cup \{\triangle_S\}$, where $\triangle_S$ denotes the stopped state. Any probability measure $\hat \mu \in \hat S := \cP(S\cup \{\triangle_S\})$ can be uniquely determined by a pair $(\bp_1 \hat \mu,\bp_2 \hat\mu)\in \cP(S)\times [0,1]$ via $\hat \mu \big|_S = \bp_1\hat \mu$ and $\hat \mu(\{\triangle_S\})=\bp_2 \hat\mu$ . Consider a transition function of the {\it representative agent} $T^r:S\times \hat S \times \{0,1\}\times \cZ\times \cZ \to S$, and denote $T^r_0:= T^r(\cdot,\cdot,0,\cdot,\cdot)$. Moreover, let $T^r(x,\hat \mu,1,z,z^0)=\triangle_S$ and $T^r(\triangle_S,\hat \mu,0,z,z^0)=\triangle_S$ for any $(x,\hat\mu,z,z^0)\in S\times \hat S\times \cZ\times \cZ$.
 Note that, in Section \ref{finiteconvergencesection}, the second variable of $T^r$ is in $\bar S = \cP(S)\cup \{\delta_{\triangle_S} \}$ because all agents either are continued (so that the distribution lies in $\cP(S)$) or stopped (so that the distribution equals to $\delta_{\triangle_S}$) in the synchronous stopping setting. In contrast, here we allow only a part of agents to be stopped, thus agents can distribute on $S\cup\{\triangle_S\}$. We make the following important assumption on $T^r$.
\begin{assumption}\label{assum:partialstopping}
	$T^r$ depends on $\hat\mu\in \hat S$ only through $\bp_1\hat \mu$. Specifically, for any $(x,a,z,z^0)\in S\times \{0,1\}\times \cZ\times \cZ$ and $\hat \mu,\hat \mu'\in \hat S$, if $\bp_1\hat \mu = \bp_1\hat \mu'$, then $T^r(x,\hat\mu,a,z,z^0)=T^r(x,\hat\mu',a,z,z^0)$.
	\end{assumption}
Under Assumption \ref{assum:partialstopping}, $T^r$ can be equivalently treated as a function on $S\times \cP(S)\times \cZ\times \cZ$, which induces a transition function $T_0:\cP(S)\times \cZ\to \cP(S)$ for the distribution of unstopped agents via $T_0(\mu,z^0) = T^r_0(\cdot,\mu,\cdot,z^0)_\# (\mu\times \cL(\cZ)'$. See similar relations in Subsection \ref{relationlimit}.

In the asynchronous stopping setting, the decision is still made by the social planner to maximize the average reward for all agents, and we still consider the randomized policies. However, a policy is now a function $\phi: S\times \cP(S)\to [0,1]$\footnote{Recall that in other sections of this paper, a policy is $\phi:\bar S\to [0,1]$}. That is, each agent receives the policy from the social planner, but whether the agent stops or not depends on his own state and the distribution of other unstopped agents. Moreover, agents have their own randomization devices $\{U^i\}_{i\in [N]}$ while in Section \ref{finiteconvergencesection} there is a common randomization device $U$. With these notations, for $i\in [N]$ we consider the finite-agent dynamics as follows:
\begin{equation}
 X^{i,N,\phi}_{k+1}=T^r\left(X^{i,N,\phi}_k,\frac{1}{N}\sum_{j\in [N]}\delta_{X_k^{i,N,\phi}},\ind_{\Big\{U^i_{k+1}\leq \phi\big(X^{i,N,\phi}_k,\frac{1}{N}\sum_{j\in [N]}\delta_{X^{i,N,\phi}_k}\big)\Big\}},Z^i_{k+1},Z^0_{k+1}\right). \label{partialstopping:finitedynamic}
 \end{equation}
The limiting mean-field dynamics are 
\begin{equation}
 X^{i,\phi}_{k+1}=T^r\left(X^{i,\phi}_k,\mP^0_{X^{i,\phi}_k},\ind_{\Big\{U^i_{k+1}\leq \phi\big(X^{i,\phi}_k,\bp_1 \mP^0_{X^{i,\phi}_k}\big)\Big\}},Z^i_{k+1},Z^0_{k+1}\right). \label{partialstopping:limitdynamic}
\end{equation}
The representative agent has a reward function $f:S\times \cP(S)\to \mR$. Therefore, if the social planner chooses a policy $\phi$, {\it the decoupled value function} of the representative agent is given by 
\begin{equation}\label{decoupled_J}
J^\phi_d(x,\mu) = \mE_d^{x,\mu,\phi} \sum_{k=0}^\infty \delta(k)( f(X_k,\mu_k)\phi(X_k,\mu_k)-\lambda \E(\phi(X_k,\mu_k))).
\end{equation}
Here, $\mE_d^{x,\mu,\phi}$ is the expectation operator under $\mP_d^{x,\mu,\phi}$, the probability measure on the canonical space $\Omega$ that is induced by the following transition rules:
\begin{align}
&\mu_{k+1} = T_0(\mu_k,Z^0_{k+1}), \quad \mu_0 = \mu,\\
&X_{k+1} = T^r(X_k, \mu_k, \ind_{\{U_{k+1}\leq \phi(X_k,\mu_k) \}} ,    Z_{k+1},   Z^0_{k+1}), \quad X_0 = x.
\end{align}
By setting $f(\triangle_S,\mu)=0$, we are safe to let $J^{\phi}_d(\triangle_S,\mu)=0$.

The goal of the social planner is to maximize $J^{\phi}(\mu) = \int_S J^\phi_d(x,\mu)\mu(dx)$. Because of the time-inconsistency due to general discount functions, the social planner aims to find a time-consistent equilibrium $\phi^*$ via the game theoretic thinking that
\begin{equation}\label{eqcond:partialstopping}
	J^{\psi \oplus_1 \phi^*}(\mu)\leq J^{\phi^*}(\mu),\forall \mu\in \cP(S),\psi:S\times \cP(S)\to [0,1].
\end{equation}
We can still transform the equilibrium condition \eqref{eqcond:partialstopping} to a fixed point problem similar to \eqref{equifix}. Indeed, it holds that
\begin{align}
	J^{\psi\oplus_1\phi^*}(\mu)=& \int_S J^{\psi\oplus_1\phi^*}_d(x,\mu)\\
	=&\int_S [f(x,\mu)\psi(x)-\lambda\E(\psi(x))]\mu(\md x)\\
	&+\int_S \Big[\mE_d^{x,\mu,\psi}\mE_d^{X_1,\mu_1,\phi^*}\sum_{k=1}^\infty\delta(k)(f(X_{k-1},\mu_{k-1})\phi(X_{k-1},\phi_{k-1})-\lambda\E(\phi(X_{k-1},\phi_{k-1}))  \Big]\mu(\md x)\\
	=&\int_S [f(x,\mu)\psi(x)-\lambda\E(\psi(x))+\mE^{x,\mu,\psi}\tilde{J}^{\phi^*}_d(X_1,\mu_1)]\mu(\md x)\\
	=&\int_S \Big[f(x,\mu)\psi(x)-\lambda\E(\psi(x))\\
	&+\mE^{U_1,Z_1,Z^0_1}\tilde{J}^{\phi^*}_d\big(T^r(x,\mu,\ind_{\{U_1\leq \psi(x)\}},Z_1,Z^0_1),T_0(\mu,Z^0)\big)\Big]\mu(\md x)\\
	=&\int_S \Big[f(x,\mu)\psi(x)-\lambda\E(\psi(x))\\
	&+(1-\psi(x))\mE^{Z,Z^0}\tilde{J}^{\phi^*}_d\big(T^r_0(x,\mu,Z,Z^0),T_0(\mu,Z^0)\big)\Big]\mu(\md x).\label{equichar:partialstopping}
\end{align}
Clearly, \eqref{eqcond:partialstopping} is satisfied if $\phi^*(\cdot,\mu)$ maximizes the integrand over $\psi$, which is equivalent to 
\begin{equation}\label{equifix:partialstopping}
	\phi^*(x,\mu) = \frac{\exp\Big(\frac{1}{\lambda}f(x,\mu)\Big)}{\exp\Big(\frac{1}{\lambda}f(x,\mu)\Big)+\exp\Big(\frac{1}{\lambda}f^{c,d}_{\phi^*}(x,\mu) \Big)},
\end{equation}
where $f^{c,r}_{\phi^*}(x,\mu) = \mE^{Z,Z^0}\tilde{J}^{\phi^*}_d\big(T^r_0(x,\mu,Z,Z^0),T_0(\mu,Z^0)\big)$\footnote{The superscript $^c$ stands for "continuation" and $^r$ stands for "representative agent".} is the reward of continuation. To derive \eqref{equifix:partialstopping}, Assumption \ref{assum:partialstopping} is crucial because it gives a transition rule of unstopped distribution that does not depend on $\psi$. Consequently, the maximizer of integrand in \eqref{equichar:partialstopping} can be solved explicitly.

To study \eqref{equifix:partialstopping}, our techniques and proofs in Section \ref{proofs} are still applicable with appropriate assumptions on $Z$, $Z^0$, $T^r_0$ and $T_0$ because we mainly use the compactness of $\bar S$. Similarly, we may also prove the convergence of \eqref{partialstopping:finitedynamic} to \eqref{partialstopping:limitdynamic} by arguments in Section \ref{finiteconvergencesection}. However, the case when $T^r$ and $f$ depend on both $\bp_1\hat \mu$ and $\bp_2\hat \mu$ (i.e., survival ratio will also affect transitions and rewards) is more sophisticated. On the other hand, taking $S\times \{0,1\}$ as the state space, we may also allow {\it distribution of stopped agents} (not only the survival ratio) to affect the system. These two extensions evidently require new tools that are beyond the scope of this paper and will be left for future studies.

}

\section{Proofs}\label{proofs}
In this section, we prove Theorems \ref{fixedpoint} and \ref{existoriginal}. To this end let us consider the following subset of $C(\barS)$:
\begin{definition}
	We define $L_K$ as the set of Lipchitz continuous functions with the Lipchitz constant $K>0$, and $B_R$ as a ball in $C(\barS)$ with radius $R>0$ centered at 0. In other words:
	\begin{align*}
	&L_K:=\{f\in C(\barS):\|f\|_{\rm Lip}\leq K \},\\
	&B_R:= \{f\in C(\barS):\|f\|_\infty \leq R    \},
	\end{align*}
where
\[
\|f\|_{\rm Lip}:= \sup_{\mu,\mu'\in \barS,\mu\neq \mu'}\frac{|f(\mu)-f(\mu')|}{\cW_1(\mu,\mu')}. 
\]
\end{definition}
The proof proceeds as follows: first, as preliminary results, we provide estimations of $\Ti^\lambda_1 v$ and the transition functions; then we consider regularized operator to ensure certain compactness, thus Schauder's fixed point theorem is applicable; finally we consider the limit $\lambda\to 0$ and prove the existence of fixed points of the original operator. Throughout this section, we assume that Assumptions \ref{as1}, \ref{Tassumption} and \ref{Tassumption2} are in force.
\subsection{Preliminary Results}
This subsection establishes several technical results that will be used repeatedly in the proofs of main theorems. We first provide simple estimations of $\Ti^\lambda_1v$, both in uniform norms and in Lipschitz seminorms. At the last of this subsection, we will provide $\cW_1$ Lipschitz estimations for the Markov family $\mP^{\mu,\phi}$. These estimations are crucial for the application of Schauder's fixed point theorem and the analysis of the limit $\lambda\to 0$. We emphasize that because $\barS$ is inherently uncountable even if $S$ is finite, we can not apply results in \cite{Bayraktar2023} and a large part of our proofs are novel.

\begin{lemma}\label{phiestuniformlemma}
  There exists a $K_1:=K_1(\lambda,\|v\|_{\infty})>0$ such that for any $v\in C(\barS)$ and $\mu,\mu'\in \barS$, we have that
    \begin{align}
        &\frac{1}{1+\exp(\frac{1}{\lambda}(\|v\|_{\infty}+L_1))}\leq \Ti^\lambda_1(v)(\mu)\leq 1,\label{phiestuniform}\\
     \text{and}\ \ \   &
         |\Ti^\lambda_1(v)(\mu)-\Ti^\lambda_1(v)(\mu')|\leq K_1\cW_1(\mu,\mu').\label{phiestLip}
    \end{align}
\end{lemma}
\begin{proof}
\eqref{phiestuniform} is a direct consequence of the expression \eqref{T1def}. To show \eqref{phiestLip}, we first note that 
\begin{align*}
|\mE^0v(T_0(\mu,Z^0))-\mE^0v(T_0(\mu',Z^0))|&\leq \|v\|_{\infty}\bar{d}_{\rm TV}(T_0(\mu,\cdot)_\#\cL(Z),T_0(\mu',\cdot)_\#\cL(Z))\\
&\leq \|v\|_{\infty}L_3\cW_1(\mu,\mu').
\end{align*}
Therefore, denoting $\Hi^v_\lambda(\mu):= \frac{1}{\lambda}[\mE^0v(T_0(\mu,Z^0))-r(\mu)]$, we have
\begin{align*}
    |\Ti^\lambda_1(v)(\mu)-\Ti^\lambda_1(v)(\mu')|&=\frac{|\exp(\Hi^v_\lambda(\mu))-\exp(\Hi^v_\lambda(\mu'))|}{(1+\exp(\Hi^v_\lambda(\mu)))(1+ \exp(\Hi^v_\lambda(\mu')) )}\\
    &\leq \frac{\int_0^1\exp(\alpha\Hi^v_\lambda(\mu)+(1-\alpha)\Hi^v_\lambda(\mu') )\md \alpha}{(1+\exp(\Hi^v_\lambda(\mu)))(1+ \exp(\Hi^v_\lambda(\mu')) )} |\Hi^v_\lambda(\mu)-\Hi^v_\lambda(\mu')|\\
    &\leq \frac{1}{\lambda}\exp\left(\frac{1}{\lambda}(\|v\|_{\infty}+L_1)\right)(L_3\|v\|_{\infty}+L_2)\cW_1(\mu,\mu').
\end{align*}
That is, \eqref{phiestLip} indeed holds by choosing $K_1=\frac{1}{\lambda}\exp\left(\frac{1}{\lambda}(\|v\|_{\infty}+L_1)\right)(L_3\|v\|_{\infty}+L_2)$.
\end{proof}

\begin{lemma}\label{W1estLemma}
    There exists a $K_2=K_2(\lambda,\|v\|_{\infty})>0$ such that for $v\in C(\barS)$, $\phi=\Ti^\lambda_1(v)$ and any $\mu,\mu'\in \barS$, we have that
    \begin{equation}\label{W1est}
        \overline{\cW}_1(\mP^{\mu,\phi}_k,\mP^{\mu',\phi}_k)\leq K_2^k\cW_1(\mu,\mu').
    \end{equation}
  Hereafter, for any $\mu\in \barS$, $\phi\in \F$, and $k\in \mT$, $\mP^{\mu,\phi}_k:= (\mu_k)_\# \mP^{\mu,\phi}$ denotes the marginal distribution of $\mu_k$ under $\mP^{\mu,\phi}$. 
\end{lemma}
\begin{proof}
 For $k\geq 1$, let us consider any pair of random variables $\mu_{k-1}$ and $\mu'_{k-1}$\footnote{We slightly abuse the notation because $\{\mu_k\}_{k\in \mT}$ is the canonical process in the main body of the paper. However, this should not produce confusion because it only appears in the proofs. The same comment applies to the proof of Lemma \ref{W1estphiLemma}.} on the same probability space $(\bar{\Omega},\bar{F},\bar{\mP})$, with margins $\mP^{\mu,\phi}_{k-1}$ and $\mP^{\mu',\phi}_{k-1}$ respectively. Consider also another two probability spaces $(\hat{\Omega},\hat{\F},\hat{\mP})$ and $(\tilde{\Omega},\tilde{F},\tilde{\mP})$. On $\hat{\Omega}$ we consider a random variable $U$ which has uniform distribution on $[0,1]$ and on $\tilde{\Omega}$ we consider a random variable $Z$ with distribution $\cL(\cZ)$ on $\cZ$ (i.e., the common noise). Finally, define $(\check{\Omega},\check{\F},\check{\mP})$ to be the product space of three aforementioned probability spaces. All random variables are naturally lifted as random variables on $\check{\Omega}$. Consider
 \begin{equation}\label{coupling1}
 \begin{aligned}
     &\mu_k=T(\mu_{k-1},\ind_{\{U\leq \phi(\mu_{k-1})\}},Z),\\
    &\mu'_k=T(\mu'_{k-1},\ind_{\{U\leq \phi(\mu'_{k-1})\}},Z),
 \end{aligned}
 \end{equation}
 which are random variables on $\check{\Omega}$. By the definition of $\mP^{\mu,\phi}$ (see \eqref{transition}), $\mu_k$ has distribution $\mP^{\mu,\phi}_k$ and $\mu_k'$ has distribution $\mP^{\mu',\phi}_k$. In other words, they constitute one specific coupling of $(\mP^{\mu,\phi}_{k},\mP^{\mu',\phi}_{k})$. By definition of the Wasserstein metric,
 \begin{equation}\label{W1est1}
 \begin{aligned}
     \overline{\cW}_1(\mP^{\mu,\phi}_k,\mP^{\mu',\phi}_k)\leq &\check{\mE}\left[ \cW_1(T(\mu_{k-1},\ind_{\{U\leq \phi(\mu_{k-1})\}},Z), T(\mu'_{k-1},\ind_{\{U\leq \phi(\mu'_{k-1})\}},Z))                 \right]\\
     \leq& \check{\mE}\left[ \cW_1(T(\mu_{k-1},\ind_{\{U\leq \phi(\mu_{k-1})\}},Z), T(\mu'_{k-1},\ind_{\{U\leq \phi(\mu_{k-1})\}},Z))                 \right]\\
     &+\check{\mE}\left[ \cW_1(T(\mu’_{k-1},\ind_{\{U\leq \phi(\mu_{k-1})\}},Z), T(\mu'_{k-1},\ind_{\{U\leq \phi(\mu'_{k-1})\}},Z))                 \right]    \\
     \leq &L_4\check{\mE}\cW_1(\mu_{k-1},\mu_{k-1}')+2\check{\mE}\left[ \cW_1(T_0(\mu’_{k-1},Z), \triangle)\right]\ind_{\{U\in [\phi(\mu_{k-1}),\phi(\mu'_{k-1})]\}}\\
     \leq &(L_4+2CK_1)\check{\mE}\cW_1(\mu_{k-1},\mu_{k-1}').
\end{aligned}
\end{equation}
The last inequality is due to \eqref{phiestLip} and the fact that $\cW(\mu,\nu)\leq C$ for all $\mu,\nu\in \barS$. By the arbitrariness of $\mu_{k-1}$ and $\mu_{k-1}'$, we conclude
\[
\overline{\cW}_1(\mP^{\mu,\phi}_k,\mP^{\mu',\phi}_k)\leq K_2\overline{\cW_1}(\mP^{\mu,\phi}_{k-1},\mP^{\mu',\phi}_{k-1}).
\]
Here, we choose $K_2:=L_4+2CK_1$. By the induction and noting that $\mP^{\mu,\phi}_0=\delta_\mu$, $\overline{\cW_1}(\delta_\mu,\delta_{\mu'})=\cW_1(\mu,\mu')$, we obtain the desired \eqref{W1est}.
\end{proof}

\begin{lemma}\label{W1estphiLemma}
    For any $K>0$, there exists a $K_3=K_3(\lambda,K)>0$ such that for $\mu\in \barS$ and $\phi,\phi'\in L_K$, we have
    \begin{equation}\label{W1estphi}
        \overline{\cW}_1(\mP^{\mu,\phi}_k,\mP^{\mu,\phi'}_k)\leq K_3^k\|\phi-\phi'\|_{\infty}.
    \end{equation}
\end{lemma}
\begin{proof}
    We use a similar coupling method as in the proof of Lemma \ref{W1estLemma}. However, instead of \eqref{coupling1}, we use the following coupling:
    \begin{equation*}
 \begin{aligned}
     &\mu_k=T(\mu_{k-1},\ind_{\{U\leq \phi(\mu_{k-1})\}},Z),\\
\text{and}\ \ \   &\mu'_k=T(\mu'_{k-1},\ind_{\{U\leq \phi'(\mu'_{k-1})\}},Z),
 \end{aligned}
 \end{equation*}
 where $(\mu_{k-1},\mu'_{k-1})$ has marginals $(\mP^{\mu,\phi}_{k-1},\mP^{\mu,\phi'}_{k-1})$, while all other steps of argument remain the same as in the proof of Lemma \ref{W1estLemma}. Similar to \eqref{W1est}, we have
  \begin{equation*}
 \begin{aligned}
     \overline{\cW}_1(\mP^{\mu,\phi}_k,\mP^{\mu,\phi'}_k)\leq &\check{\mE}\left[ \cW_1(T(\mu_{k-1},\ind_{\{U\leq \phi(\mu_{k-1})\}},Z), T(\mu'_{k-1},\ind_{\{U\leq \phi'(\mu'_{k-1})\}},Z))                 \right]\\
     \leq& \check{\mE}\left[ \cW_1(T(\mu_{k-1},\ind_{\{U\leq \phi(\mu_{k-1})\}},Z), T(\mu'_{k-1},\ind_{\{U\leq \phi(\mu_{k-1})\}},Z))                 \right]\\
     &+\check{\mE}\left[ \cW_1(T(\mu’_{k-1},\ind_{\{U\leq \phi(\mu_{k-1})\}},Z), T(\mu'_{k-1},\ind_{\{U\leq \phi'(\mu'_{k-1})\}},Z))                 \right]    \\
     \leq &L_4\check{\mE}\cW_1(\mu_{k-1},\mu_{k-1}')+2\check{\mE}\left[ \cW_1(T_0(\mu’_{k-1},Z), \triangle)\right]\ind_{\{U\in [\phi(\mu_{k-1}),\phi'(\mu'_{k-1})]\}}\\
     \leq &L_4\check{\mE}\cW_1(\mu_{k-1},\mu_{k-1}')+2C\check{\mE}|\phi(\mu_{k-1})-\phi'(\mu'_{k-1})| \\
     \leq &L_4\check{\mE}\cW_1(\mu_{k-1},\mu_{k-1}')+2C\check{\mE}|\phi(\mu_{k-1})-\phi(\mu'_{k-1})|+2C\check{\mE}|\phi(\mu'_{k-1})-\phi'(\mu'_{k-1})|\\
     \leq &(L_4+2CK)\check{\mE}\cW_1(\mu_{k-1},\mu_{k-1}')+2C\|\phi-\phi'\|_{\infty}.
\end{aligned}
\end{equation*}
Therefore, it holds that
\[
\overline{\cW}_1(\mP^{\mu,\phi}_k,\mP^{\mu,\phi'}_k)\leq(L_4+2CK)\overline{\cW_1}(\mP^{\mu,\phi}_{k-1},\mP^{\mu,\phi'}_{k-1})+2C\|\phi-\phi'\|_{\infty}.
\]
Using induction and the fact $\cW_1(\mu_0,\mu'_0)=0$, we obtain
\[
\overline{\cW}_1(\mP^{\mu,\phi}_k,\mP^{\mu,\phi'}_k)\leq\frac{(L_4+2CK)^k-1}{L_4+2CK-1}2C\|\phi-\phi'\|_{\infty}.
\]
We can choose $K_3$ large enough so that \eqref{W1estphi} holds true, and it is clear to see that $K_3$ only depends on $\lambda$ and $K$ (the Lipschitz constant of $\phi$ and $\phi'$).
\end{proof}

Finally, we provide a convergence result with vanishing regularization parameter, i.e., with $\lambda\to 0$.
\begin{proposition}\label{lambdavanishing}
    As $\lambda\to 0$, we have
    \[
    \lambda \sum_{k=0}^{\infty}\left( \frac{1}{1+\lambda}   \right)^{k^2}\to 0.
    \]
\end{proposition}
\begin{proof}
    For any $N\in \mN$ and {  $\lambda>0$}, it holds that
    \begin{align}
    \sum_{k=0}^\infty\left( \frac{1}{1+\lambda}   \right)^{k^2}&\leq \sum_{k=0}^N\left( \frac{1}{1+\lambda}   \right)^{k^2}+\sum_{k=N+1}^\infty\left( \frac{1}{1+\lambda}   \right)^{k^2}\\
    &\leq N+\sum_{k=N+1}^\infty \left[\left(\frac{1}{1+\lambda}   \right)^{N+1}    \right]^k\\
    &=N+\frac{1}{(1+\lambda)^{N(N+1)}((1+\lambda)^{N+1}-1)}\\
    &\leq N+\frac{1}{((1+\lambda)^{N+1}-1)}.\label{lambda-N-bound}
    \end{align}
{   Fix an arbitrary $\lambda>0$ and take $N=\frac{\log(1+\sqrt{\lambda})}{\log (1+\lambda)}-1$ in \eqref{lambda-N-bound}, i.e., $(1+\lambda)^{N+1}=1+\sqrt{\lambda}$. Then,
\begin{align*}
    \sum_{k=0}^\infty\left( \frac{1}{1+\lambda}   \right)^{k^2}&\leq  \frac{\log(1+\sqrt{\lambda})}{\log (1+\lambda)}-1+\frac{1}{\sqrt{\lambda}}\\
    &\leq \frac{(1+\lambda)\log (1+\sqrt{\lambda})}{\lambda}+\frac{1}{\sqrt{\lambda}}\end{align*}
where we have used the elementary inequality $\log (x)\geq 1-\frac{1}{x}$ for any $x>0$. Consequently, for any $\lambda>0$, we have 
\begin{align}
	\lambda \sum_{k=0}^\infty\left( \frac{1}{1+\lambda}   \right)^{k^2} \leq (1+\lambda)\log(1+\sqrt{\lambda})+\sqrt{\lambda}.\label{lambda-bound}
\end{align}
Letting $\lambda\to 0$ in \eqref{lambda-bound} completes the proof.
}
\end{proof}

\subsection{Existence of Regularized Equilibria}\label{existregular}
To establish the existence of fixed points of $\Ti^\lambda:= \Ti^\lambda_2\circ\Ti^\lambda_1:C(\barS)\to C(\barS)$, we resort to Schauder's fixed point theorem. For ease of presentation, let us first take it for granted that $T^{\lambda}(v)$ is continuous if $v$ is, which will be rigorously verified later. 

Let us first show that $\Ti^\lambda$ is uniformly bounded.
\begin{lemma}\label{Tibound}
    There exists a constant $R>0$ such that $\|\Ti^\lambda(v)\|_{\infty}\leq R$, for any $\lambda\in (0,1]$ and $v\in C(\barS)$.
\end{lemma}
\begin{proof}
    From the proof of Lemma \ref{lambdavanishing}, it holds that when $\lambda\leq 1$,
\[
    \lambda\sum_{k=0}^\infty\delta_\lambda(k+1)\leq 2\log 2+1=: L_0.
\]
Choose any $v\in C(\barS)$ and denote $\phi^\lambda=\Ti^\lambda_1(v)$. By the definition of $\Ti^\lambda$ and Lemma \ref{formulationclassical}, we obtain that
\begin{align*}
    |\Ti^\lambda(v)(\mu)|&\leq \left|\sum_{k=0}^\infty \delta_\lambda(k+1)\mE^{\mu,\phi_\lambda}\left[r(\mu_k)\phi_\lambda(\mu_k)-\lambda\E(\phi(\mu_k))    \right]\right|\\
    &\leq \sum_{k=0}^\infty \delta_\lambda(k+1)|\mE^{\mu,\phi^\lambda}r(\mu_k)\phi^\lambda(\mu_k)|  +(\log 2)\lambda\sum_{k=0}^\infty \delta_\lambda(k+1)\\
    &\leq L_1 \sum_{k=0}^\infty \tilde{\mE}^\mu \phi^\lambda(\mu_k)\prod_{j=0}^{k-1}(1-\phi^\lambda(\mu_j)) +L_0\log 2 \\
    &\leq L_1+L_0\log 2
\end{align*}
Choosing $R:=L_1+L_0\log 2$, we complete the proof.
\end{proof}
As a consequence of Lemma \ref{Tibound}, $\Ti^\lambda$ is invariant in $B_R$. We emphasize that $R$ is a universal constant in this paper, which is independent of $\lambda$. 

The next lemma provides Lipschitz properties of $\Ti^{\lambda}(v)$ for any $v\in C(\barS)$, whose proof crucially relies on the proper regularization of the discount function $\delta_\lambda$.

\begin{lemma}\label{TLipLemma}
    There exists a $K_4=K_4(\lambda)$, such that for any $v\in B_R$, $\Ti^{\lambda}(v)\in L_{K_4}$.
\end{lemma}
\begin{proof}
    Fix a $v\in C(\barS)$ and denote $\phi=\Ti^{\lambda}_1(v)$. Let us consider a function on $\barS$ that 
    \[
    f(\mu):=r(\mu)\phi(\mu)-\lambda \E(\phi(\mu)).
    \]
    By Lemma \ref{phiestuniformlemma} (the lower bound), there exists a $\tilde{K}'$ such that $\sup_{\mu\in \barS}|\log \phi(\mu)|\leq \tilde{K}'$. Consequently, for $\mu,\mu'\in \barS$, it holds that
    \begin{align*}
    |\E(\phi(\mu))-\E(\phi(\mu'))|
    \leq& |\phi(\mu)-\phi(\mu')|\times\\
    &\int_0^1(2+|\log(\alpha\phi(\mu)+(1-\alpha)\phi(\mu'))|+|\log(1-\alpha\phi(\mu)-(1-\alpha)\phi(\mu'))|)\md \alpha\\
    \leq& 2(\tilde{K}'+1)|\phi(\mu)-\phi(\mu')|\\
    \leq& 2K_1(\tilde{K}'+1) \cW_1(\mu,\mu').
    \end{align*}
Therefore, we have
\[
    |f(\mu)-f(\mu')|\leq (L_1K_1+L_2+2\lambda K_1(\tilde{K}'+1) )\cW(\mu,\mu').
\]
That is, $f\in L_{\tilde{K}''}$ for some $\tilde{K}''>0$. Recall that 
\begin{align*}
     \Ti^{\lambda}_2(\phi)(\mu)&=\sum_{k=0}^{\infty}\delta_\lambda(k+1)\mE^{\mu,\phi}\left[r(\mu_k)\phi(\mu_k)-\lambda\E(\phi(\mu_k)   \right]\\
     &=\sum_{k=0}^{\infty}\delta_\lambda(k+1) \mE^{\mP^{\mu,\phi}_k}f(\bar{\mu}),
\end{align*}
where $\bar{\mu}$ denotes the canonical random variable on $\barS$. By Kantorovich’s duality and Lemma \ref{W1estLemma}, it holds that
\begin{align*}
 |\Ti^{\lambda}_2(\phi)(\mu)-\Ti^{\lambda}_2(\phi)(\mu')|&\leq \sum_{k=0}^{\infty}\delta_\lambda(k+1)\left|\mE^{\mP^{\mu,\phi}_k}f(\bar{\mu})-\mE^{\mP^{\mu',\phi}_k}f(\bar{\mu})   \right|\\
 &\leq \tilde{K}''\sum_{k=0}^{\infty}\delta_\lambda(k+1)\cW_1(\mP^{\mu,\phi}_k,\mP^{\mu',\phi}_k)\\
 &\leq \tilde{K}''\left(\sum_{k=0}^{\infty}\delta_\lambda(k+1)K_2^k\right)\cW(\mu,\mu').
\end{align*}
The conclusion is verified with 
\begin{align*}
K_4&=\tilde{K}''\left(\sum_{k=0}^{\infty}\delta_\lambda(k+1)K_2^k\right)\\
&=\tilde{K}''\left(\sum_{k=0}^{\infty}\delta(k+1)\left(\frac{1}{1+\lambda}\right)^{(k+1)^2}K_2^k\right)<\infty,
\end{align*}
where the series converges because of the fact $\frac{1}{1+\lambda}\in (0,1)$. 
\end{proof}
Lemma \ref{TLipLemma} is a key step towards the existence of fixed points, because it provides the compactness which is needed in Schauder's fixed point theorem.
\begin{corollary}\label{corocompact}
   $\Ti^{\lambda}(B_R)$ is relatively compact in $C(\bar{S})$.
\end{corollary}
\begin{proof}
    Because $\bar{S}$ is a compact metric space (equipped with $\cW_1$), this is a direct consequence of Lemma \ref{TLipLemma} and the Azela-Ascoli theorem.
\end{proof}
Finally, it remains to verify the continuity of $\Ti^{\lambda}$. The proof is split into two lemmas.
\begin{lemma}\label{T1continuous}
   For a sequence $\{v_n\}_{n\in \mN}\subset B_R$ and $v\in B_R$, if $\|v-v_n\|_{\infty}\to 0$ as $n\to \infty$, we have $\|\phi_n-\phi\|_{\infty}\to 0$ as $n\to \infty$. Here $\phi_n:=\Ti^{\lambda}_1(v_n)$ and $\phi:=\Ti^{\lambda}_1(v)$.
\end{lemma}
\begin{proof}
    We recall the notation $\Hi^v_\lambda(\mu):= \frac{1}{\lambda}[r(\mu)-\mE^0v(T(\mu,a,Z^0))]$. It is clear that for $u=v,v_n$, $n\in \mN$,
    \begin{equation}\label{Hbound}
    \sup_{\mu}|\Hi^{u}_\lambda(\mu)|\leq \frac{1}{\lambda}(L_1+R),
    \end{equation}
    and 
    \begin{equation}
     \sup_{\mu}|\Hi^{v_n}_\lambda(\mu)-\Hi_\lambda^v(\mu)|\leq \frac{1}{\lambda}\|v-v_n\|_{\infty}.   
    \end{equation}
    Using similar computations as in the proof of Lemma \ref{phiestuniformlemma}, we conclude that 
    \[
        |\phi_n(\mu)-\phi(\mu)|\leq \frac{1}{\lambda}\exp\left(\frac{1}{\lambda}(L_1+R)\right)\|v-v_n\|_{\infty},
    \]
which completes the proof with the fixed $\lambda>0$.
\end{proof}
\begin{lemma}\label{Tcontinuous}
    For $\phi_n:=\Ti^\lambda_1(v_n)$, $\phi:=\Ti^\lambda_1(v)$ with $v_n,v\in B_R$ and $\|v_n-v\|_{\infty}\to 0$ as $n\to \infty$, we have $\|\Ti^{\lambda}_2(\phi_n)-\Ti^{\lambda}_2(\phi)\|_{\infty}\to 0$ as $n\to \infty$.
\end{lemma}
\begin{proof}
    For $\psi=\phi_n$ and $\phi$, it follows from definition that 
    \[
    \Ti^{\lambda}_2(\psi)(\mu)=\tilde{J}^\phi_\lambda(\mu)=\sum_{k=0}^\infty\delta_\lambda(k+1)\mE^{\mu,\psi}\cJ(\mu_k,\psi),
    \]
    where
    \[
    \cJ(\mu,\psi):= r(\mu)\psi(\mu)-\lambda\E(\psi(\mu)).
    \]
    From the proof of Lemma \ref{TLipLemma} (see the estimation of $f$ therein), we know that $\mu\mapsto\cJ(\mu,\phi_n)$ is Lipschitz with a Lipschitz constant $\tilde{K}''$ that is independent of $n$. Moreover, for any $\mu\in \barS$, it holds that
    \begin{align*}
    |\cJ(\mu,\phi_n)-\cJ(\mu,\phi)|\leq& |r(\mu)||\phi(\mu)-\phi_n(\mu)|+\lambda|\E(\phi(\mu))-\E(\phi_n(\mu))|\\
    \leq& L_1\|\phi_n-\phi\|_{\infty}\\
    &+\lambda|\phi_n(\mu)-\phi(\mu)|\int_0^1(1+|\log (\alpha \phi_n(\mu)+(1-\alpha)\phi(\mu))|)\md \alpha\\
    &+\lambda|\phi_n(\mu)-\phi(\mu)|\int_0^1(1+|\log (1-\alpha \phi_n(\mu)-(1-\alpha)\phi(\mu))|)\md \alpha\\
    \leq& K_5\|\phi_n-\phi\|_{\infty},
    \end{align*}
    where $K_5=K_5(\lambda,R)$ is a constant that only depends on $\lambda$ and $R$. Here, we use Lemma \ref{phiestuniformlemma} to conclude that $|\log (\alpha \phi_n(\mu,a)+(1-\alpha)\phi(\mu,a))|$ and $|\log (1-\alpha \phi_n(\mu)-(1-\alpha)\phi(\mu))|$ are uniformly bounded in $n$. Therefore, we obtain that
    \begin{align*}
        |\Ti^{\lambda}_2(\phi)(\mu)-\Ti^{\lambda}_2(\phi_n)(\mu)|=&\sum_{k=0}^\infty\delta_\lambda(k+1)|\mE^{\mu,\phi_n}\cJ(\mu_k,\phi_n)-\mE^{\mu,\phi}\cJ(\mu_k,\phi)|\\
        \leq& \sum_{k=0}^\infty\delta_\lambda(k+1)[|\mE^{\mu,\phi_n}\cJ(\mu_k,\phi)-\mE^{\mu,\phi}\cJ(\mu_k,\phi)|\\
        &+\mE^{\mu,\phi_n}|\cJ(\mu_k,\phi_n)-\cJ(\mu_k,\phi)|]\\
        \leq& \sum_{k=0}^\infty\delta_\lambda(k+1)\tilde{K}''\overline{\cW_1}(\mP^{\mu,\phi_n}_k,\mP^{\mu,\phi}_k)\\
        &+K_5\|\phi-\phi_n\|_{\infty}\\
        \leq& \left(K_5+  \tilde{K}''\sum_{k=0}^\infty\delta_\lambda(k+1)K_3^k \right)\|\phi_n-\phi\|_{\infty},
    \end{align*}
where we have used Lemma \ref{W1estphiLemma}. Note that the constant on the right hand side is independent of $\mu$ and $n$, in view of Lemma \ref{T1continuous}, we can conclude the proof.
\end{proof}

\begin{proof}[Proof of Theorem \ref{fixedpoint}]
    By Lemma \ref{Tcontinuous}, $\Ti^{\lambda}$ is continuous from $C(\barS)$ to itself. Clearly, it maps $B_R$ into itself, and by Corollary \ref{corocompact} it maps $B_R$ to a relatively compact subset of $B_R$. Therefore, it follows from Schauder's theorem that $\Ti^{\lambda}$ has a fixed point, denoted by $\phi_\lambda$. From \eqref{phiestLip} we deduce that $\phi_\lambda\in \F^{\rm Lip}_S$.
\end{proof}

\subsection{Existence of Relaxed Equilibria}\label{proof:existoriginal}
 In this subsection, we aim to discuss the limit as $\lambda\to 0$, which gives the existence of the relaxed equilibria of the original problem without the entropy regularization. We stress that the analysis is highly technical because all estimations in the previous subsections, such as the constants $K_i$, $i=1,\cdots,5$, may blow up as $\lambda\to 0$. This challenge is brought by the state space $\barS$ in our mean-field setting, which is inherently uncountably infinite. We use tools in functional analysis to obtain the desired convergence. 

Throughout this subsection, let $v_\lambda$ be a fixed point of $\Ti^\lambda$, whose existence is guaranteed by Theorem \ref{fixedpoint}. We also denote $\phi_\lambda:=\Ti^\lambda_1(v_\lambda)$ for convenience. By definitions of these operators (see Section \ref{exisresults}), we have
\begin{align}
    &\phi_\lambda(\mu)=\frac{1}{1+\exp\left(\frac{1}{\lambda}[\mE^0 v_\lambda(T_0(\mu,Z^0)) -r(\mu)]     \right)},\label{philambdadef}\\
    &v_\lambda(\mu)=\sum_{k=0}^\infty \delta_\lambda(k+1)\mE^{\mu,\phi_\lambda}\left[r(\mu_k)\phi_\lambda(\mu_k)-\lambda\E(\phi_\lambda(\mu_k))    \right].\label{vlambdadef}
\end{align}

The next result provides the candidate limit points of $\phi_\lambda$ and $v_\lambda$ as $\lambda\to 0$. It is pivotal for our later convergence analysis. We recall that $T^0$ is a measure on $\cP(S)$ under Assumption \ref{dominate}, and it can be naturally extended to $\barS$ by setting $T^0(\{\triangle\})=0$.
\begin{lemma}\label{weakstarconvergence}
    There exist $\phi_0,v_0\in L^\infty(\barS,T^0)$, such that for some sequence $\lambda_n\to 0$, we have $\phi_{\lambda_n}\weakstar \phi_0$, $v_{\lambda_n}\weakstar v_0$, both in the weak-$*$ topology $\sigma(L^\infty(\barS,T^0),L^1(\barS,T^0))$ of $L^{\infty}(\barS,T^0)$.
\end{lemma}
\begin{proof}
    Suppose $\lambda\leq 1$. By Lemma \ref{Tibound}$, \|\phi_\lambda\|_{\infty}\leq 1$ and $\|v_\lambda\|_{\infty}\leq R$. Therefore, the desired conclusion is a direct consequence of Alaoglu's theorem and the duality relation $(L^1)^*=L^\infty$. 
\end{proof}
\begin{remark}\label{phi0value}
    Because $\phi_0$ serves as the candidate equilibrium for the original problem, a natural question is whether we have $\phi_0\in [0,1]$. We claim that this is true in the sense of $T^0$-almost sure. Indeed, by weak-$*$ convergence, as $n\to \infty$
    \[
\int_{\barS}\phi_{\lambda_n}\ind_{\{\phi_0<0  \}}T^0(\md \nu)\to \int_{\barS}\phi_0\ind_{\{\phi_0<0  \}}T^0(\md \nu)\geq 0.
    \]
    The last inequality is due to the fact that $\phi_{\lambda_n}\geq 0$ for every $n\in \mN$. Because $\phi_0\ind_{\{\phi_0<0  \}} \leq 0$, we conclude that $\phi_0\ind_{\{\phi_0<0  \}} = 0$, $T^0$-almost surely. In other words, $\phi_0\geq 0$, $T^0$-almost surely. The fact that $\phi_0\leq 1$, $T^0$-almost surely, can be proved similarly.
\end{remark}

Under Assumption \ref{dominate}, the above weak-$*$ convergence can be improved to uniform convergence if we consider the one-step transition.
\begin{lemma}
    For some subsequence of $\{\lambda_n\}_{n\in \mN}$, which we still denote by $\{\lambda_n\}$ for simplicity, it holds that  
    \begin{align}
       & \sup_{\mu\in \barS}|\mE^0\phi_{\lambda_n}(T_0(\mu,Z^0))-\mE^0\phi_0(T_0(\mu,Z^0))|\to 0,\label{uconvergencephi}\\
       &\sup_{\mu\in \barS}|\mE^0v_{\lambda_n}(T_0(\mu,Z^0))-\mE^0v_0(T_0(\mu,Z^0))|\to 0.\label{uconvergencev}
    \end{align}
\end{lemma}
\begin{proof}
    From proofs in Subsection \ref{existregular}, we know that $\phi_\lambda\in C(\barS)$ and $\|\phi_{\lambda}\|_{\infty}\leq 1$. By Assumption \ref{dominate}, we then have
    \[
   |\mE^0\phi_{\lambda}(T_0(\mu,Z^0))-\mE^0\phi_\lambda(T_0(\mu',Z^0))|\leq L_4\cW(\mu,\mu').
    \]
    Consequently, the family $\{\mE^0\phi_{\lambda_n}(\cdot,Z^0)\}_{n\in \mN}\subset C(\barS)$ is uniformly equi-continuous. Because the uniform boundedness is obvious, an application of the Azela-Ascoli theorem yields that for some subsequence $\lambda_n\to 0$ and some $\tilde{\phi}_0\in C(\barS)$, we have
    \begin{equation}\label{uconvergencephitilde}
    \sup_{\mu\in \barS}|\mE^0\phi_{\lambda_n}(T_0(\mu,Z^0))-\tilde{\phi}_0(\mu)|\to 0.
    \end{equation}
    On the other hand, for any fixed $\mu$, by Lemma \ref{weakstarconvergence}, we have that 
    \begin{equation}\label{uconvergencephi0}
    \begin{aligned}
         \mE^0\phi_{\lambda_n} (T_0(\mu,Z^0))=&\int_{\barS} \phi_{\lambda_n}(\nu)f(\mu,\nu)T^0(\md \nu)\\
         &\to  \int_{\barS} \phi_0(\nu)f(\mu,\nu)T^0(\md \nu)\\
         =&\mE^0\phi_0(T_0(\mu,Z^0)),
    \end{aligned}
    \end{equation}
    which results from $f(\mu,\cdot)\in L^1(\barS,T^0)$. Combining \eqref{uconvergencephitilde} and \eqref{uconvergencephi0}, we obtain $\tilde{\phi}_0(\mu)=\mE^0\phi_0(T_0(\mu,Z^0))$ for any $\mu\in \barS$, and hence \eqref{uconvergencephitilde} further implies \eqref{uconvergencephi}. The proof of \eqref{uconvergencev} is similar.
\end{proof}
Finally, we verify that $\phi_0$ is a relaxed equilibrium of the original problem without entropy regularization, and $v_0$ corresponds to its value function. We first need the preparation of the next technical lemmas.
\begin{lemma}\label{equicondphi}
Denote $\tilde{v}_0(\mu):= \mE^0v_0(T_0(\mu,Z^0))$. Then, $\phi_0=0$ on $\barS_-:= \{\mu\in \barS:\tilde{v}_0(\mu)>r(\mu) \}$ and $\phi_0=1$ on $\barS_+:= \{\mu\in \barS;\tilde{v}_0(\mu)<r(\mu)  \}$, $T^0$-almost surely.
\end{lemma}
\begin{proof}
    In this proof, we denote $B_\epsilon^{\barS}(\mu):= \{\mu'\in \barS: \cW_1(\mu',\mu)\leq \epsilon  \}$ for $\epsilon>0$ and $\mu\in \barS$. For each $\mu\in \barS_-$, we choose $\epsilon,\delta>0$ (both depend on $\mu$) such that $\tilde{v}_0(\nu)\geq r(\nu)+\delta$, $\forall \nu\in B_\epsilon^{\barS}(\mu)$. Thanks to \eqref{uconvergencev}, for any sufficiently large $n\in \mN$, we have $\mE^0v_{\lambda_n}(T_0(\nu,Z^0))\geq r(\nu)+\delta/2$, $\forall \nu\in B_\epsilon^{\barS}(\mu)$. Choose $\psi(\nu)=\ind_{B^{\barS}_\epsilon(\mu)}$. First we have that, as $n\to \infty$,
    \begin{align*}
    \left| \int_{\barS}\phi_{\lambda_n}(\nu)\psi(\nu)T^0(\md \nu) \right|&\leq \int_{B^{\barS}_\epsilon(\mu)}\frac{1}{1+\exp\left(\frac{1}{\lambda_n}\left[\mE^0v_{\lambda_n}(T_0(\nu,Z^0))  -r(\nu) \right]    \right)}   T^0(\md \nu)\\
    &\leq \frac{1}{1+\exp\left( \frac{\delta}{2\lambda_n}  \right)}\\
    &\to 0.
    \end{align*}
    On the other hand, in view of $\psi\in L^1(\barS,T^0)$, we obtain from Lemma \ref{weakstarconvergence} that,
    \[
        \int_{\barS}\phi_{\lambda_n}(\nu)\psi(\nu)T^0(\md \nu)\to   \int_{\barS}\phi_0(\nu)\psi(\nu)T^0(\md \nu)=\int_{B^{\barS}_\epsilon(\mu)}\phi_0(\nu)T^0(\md \nu).
    \]
    As a result, it holds that
    \[  
    \int_{B^{\barS}_\epsilon(\mu)}\phi_0(\nu)T^0(\md \nu)=0.
    \]
    By virtue of $\phi_0\geq 0$, $T^0$-almost surely (see Remark \ref{phi0value}), we have $\phi_0=0$, for $T^0$-almost sure $\nu\in B^{\barS}_\epsilon(\mu)$. Note that, because $\barS_-$ is an open subset of $\barS$, which is Polish, we have that $\barS_-$ is Lindel\"of, i.e., each open covering has countable sub-covering (c.f. Theorem 16.11 of \cite{willard2004general}). Therefore, we can choose a countable collection of balls $\{B^{\barS}_{\epsilon_j}(\mu_j)\}_{j=1}^\infty$, where $\mu_j\in \barS_-$ and $\epsilon_j>0$ is such that $B^{\barS}_{\epsilon_j}(\mu_j)\subset \barS_-$, and $\bigcup_{j=1}^\infty B^{\barS}_{\epsilon_j}(\mu_j)=\barS_-$. We have proved that for each $j$, there exists a $T^0$-null set $N_j\subset B^{\barS}_{\epsilon_j}(\mu_j)$ such that $\phi_0(\nu)=0$ for each $\nu\in B^{\barS}_{\epsilon_j}(\mu_j)\backslash N_j$. Take $N=\bigcup_{j=1}^\infty N_j$, which is a $T^0$-null set, we can conclude that $\phi_0(\nu)=0$ for each $\nu\in \barS_-\backslash N$ by the covering property. Thus it is proved that $\phi_0=0$, $T^0$-almost surely on $\barS_-$. The other part of the lemma can be proved in the same fashion.
\end{proof}

In the next two lemmas, we go back to the probability measure $\tilde{\mP}^\mu$ (see Lemma \ref{formulationclassical}) to avoid difficulties caused by the dependence of $\phi$ in $\mP^{\mu,\phi}$.
\begin{lemma}\label{Jconvergence}
    Fix any $\mu\in \barS$ and any finite subset $J\subset \{1,2,\cdots\}$. Denote $j_m=\max J$. We have that, for any measurable function $g:\barS\to \mR$ with $\|g\|_{\infty}<\infty$\footnote{Recall that $\|\cdot\|_\infty$ is the uniform norm, independent of the measure we choose on $\barS$.}, 
    \[
    \tilde{\mE}^\mu\left[g(\mu_{j_m})\prod_{j\in J}\phi_{\lambda_n}(\mu_j)\right]\to \tilde{\mE}^\mu \left[g(\mu_{j_m})\prod_{j\in J}\phi_0(\mu_j)\right],
    \]
    as $n\to \infty$.
\end{lemma}
\begin{proof}
    We argue by induction. First, suppose $J=\{j\}$ with $j\geq 1$. Indeed, by Assumption \ref{dominate}, we have that
    \begin{align*}
    \tilde{\mE}^\mu g(\mu_j)\phi_{\lambda_n}(\mu_j)&=\tilde{\mE}^\mu\mE^0 g(T_0(\mu_{j-1},Z^0))\phi_{\lambda_n}(T_0(\mu_{j-1},Z^0))\\
    &=\tilde{\mE}^\mu \left[\int_{\barS}\phi_{\lambda_n}(\nu)g(\nu)f^0(\mu_{j-1},\nu)T^0(\md \nu)\right]\\
    &\to \tilde{\mE}^\mu \left[\int_{\barS}\phi_0(\nu)g(\nu)f^0(\mu_{j-1},\nu)T^0(\md \nu)\right]\\
    &=\tilde{\mE}^\mu g(\mu_j)\phi_0(\mu_j).
    \end{align*}
    To obtain the limit, we use Lemma \ref{weakstarconvergence} to get 
    \[
    \int_{\barS}\phi_{\lambda_n}(\nu)g(\nu)f^0(\mu_{j-1},\nu)T^0(\nu)\to \int_{\barS}\phi_0(\nu)g(\nu)f^0(\mu_{j-1},\nu)T^0(\nu),
    \]
    for any $\omega\in \Omega$, and then apply the dominated convergence theorem based on the fact
    \[
     \left|\int_{\barS}\phi_{\lambda_n}(\nu)g(\nu)f^0(\mu_{j-1},\nu)T^0(\nu)\right|\leq \|g\|_{\infty}.
    \]
    Suppose that $|J|=l$ and the conclusion holds for any $J'$ with $|J'|=l-1$. We first write
    \begin{align*}
   & \left| \tilde{\mE}^\mu\left[g(\mu_{j_m})\prod_{j\in J}\phi_{\lambda_n}(\mu_j)\right]-\tilde{\mE}^\mu\left[g(\mu_{j_m})\prod_{j\in J}\phi_0(\mu_j)\right]\right|\\
    =&\left|\tilde{\mE}^\mu\left[g(\mu_{j_m})\phi_{\lambda_n}(\mu_{j_m})\prod_{j\in J\backslash \{j_m \}}\phi_{\lambda_n}(\mu_j)\right] -\tilde{\mE}^\mu\left[g(\mu_{j_m})\phi_0(\mu_{j_m})\prod_{j\in J\backslash \{j_m \}}\phi_0(\mu_j)\right] \right|\\
    \leq &\left| \tilde{\mE}^\mu\left[g(\mu_{j_m})\phi_{\lambda_n}(\mu_{j_m})\prod_{j\in J\backslash \{j_m \}}\phi_{\lambda_n}(\mu_j)\right]-\tilde{\mE}^\mu\left[g(\mu_{j_m})\phi_0(\mu_{j_m})\prod_{j\in J\backslash \{j_m \}}\phi_{\lambda_n}(\mu_j)\right]   \right|\\
    &+\left| \tilde{\mE}^\mu\left[g(\mu_{j_m})\phi_0(\mu_{j_m})\prod_{j\in J\backslash \{j_m \}}\phi_{\lambda_n}(\mu_j)\right] -\tilde{\mE}^\mu\left[g(\mu_{j_m})\phi_0(\mu_{j_m})\prod_{j\in J\backslash \{j_m \}}\phi_0(\mu_j)\right]  \right|\\
    =: &I^n_1+I^n_2.
    \end{align*}
Noting that $0\leq \phi_{\lambda_n}\leq 1$, we can rewrite $I^n_1$ by
\begin{align*}
I^n_1=& \left| \tilde{\mE}^\mu \left[\tilde{\mE}^\mu \left[ \left.g(\mu_{j_m})\phi_{\lambda_n}(\mu_{j_m}) \right| \mu_{j_m-1}  \right]\prod_{j\in J\backslash \{ j_m \} }\phi_{\lambda_n}(\mu_j) \right] \right.\\
&\ \ \ \ -\left. \tilde{\mE}^\mu \left[\tilde{\mE}^\mu \left[ \left.g(\mu_{j_m})\phi_0(\mu_{j_m}) \right| \mu_{j_m-1}  \right]\prod_{j\in J\backslash \{ j_m \} }\phi_{\lambda_n}(\mu_j) \right]                         \right| \\
 \leq&\tilde{\mE}^\mu \left[\left| \mE^0 g(T_0(\mu_{j_m-1},Z^0))\phi_{\lambda_n}(T_0(\mu_{j_m-1},Z^0)) \right.\right.\\
 &\ \ \ \ -\left.\left. \mE^0g(T_0(\mu_{j_m-1},Z^0))\phi_0(T_0(\mu_{j_m-1},Z^0)) \right|\right]  \\
 =&\tilde{\mE}^\mu \left|\int_{\barS}\phi_{\lambda_n}(\nu)g(\nu)f^0(\mu_{j_m-1},\nu)T^0(\md \nu)-\int_{\barS}\phi_0(\nu)g(\nu)f^0(\mu_{j_m-1},\nu)T^0(\md \nu) \right|.
\end{align*}
By virtue of Lemma \ref{weakstarconvergence} and Dominated Convergence Theorem again, we obtain $I^n_1\to 0$. To handle the term $I^n_2$, we first note that, for $j_m':= \max\{j\in J:j\neq j_m \}$,
\begin{align*}
    \tilde{\mE}^\mu [g(\mu_{j_m})\phi_0(\mu_{j_m})|\mu_{j_m'}]&=\tilde{\mE}^{\mu_{j_m'}}[g(\mu_{j_m-j_m'})\phi_0(\mu_{j_m-j_m'})]\\
    &=\tilde{\mE}^{\mu_{j_m'}}\mE^0 g(T_0(\mu_{j_m-j_m'-1},Z^0))\phi_0(T_0(\mu_{j_m-j_m'-1},Z^0))\\
    &=: \tilde{g}(\mu_{j_m'}).
\end{align*}
Here, for any $\mu'\in \barS$, let us define
\begin{align*}
\tilde{g}(\mu')&:= \tilde{\mE}^{\mu'}\mE^0 g(T_0(\mu_{j_m-j_m'-1},Z^0))\phi_0(T_0(\mu_{j_m-j_m'-1},Z^0))\\
&=\tilde{\mE}^{\mu'} \left[\int_{\barS}\phi_0(\nu)g(\nu)f^0(\mu_{j_m-j_m'-1},\nu)T^0(\md \nu)\right].
\end{align*}
Thanks to the fact that $\phi_0\in [0,1]$, $T^0$-almost surely, we have that
\begin{align*}
    |\tilde{g}(\mu')|&\leq \|g\|_{\infty}\tilde{\mE}^{\mu'}\left[\int_{\barS}f^0(\mu_{j_m-j_m'-1},\nu)T^0(\md \nu)\right]\\
    &=\|g\|_{\infty},
\end{align*}
for any $\mu'\in \barS$. That is, $\|\tilde{g}\|_{\infty}\leq \|g\|_{\infty}<\infty$. Applying induction hypothesis leads to 
\begin{align*}
I^n_2&=\left| \tilde{\mE}^\mu\left[\tilde{\mE}^\mu [g(\mu_{j_m})\phi_0(\mu_{j_m})|\mu_{j_m'}]\prod_{j\in J\backslash \{j_m \}}\phi_{\lambda_n}(\mu_j)\right] -\tilde{\mE}^\mu\left[\tilde{\mE}^\mu[g(\mu_{j_m})\phi_0(\mu_{j_m})|\mu_{j_m'}]\prod_{j\in J\backslash \{j_m \}}\phi_0(\mu_j)\right]  \right|\\
&=\left|  \tilde{\mE}^\mu\left[\tilde{g}(\mu_{j_m'})\prod_{j\in J\backslash \{j_m \}}\phi_{\lambda_n}(\mu_j)\right] -\tilde{\mE}^\mu\left[\tilde{g}(\mu_{j_m'})\prod_{j\in J\backslash \{j_m \}}\phi_0(\mu_j)\right]            \right|\\
&\to 0.
\end{align*}
The desired result is thus proved.
\end{proof} 
\begin{remark}\label{equivalencephi}
    From the proof of Lemma \ref{Jconvergence}, for any $\mu\in \barS$, it is observed that the value of $\tilde{\mE}^\mu \left[g(\mu_{j_m})\prod_{j\in J}\phi_0(\mu_j)\right]$ only depends on the equivalence class of $\phi_0$ in $L^{\infty}(\barS,T^0)$, because only the integrations of $\phi_0$ (multiplied by integrable functions) with respect to $T^0$ are involved.
\end{remark}

\begin{lemma}\label{eqcondv}
    For each $\mu\in \barS$, $\tilde{v}_0(\mu):= \mE^0v_0(T_0(\mu,Z^0))=\mE^0\tilde{J}^{\phi_0}(T_0(\mu,Z^0))$, where $\tilde{J}^{\phi_0}$ is the auxiliary value function defined in \eqref{auxdef}. 
\end{lemma}
\begin{proof}
    By Lemma \ref{formulationclassical}, we have that 
    \begin{align*}
     \mE^0v_{\lambda_n}(T_0(\mu,Z^0))=&\sum_{k=0}^\infty \delta_{\lambda_n}(k+1)\mE^0\mE^{T_0(\mu,Z^0),\phi_{\lambda_n}}[r(\mu_k)\phi_{\lambda_n}(\mu_k)-\lambda_n \E(\phi_{\lambda_n}(\mu_k))] \\
     =&\sum_{k=0}^\infty\delta_{\lambda_n}(k+1)\tilde{\mE}^{\mu}\mE^{\mu_1,\phi_{\lambda_n}}[r(\mu_k)\phi_{\lambda_n}(\mu_k)-\E(\phi_{\lambda_n}(\mu_k))]\\
     =&\sum_{k=0}^\infty\delta_{\lambda_n}(k+1)\tilde{\mE}^\mu\tilde{\mE}^{\mu_1} r(\mu_k)\phi_{\lambda_n}(\mu_k)\prod_{j=0}^{k-1}(1-\phi_{\lambda_n}(\mu_j))\\
     &-\lambda_n\sum_{k=0}^\infty \delta_{\lambda_n}(k+1)\tilde{\mE}^{\mu}\mE^{\mu_1,\phi_{\lambda_n}}\E(\phi_{\lambda_n}(\mu_k))\\
     =&\sum_{k=1}^\infty\delta_{\lambda_n}(k)\tilde{\mE}^\mu r(\mu_k)\phi_{\lambda_n}(\mu_k)\prod_{j=1}^{k-1}(1-\phi_{\lambda_n}(\mu_j))\\
     &-\lambda_n\sum_{k=0}^\infty \delta_{\lambda_n}(k+1)\tilde{\mE}^{\mu}\mE^{\mu_1,\phi_{\lambda_n}}\E(\phi_{\lambda_n}(\mu_k))\\
     =: & R^n_1+R^n_2.
    \end{align*}
    Because $\E(\phi_{\lambda_n}(\mu_k))$ is uniformly bounded, we deduce from Proposition \ref{lambdavanishing} that $R^n_2\to 0$. On the other hand, for each $k\geq 1$, using Lemma \ref{Jconvergence}, we can derive that
    \begin{align*}
      \tilde{\mE}^\mu r(\mu_k)\phi_{\lambda_n}(\mu_k)\prod_{j=1}^{k-1}(1-\phi_{\lambda_n}(\mu_j))  =&\sum_{J\subset \{1,2,\cdots,k-1\}}c_J\tilde{\mE}^\mu r(\mu_k)\phi_{\lambda_n}(\mu_k)\prod_{j\in J}\phi_{\lambda_n}(\mu_j)\\
      &\to \sum_{J\subset \{1,2,\cdots,k-1\}}c_J\tilde{\mE}^\mu r(\mu_k)\phi_0(\mu_k)\prod_{j\in J}\phi_0(\mu_j)\\
      =&\tilde{\mE}^\mu r(\mu_k)\phi_0(\mu_k)\prod_{j=1}^{k-1}(1-\phi_0(\mu_j)).
    \end{align*}
    Note that $\delta_{\lambda_n}(k)\uparrow \delta(k)$ as $n\to \infty$. By Dominated Convergence Theorem and the fact $\sum_{k=0}^\infty \phi_0(\mu_k)\prod_{j=1}^{k-1}(1-\phi_0(\mu_j))\leq 1$, we conclude
    \begin{equation}\label{Jphi0exp}
    R^n_1\to \sum_{k=1}^\infty\delta(k)\tilde{\mE}^\mu r(\mu_k)\phi_0(\mu_k)\prod_{j=1}^{k-1}(1-\phi_0(\mu_j))=\mE^0\tilde{J}^{\phi_0}(T_0(\mu,Z^0)).
    \end{equation}
    Finally, \eqref{uconvergencev} leads to the desired relation $\mE^0 v_0(T_0(\mu,Z^0))=\mE^0\tilde{J}^{\phi_0}(T_0(\mu,Z^0))$.
\end{proof}
At the end of this subsection, we establish the existence of the relaxed equilibria for the original problem (see Definition \ref{originalequilibrium}) without entropy regularization.

\begin{proof}[Proof of Theorem \ref{existoriginal}]
    From Lemma \ref{equicondphi} and Lemma \ref{eqcondv}, it clearly follows that $\phi_0$ satisfies \eqref{mixedequilibriumchar}, $T^0$-a.s.. However, due to \eqref{Jphi0exp} and Remark \ref{equivalencephi}, if we redefine the value of $\phi_0$ on a $T^0$-null set, the value of $\mE^0\tilde{J}^{\phi_0}(T_0(\mu,Z^0))$ does not change. Consequently, we can redefine $\phi_0$ to make \eqref{mixedequilibriumchar} hold for all $\mu\in \barS$. Therefore, $\phi_0$ is a relaxed equilibrium of the original problem.
\end{proof}

\ \\
\ \\
\textbf{Acknowledgement}:
 The authors sincerely thank anonymous referees for their constructive comments that improved the paper significantly. X. Yu and F. Yuan are supported by the Hong Kong RGC General Research Fund (GRF) under grant no. 15304122, and by the Research Centre for Quantitative Finance at the Hong Kong Polytechnic University under grant no. P0042708.
\ \\
\ \\

\end{document}